\documentclass[11pt]{article}

\usepackage{amsmath}

\usepackage{latexsym}
\usepackage{amssymb}
\usepackage{amsfonts}\large
\usepackage{amsthm}
\usepackage{fontenc}
\usepackage{fullpage}
\usepackage{enumerate}
\usepackage{esint}
\usepackage{bm}
\usepackage{hyperref}
\usepackage{titling}

\hypersetup{colorlinks=true}
\usepackage{xcolor}
\usepackage[normalem]{ulem}
\usepackage{marginnote}
\usepackage{enumerate}

\DeclareMathOperator{\dive}{div}

\DeclareMathOperator{\e}{\varepsilon}

\DeclareMathOperator{\dist}{dist}
\DeclareMathOperator{\Lip}{Lip}

\DeclareMathOperator{\loc}{loc}
\DeclareMathOperator{\bR}{\mathbb R}

\DeclareMathOperator{\sgn}{sgn}

\newtheorem{thm}{Theorem}[section]
\newtheorem{lemma}[thm]{Lemma}
\newtheorem{prop}[thm]{Proposition}
\newtheorem{cor}[thm]{Corollary}
\newtheorem{dfn}[thm]{Definition}

\theoremstyle{definition}
\newtheorem*{acknowledgments*}{Acknowledgments}

\theoremstyle{definition}
\newtheorem{remark}[thm]{Remark}

\renewcommand\footnotemark{}

\makeatletter
\newcommand{\@endstuff}{\par\vspace{\baselineskip}\noindent\small
\begin{tabular}{@{}l}\scshape{Department of Mathematics, Universitat Aut{\`o}noma de Barcelona,}\\\scshape{08193 Bellaterra (Barcelona), Spain}\\ \\\textit{E-mail address}: \texttt{gsakellaris@mat.uab.cat} \end{tabular}}
\AtEndDocument{\@endstuff}
\makeatother

\numberwithin{equation}{section}

\begin{document}
\title{On scale invariant bounds for Green's function for second order elliptic equations with lower order coefficients and applications}

\author{Georgios Sakellaris \thanks{\hspace*{-7pt}2010 \textit{Mathematics Subject Classification}. Primary 35A08, 35J08, 35J15. Secondary 35B50, 35J20, 35J86. \newline
\hspace*{10.5pt} \textit{Key words and phrases}. Green’s function; fundamental solution; lower order coefficients; pointwise bounds; Lorentz bounds; maximum principle; Moser type estimate.\newline
\hspace*{14pt}The author has received funding from the European Union's Horizon 2020 research and innovation programme under Marie Sk{\l}odowska-Curie grant agreement No 665919, and is partially supported by MTM-2016-77635-P (MICINN, Spain) and 2017 SGR 395 (Generalitat de Catalunya).}}

\date{\vspace{-5ex}}
\maketitle
\begin{abstract}
We construct Green's functions for elliptic operators of the form $\mathcal{L}u=-\dive(A\nabla u+bu)+c\nabla u+du$ in domains $\Omega\subseteq\bR^n$, under the assumption $d\geq\dive b$, or $d\geq\dive c$. We show that, in the setting of Lorentz spaces, the assumption $b-c\in L^{n,1}(\Omega)$ is both necessary and optimal to obtain pointwise bounds for Green's functions. We also show weak type bounds for Green's functions and their gradients. Our estimates are scale invariant and hold for general domains $\Omega\subseteq\bR^n$. Moreover, there is no smallness assumption on the norms of the lower order coefficients. As applications we obtain scale invariant global and local boundedness estimates for subsolutions to $\mathcal{L}u\leq -\dive f+g$ in the case $d\geq\dive c$.
\end{abstract}

\section{Introduction}
In this article we are interested with Green's function for the operator
\[
-\dive(A\nabla u+bu)+c\nabla u+du=0
\]
in a domain (open and connected set) $\Omega\subseteq\bR^n$, where $n\geq 3$.

We will assume that the matrix $A$ is bounded and uniformly elliptic in $\Omega$: that is,
\[
\left<A(x)\xi,\xi\right>\geq\lambda|\xi|^2,\quad \forall x\in\Omega,\,\,\forall\xi\in\bR^n.
\]
For the lower order coefficients, we will assume that
\[
b,c\in L^{n,q}(\Omega)\,\,\,\text{for some}\,\,\,q\in[1,\infty),\quad b-c\in L^{n,1}(\Omega),\quad d\in L^{\frac{n}{2},\infty}(\Omega),
\]
where the spaces $L^{n,q}(\Omega),L^{\frac{n}{2},\infty}(\Omega)$ are defined in \eqref{eq:LorentzDfn}. Moreover, we will assume that either $d\geq\dive b$, or $d\geq\dive c$ in the sense of distributions. We remark that, throughout this article, there will be no smallness assumption on the norms of the coefficients; in addition, there will be no size assumption on $\Omega$ and no regularity assumption on $\partial\Omega$. In particular, we can have $\Omega=\bR^n$, or $\Omega=\bR^n_+$.

The consideration of the Lorentz spaces described above is natural if we want to show scale invariant estimates, since these spaces remain invariant under the natural scaling of the equation. Moreover, this consideration is necessary, since the assumption $b-c\in L^{n,q}$ for some $q>1$ does not guarantee weak type and pointwise bounds for Green's function (Definition~\ref{GreenDfn}). Indeed, if we set
\begin{equation}\label{eq:CFormula}
c(x)=-\frac{x}{r^2\ln r},\quad x\in B=B_{1/e}(0),
\end{equation}
then Proposition 7.5 in \cite{KimSak} and the comments after it show that Green's function for the equation $-\Delta u+\delta c\nabla u=0$ in $B$ cannot satisfy $L^{\frac{n}{n-2},\infty}$ and pointwise bounds, for any $\delta>0$. In the setting of Lorentz spaces, it can be shown that $c\in L^{n,q}(B)$ for any $q>1$, but $c\notin L^{n,1}(B)$ (Lemma~\ref{bNorms}); hence, we show in Proposition~\ref{NoPointwiseG} that weak type and pointwise bounds cannot be expected for the operator
\[
-\Delta u+c\nabla u=0,
\]
even if $c\in L^{n,q}$ for some $q>1$ is assumed to have small norm.

On the other hand, by considering the Lorentz space $L^{n,1}$, we will show in Theorem~\ref{Green} and Proposition~\ref{GreenUniqueness} that Green's function $G$ for the operator $\mathcal{L}u=-\dive(A\nabla u+bu)+c\nabla u+du=0$ in a domain $\Omega$ for the case $d\geq\dive b$, exists, it is unique, it satisfies scale invariant pointwise and weak type bounds, and its derivative satisfies weak type bounds; that is,
\[
\|G(\cdot,y)\|_{L^{\frac{n}{n-2},\infty}(\Omega)}+\|\nabla G(\cdot,y)\|_{L^{\frac{n}{n-1},\infty}(\Omega)}\leq C,\quad G(x,y)\leq C'|x-y|^{2-n},
\]
where $C$ depends on $n,\lambda$ and $\|b-c\|_{L^{n,1}}$ only, and $C'$ depends on $n,\lambda,\|A\|_{\infty}$ and $\|b-c\|_{L^{n,1}}$ only. Considering the fore mentioned counterexample, we obtain that the space $L^{n,1}(\Omega)$ is both necessary and optimal in the setting of Lorentz spaces for good bounds on Green's function in the case $d\geq\dive b$.

In this article we also study Green's function in the case $d\geq\dive c$. In this case, Green's function was constructed in \cite{KimSak} (for domains $\Omega$ with finite measure, and with $b-c\in L^n(\Omega)$, $d\in L^{\frac{n}{2}}(\Omega)$), and was shown to satisfy weak type bounds. However, using the function $c$ in \eqref{eq:CFormula}, we show in Proposition~\ref{NoPointwiseGt} that for the equation
\[
-\Delta u-\dive(cu)=0,
\]
even assuming that the $L^{n,q}$ norm of $c$ for some $q>1$ is small, the pointwise bounds for Green's function can fail. On the other hand, if $b-c\in L^{n,1}(\Omega)$, we show in Theorem~\ref{Green} and Proposition~\ref{GreenUniqueness} that in the case $d\geq\dive c$, Green's function for the operator $\mathcal{L}u=-\dive(A\nabla u+bu)+c\nabla u+du=0$ exists, it is unique, and satisfies bounds of the form
\[
\|g(\cdot,x)\|_{L^{\frac{n}{n-2},\infty}(\Omega)}+\|\nabla g(\cdot,x)\|_{L^{\frac{n}{n-1},\infty}(\Omega)}\leq C,\quad g(y,x)\leq C'|y-x|^{2-n},
\]
where $C$ depends on $n,\lambda$ and $\|b-c\|_{L^{n,1}}$ only, and $C'$ depends on $n,\lambda,\|A\|_{\infty}$ and $\|b-c\|_{L^{n,1}}$ only. We also show the symmetry relation $G(x,y)=g(y,x)$ for almost every $(x,y)\in\Omega^2$, where $G$ is Green's function for the adjoint operator mentioned above. Hence, the setting of $L^{n,1}$ is optimal for the pointwise bounds in the case $d\geq\dive c$ as well.

As a first application of our results we show in Proposition~\ref{MaxPrinciplet} a scale invariant inhomogeneous maximum principle for subsolutions $u\in W^{1,2}(\Omega)$ to $\mathcal{L}u=-\dive(A\nabla u+bu)+c\nabla u+du\leq-\dive f+g$ in the case $d\geq\dive c$, when $|\Omega|<\infty$ and $f\in L^{n,1}(\Omega),g\in L^{\frac{n}{2}}(\Omega)$: that is,
\[
\sup_{\Omega}u\leq C\left(\sup_{\partial\Omega}u^++\|f\|_{L^{n,1}(\Omega)}+\|g\|_{L^{\frac{n}{2},1}(\Omega)}\right),
\]
where $C$ depends on $n,\lambda$ and $\|b-c\|_{n,1}$ only. Note that, in the case when $c$ and $d$ are identically $0$ and $b$ has arbitrarily small $L^{n,q}(\Omega)$ norm for some $q>1$, Lemma 7.4 in \cite{KimSak} and Proposition~\ref{bNorms} show that this bound does not necessarily hold. So, in the case $d\geq\dive c$, the assumption $b-c\in L^{n,1}$ is both necessary and optimal to obtain such a maximum principle.

A second application of our results is Proposition~\ref{LocalBoundt}, in which we show a Moser-type local boundedness estimate for nonnegative subsolutions and solutions to $\mathcal{L}u\leq-\dive f+g$ in a ball $B_r$ in the case $d\geq\dive c$ and $f\in L^{n,1}(B_r),g\in L^{\frac{n}{2},1}(B_r)$: that is,
\[
\sup_{B_{r/2}}|u|\leq C\left(\fint_{B_r}|u|+\|f\|_{L^{n,1}(B_r)}+\|g\|_{L^{\frac{n}{2},1}(B_r)}\right),
\]
where $C$ depends on $n,\lambda,\|A\|_{\infty}$ and $\|b-c\|_{n,1}$ only. Again in this case, Lemma 7.4 in \cite{KimSak} and Proposition~\ref{bNorms} show the optimality of $L^{n,1}$ to deduce those types of estimates.

We remark that analogous estimates to the previous two are harder to obtain in the case $d\geq\dive c$ than in the case $d\geq\dive b$. This can be seen, for example, by the fact that the usual maximum principle is not necessarily satisfied for solutions in the case $d\geq\dive c$, while it is satisfied if $d\geq\dive b$ (see Theorem 8.1 in \cite{Gilbarg} for example, and Proposition~\ref{MaxPrinciple}).

In order to show our results, the main core of this article relies on estimates for decreasing rearrangements. The main idea is that, by considering the decreasing rearrangement $u^*$ of a function $u$ (defined in \eqref{eq:DecrRearr}), we obtain a radial function such that various quantities involved with $u$ are maximized or minimized by the analogous quantities for $u^*$. This idea is exhibited by the P{\'o}lya-Szeg{\"o} inequality (see for example (1) in \cite{BrothersZiemer}), and the fact that extremizers that achieve equality in the Sobolev inequality are radially symmetric functions \cite{TalentiSobolev}. Furthermore, this technique has been applied in many past works in order to show estimates of solutions to various problems concerning second order elliptic equations, for example in \cite{TalentiElliptic}, \cite{AlvinoTrombetti78}, \cite{AlvinoTrombetti81}, \cite{BettaMercaldoComparisonAndRegularity}, \cite{DelVecchioPosteraroExistenceMeasure}, \cite{DelVecchioPosteraroNoncoercive}, \cite{AlvinoTrombettiLionsMatarasso}, \cite{AlvinoFeroneTrombetti}, and the more recent \cite{Buccheri}.

A few historical remarks are in order. Green's function for second order elliptic operators of the form $\mathcal{L}u=-\dive(A\nabla u)$ with elliptic and bounded $A$ in bounded domains $\Omega$ was studied in \cite{Littman}, and also later in \cite{Gruter}. More recently, Green's function was constructed in \cite{HofmannLewis} (Chapter III, Lemma 4.3) where the Bourgain condition on the harmonic measure of the domain was assumed (estimate (4.2), Chapter III in \cite{HofmannLewis}). Later on, Hofmann and Kim in \cite{HofmannKim} constructed Green’s functions for elliptic systems, and Kim and Kang showed pointwise bounds for Green's functions for systems in \cite{KimKangGreen}. In all of the previous papers, lower order coefficients are not present. Considering lower order coefficients, Green's function was constructed in \cite{RiahiGreen} by considering the Kato class in domains with $C^{1,1}$ boundary. In \cite{ZhugeZhangGreen} elliptic systems are considered, but smallness assumptions on the norms and coercivity is imposed. Systems are also considered in \cite{MayborodaGreen}.

Towards removing smallness assumptions and considering general domains, the critical and subcritical settings are considered in \cite{KimSak}  (see also \cite{thesis} for the case when $b,d$ vanish and $c$ is bounded, or the case when $c,d$ vanish and $b$ is bounded). The article \cite{KimSak} is the first in which Green's function in the critical setting $b-c\in L^n$ for the case $d\geq\dive c$ is constructed, without any coercivity and smallness assumptions; however, the estimates are not scale invariant and domains with finite measure are considered. As we mentioned above, \cite{KimSak} also shows that an assumption of the form $b-c\in L^n$ does not suffice for good bounds for Green's function in the case $d\geq\dive b$.

Green's function is also constructed in the very recent article \cite{MourgoglouRegularity}, for which the author shows scale invariant weak type and pointwise bounds (Theorem 6.1 and Lemma 6.3 in \cite{MourgoglouRegularity}) under the assumption $d\geq\dive b$ or $d\geq\dive c$, where $b-c$ is a member of a Dini-type Kato-Stummel class $\mathcal{K}_{\text{Dini},1/2}$ (Section 2.1 in \cite{MourgoglouRegularity}). We note that the Lorentz space $L^{n,1}$ that we consider in the present article is not contained in $\mathcal{K}_{\text{Dini},1/2}$, since it can be checked that for any $a>1$, $x|x|^{-2}\left(-\ln|x|\right)^a\in L^{n,1}(B)\setminus\mathcal{K}_{\text{Dini},1/2}(B)$, where $B$ is the ball centered at $0$ with radius $1/e$; moreover, the techniques in the present article are completely different compared to the ones in \cite{MourgoglouRegularity}. We remark that in \cite{MourgoglouRegularity}, except for Green's functions, a wide range of properties is also studied, including solvability with right hand sides and scale invariant estimates under the weaker assumption $b-c\in L^n(\Omega)$, as well as continuity of solutions.

To the best of our knowledge, the present article and \cite{MourgoglouRegularity} are the first to show scale invariant bounds (both of weak type, and pointwise) for Green's function for operators $\mathcal{L}$ with lower order coefficients, under no coercivity assumption on $\mathcal{L}$, no smallness assumption on the lower order coefficients, and no regularity on the domain.

The organization of this article is as follows. In Section 2 we introduce Lorentz spaces, we show preliminary lemmas on Lorentz functions, changes of variables and symmetrization techniques, and we define Green's function. In Section 3 we show various estimates on solutions and subsolutions, where we impose less regularity than what we will need for the sake of generality. In Section 4 we show the main scale invariant estimates for approximate Green's functions: the $L^{\frac{n}{n-2},\infty}$ and the pointwise estimate on approximate Green's functions, and the $L^{\frac{n}{n-1},\infty}$ estimate on their gradients. In this section, the lower order coefficients are assumed to be Lipschitz and $\Omega$ is assumed to be bounded for technical reasons. Those assumptions are removed in Section 5: we first drop the boundedness assumption on the lower order coefficients, and we then pass to general domains. The optimality of our conditions, concerning the pointwise bounds, is shown in Section 6, where counterexamples are provided. Finally, in Section 7 we show global and local scale invariant boundedness estimates for solutions and subsolutions with right hand sides in Lorentz spaces.

\begin{acknowledgments*}
We would like to thank Seick Kim for the collaboration in \cite{KimSak}, which served as a starting point to this article. We would also like to thank Mihalis Mourgoglou for sharing with us an early draft of his paper \cite{MourgoglouDraft} and for helpful conversations.	
\end{acknowledgments*}

\section{Preliminaries}
\subsection{Definitions}
For a domain $\Omega\subseteq\bR^n$, $W_0^{1,2}(\Omega)$ will denote the closure of $C_c^{\infty}(\Omega)$ under the $W^{1,2}$ norm, where
\[
\|u\|_{W^{1,2}(\Omega)}=\|u\|_{L^2(\Omega)}+\|\nabla u\|_{L^2(\Omega)}.
\]
The fact that $W_0^{1,2}(\Omega)$ is a Hilbert space is important in showing existence of solutions using the Lax-Milgram theorem (as in Section 6.2 in \cite{Evans}, or Section 4 of \cite{KimSak}, for example). However, the space $W_0^{1,2}(\Omega)$ is not well suited to the problems that we are interested in, if $\Omega$ has infinite measure. For this, we set $Y_0^{1,2}(\Omega)$ to be the closure of $C_c^{\infty}(\Omega)$ under the $Y^{1,2}$ norm, where
\[
\|u\|_{Y^{1,2}(\Omega)}=\|u\|_{L^{2^*}(\Omega)}+\|\nabla u\|_{L^2(\Omega)},
\]
and $2^*=\frac{2n}{n-2}$ is the Sobolev conjugate to $2$. From the Sobolev estimate
\[
\|\phi\|_{L^{2^*}(\Omega)}\leq C_n\|\nabla\phi\|_{L^2(\Omega)},
\]
for all $\phi\in C_c^{\infty}(\Omega)$, we obtain that $Y_0^{1,2}(\Omega)=W_0^{1,2}(\Omega)$ if $|\Omega|<\infty$. We also set $Y^{1,2}(\Omega)$ to be the space of weakly differentiable $u\in L^{2^*}(\Omega)$, such that $\nabla u\in L^2(\Omega)$, with the $Y^{1,2}$ norm. Then, considering the embedding
\[
T:Y^{1,2}(\Omega)\to L^{2^*}(\Omega)\times (L^2(\Omega))^n,\quad Tu=(u,\partial_1u,\dots\partial_nu),
\]
we can identify $Y^{1,2}(\Omega)$ with a closed subspace of $L^{2^*}(\Omega)\times (L^2(\Omega))^n$. Since $L^{2^*}(\Omega)\times (L^2(\Omega))^n$ is reflexive, we obtain that $Y^{1,2}(\Omega)$ is reflexive as well. Therefore $Y_0^{1,2}(\Omega)$ is also reflexive.

We denote by $\Lip(\Omega)$ the space of Lipschitz functions in $\Omega$: that is, we say that $f\in\Lip(\Omega)$ if $|f(x)-f(y)|\leq M|x-y|$ for some $M>0$ and for all $x,y\in\Omega$. Moreover, $L_c^{\infty}(\Omega)$ will denote the space of compactly supported bounded functions in $\Omega$.

If $u$ is a measurable function in $\Omega$, we define the distribution function
\[
\mu_u(t)=\left|\left\{x\in\Omega: |u(x)|>t\right\}\right|,\quad t>0.
\]
If $u\in L^p(\Omega)$ for some $p\geq 1$, then $\mu_u(t)<\infty$ for any $t>0$. Moreover, we define the decreasing rearrangement (as in (1.4.2), page 45 of \cite{Grafakos}) by
\begin{equation}\label{eq:DecrRearr}
u^*(s)=\inf\{t>0:\mu_u(t)\leq s\}.
\end{equation}
An important property of $u^*$ is that it is equimeasurable to $u$: that is,
\[
\left|\left\{x\in\Omega:|u(x)|>t\right\}\right|=\left|\left\{s>0:u^*(s)>t\right\}\right|\,\,\,\text{for all}\,\,\,t>0.
\]
Note that, if $|\Omega|<\infty$, then $u^*(s)=0$ for $s>|\Omega|$. Also, $\mu_u,u^*$ are right continuous in $(0,\infty)$. 

From Proposition 1.4.5 (2) in \cite{Grafakos}, we obtain that
\begin{equation}\label{eq:muu}
\mu_u(u^*(s))\leq s\,\,\,\text{for all}\,\,\, s\geq 0.
\end{equation}
We also recall Hardy's inequality: for $u,v\in L^1(\Omega)$,
\begin{equation}\label{eq:Hardy}
\int_{\Omega}|uv|\leq\int_0^{\infty}u^*v^*.
\end{equation}
Moreover, we will need the following version of a reverse inequality to the above: for $f,g:[0,\infty)\to[0,\infty)$ with $f,g\geq 0$, $f$ increasing and $g$ integrable, we have that
\begin{equation}\label{eq:HardyLower}
\int_0^{\infty}f(x)g^*(x)\,dx\leq\int_0^{\infty}f(x)g(x)\,dx.
\end{equation}
Let $p\in(0,\infty)$ and $q\in(0,\infty]$. If $f$ is a measurable function defined in $\Omega$, using the decreasing rearrangement of $f$, as on page 48 in \cite{Grafakos} we define the Lorentz seminorm
\begin{equation}\label{eq:LorentzDfn}
\|f\|_{L^{p,q}(\Omega)}=\left\{\begin{array}{c l} \displaystyle \left(\int_0^{\infty}\left(t^{\frac{1}{p}}f^*(t)\right)^q\frac{dt}{t}\right)^{\frac{1}{q}}, & q<\infty \\ \displaystyle\sup_{t>0}t^{\frac{1}{p}}f^*(t),& q=\infty.\end{array}\right.
\end{equation}
Then, from Propositions 1.4.5 and 1.4.9 in \cite{Grafakos}, we obtain that
\begin{equation}\label{eq:EquivalentForInf}
\|f\|_{L^{p,q}(\Omega)}=\left\{\begin{array}{c l} \displaystyle p^{\frac{1}{q}}\left(\int_0^{\infty}\left(\mu_f(s)^{\frac{1}{p}}s\right)^q\frac{ds}{s}\right)^{\frac{1}{q}}, & q<\infty \\ \displaystyle \sup_{s>0}s\mu_f(s)^{\frac{1}{p}},& q=\infty.\end{array}\right.
\end{equation}
We say that $f\in L^{p,q}(\Omega)$ if $\|f\|_{L^{p,q}(\Omega)}<\infty$. Note that if $p\in(1,\infty]$ and $q\in[1,\infty)$, then $L^{p,q}(\Omega)$ is a Banach space with a norm equivalent to the seminorm defined above (see for example Theorem 3.21, page 204 in \cite{SteinWeiss}).

From Remark 1.4.7 on page 48 in \cite{Grafakos},
\begin{equation}\label{eq:LorentzNormPowers}
\left\||u|^r\right\|_{p,q}=\|u\|_{pr,qr}^r,\quad 0<p,r<\infty,\,\,0<q\leq\infty.
\end{equation}
We also have, from Proposition 1.4.10 in \cite{Grafakos}, that Lorentz spaces increase if we increase the second index, and also
\begin{equation}\label{eq:LorentzNormsRelations}
\|f\|_{L^{p,r}}\leq C_{p,q,r}\|f\|_{L^{p,q}}\,\,\,\text{for all}\,\,\,0<p\leq\infty,\,\,0<q<r\leq\infty.
\end{equation}

For $b\in L^{n,\infty}(\Omega)$ and $d\in L^{\frac{n}{2},\infty}(\Omega)$, the assumption $d\geq\dive b$ in the sense of distributions will mean that, for every $\phi\in C_c^{\infty}(\Omega)$ with $\phi\geq 0$,
\begin{equation}\label{eq:Div0}
\int_{\Omega}b\nabla\phi+d\phi\geq 0.
\end{equation}
Using an approximation argument, we can see that \eqref{eq:Div0} is equivalent to the statement
\begin{equation}\label{eq:Div}
\int_{\Omega}b\nabla v+dv\geq 0\,\,\,\text{for every}\,\,\,v\in Y_0^{\left(\frac{n}{2},1\right),1}(\Omega),\,\,\,\text{with}\,\,\,v\geq 0,
\end{equation}
where $Y_0^{(p,q),1}(\Omega)$ for $1<p<n$ and $1\leq q\leq\infty$ is the closure of $C_c^{\infty}(\Omega)$ under the seminorm
\[ \|\phi\|_{Y_0^{(p,q),1}(\Omega)}=\|\phi\|_{L^{\frac{np}{n-p},q}(\Omega)}+\|\nabla\phi\|_{L^{p,q}(\Omega)}.
\]
From Theorem 4.2 (i) in \cite{CosteaSobolevLorentz}, the above seminorm is equivalent to $\|\nabla\phi\|_{L^{p,q}(\Omega)}$ in $C_c^{\infty}(\Omega)$.

For a function $u\in Y^{1,2}$, we will say that $u\leq 0$ on $\partial\Omega$ if $u^+=\max\{u,0\}\in Y_0^{1,2}(\Omega)$. Moreover, $\sup_{\partial\Omega}u$ will be defined as the infimum of all $s>0$ such that $u\leq s$ on $\partial\Omega$.

We now turn to solutions and subsolutions. For this, let $\Omega\subseteq\bR^n$ be a domain, and suppose that $A$ is bounded, $b,c\in L^{n,\infty}(\Omega)$, $d\in L^{\frac{n}{2},\infty}(\Omega)$ and $f\in L^2(\Omega)$, $g\in L^{2_*}(\Omega)$, where $2_*=\frac{2n}{n+2}$. We say that $u\in Y^{1,2}(\Omega)$ is a solution to the equation $\mathcal{L}u=-\dive(A\nabla u+bu)+c\nabla u+du=-\dive f+g$ in $\Omega$, if
\[
\int_{\Omega}A\nabla u\nabla\phi+b\nabla\phi\cdot u+c\nabla u\cdot\phi+du\phi=\int_{\Omega}f\nabla\phi+g\phi,\,\,\,\text{for all}\,\,\,\phi\in C_c^{\infty}(\Omega).
\]
We also say that $u\in Y^{1,2}(\Omega)$ is a subsolution to $\mathcal{L}u\leq-\dive f+g$ in $\Omega$, if
\[
\int_{\Omega}A\nabla u\nabla\phi+b\nabla\phi\cdot u+c\nabla u\cdot\phi+du\phi\leq\int_{\Omega}f\nabla\phi+g\phi,\,\,\,\text{for all}\,\,\,\phi\in C_c^{\infty}(\Omega),\,\,\phi\geq 0.
\]
Finally, we say that $u\in Y^{1,2}(\Omega)$ is a supersolution to $\mathcal{L}u\geq-\dive f+g$ in $\Omega$, if $-u$ is a subsolution to $\mathcal{L}(-u)\leq\dive f-g$ in $\Omega$.

Using the definitions above, we can now define Green's function.
\begin{dfn}\label{GreenDfn}
Let $\Omega\subseteq\bR^n$ be open. Let $A$ be uniformly bounded and elliptic, and $b,c\in L^{n,\infty}(\Omega)$, $d\in L^{\frac{n}{2},\infty}(\Omega)$. Set $\mathcal{L}u=-\dive(A\nabla u+bu)+c\nabla u+du$. We say that $G(x,y)=G_y(x)$ is Green's function for $\mathcal{L}$ in $\Omega$, if $G_y\in L^1(\Omega)$ for almost every $y\in\Omega$, and if, for any $f\in L^{\infty}_c(\Omega)$, the function
\[
u(y)=\int_{\Omega}G(x,y)f(x)\,dx
\]
is a $Y_0^{1,2}(\Omega)$ solution to the adjoint equation $\mathcal{L}^tu=-\dive(A^t\nabla u+cu)+b\nabla u+du=f$ in $\Omega$. Similarly, we say that $g(y,x)=g_x(y)$ is Green's function for $\mathcal{L}^t$ in $\Omega$, if $g_x\in L^1(\Omega)$ for almost every $y\in\Omega$, and if, for any $f\in L^{\infty}_c(\Omega)$, the function
\[
v(x)=\int_{\Omega}g(y,x)f(y)\,dy
\]
is a $Y_0^{1,2}(\Omega)$ solution to the equation $\mathcal{L}u=f$ in $\Omega$.
\end{dfn}

Note that Definition 5.1 of Green's function in \cite{KimSak} coincides with Definition~\ref{GreenDfn} in the case that $|\Omega|<\infty$, $A$ is uniformly elliptic and bounded, $b,c\in L^n(\Omega)$, $d\in L^{\frac{n}{2}}(\Omega)$, and also $d\geq\dive b$ or $d\geq\dive c$. This follows from Lemmas 4.2 and 4.4 in \cite{KimSak}.

\subsection{Basic Lemmas}
The following lemma will be used in order to deduce estimates for coefficients with low regularity.

\begin{lemma}\label{ImprovedSobolev}
Let $\Omega\subseteq\bR^n$ be a domain. Then for any $u\in Y_0^{1,2}(\Omega)$,
\begin{equation}\label{eq:ImprovedSobolev}
\|u\|_{L^{2^*,2}(\Omega)}\leq C_n\|\nabla u\|_{L^2(\Omega)},
\end{equation}
where $C_n$ depends only on $n$.
\end{lemma}
\begin{proof}
The estimate holds if $u=\phi\in C_c^{\infty}(\Omega)$ (see Remark 5 in \cite{TartarImbedding}, or Theorem A in \cite{MalyPick} for example). Now, if $u\in Y_0^{1,2}(\Omega)$, then there exists a sequence $(\phi_m)$ in $C_c^{\infty}(\Omega)$ such that $\phi_m\to u$ in $Y_0^{1,2}(\Omega)$. Then $(\phi_m)$ is bounded in $L^{2^*,2}(\Omega)$ from \eqref{eq:ImprovedSobolev}. From Theorem 1.4.17 in \cite{Grafakos}, $L^{2^*,2}(\Omega)$ is the dual to $L^{2_*,2}(\Omega)$, hence, from the Banach-Alaoglou theorem, $(\phi_m)$ has a subsequence $(\phi_{k_m})$ that converges weakly-* to some $v\in L^{2^*,2}(\Omega)$. Since $\phi_m\to u$ in $L^{2^*}(\Omega)$, we then obtain that $u=v\in L^{2^*,2}(\Omega)$, and also
\[
\|u\|_{L^{2^*,2}(\Omega)}\leq\liminf_{m\to\infty}\|\phi_{k_m}\|_{L^{2^*,2}(\Omega)}\leq C_n\liminf_{m\to\infty}\|\nabla\phi_m\|_{L^2(\Omega)}=C_n\|\nabla u\|_{L^2(\Omega)},
\]
which completes the proof.
\end{proof}

The next lemma deals with the validity of the formula $\mu_u(u^*(s))=s$.

\begin{lemma}\label{MuInverse}
Let $\Omega\subseteq\bR^n$ be a domain, $u\in L^p(\Omega)$ for some $p\geq 1$, and set
\[
A_u=\left\{s\in(0,\infty):|[u^*=s]|>0\right\}.
\]
Then $A_u$ is at most countable. Moreover, if $u^*(s)\notin A_u$, then $\mu_u(u^*(s))=s$.
\end{lemma}
\begin{proof}
Since $u^*$ is decreasing, for different $s\in A_u$, the sets $([u^*=s])^{\mathrm{o}}$ are nonempty, pairwise disjoint open intervals; hence there can only be countably many of those sets. Therefore $A_u$ is at most countable.

Let now $s\in(0,\infty)$. From \eqref{eq:muu}, we have that $\mu_u(u^*(s))\leq s$. So, in order to show the second part, we will show that, if $s$ is such that $\mu_u(u^*(s))<s$, then $u^*(s)\in A_u$. To show this, let $t_0\in(\mu_u(u^*(s)),s)$. Since $u^*$ is decreasing, we have that $u^*(t_0)\geq u^*(s)$. If now $u^*(t_0)>u^*(s)$, then $u^*(t)>u^*(s)$ for every $t\in(0,t_0)$, hence
\[
(0,t_0)\subseteq [u^*>u^*(s)]\,\,\Rightarrow\,\,\,t_0\leq \left|[u^*>u^*(s)]\right|=\left|[|u|>u^*(s)]\right|=\mu_u(u^*(s)),
\]
which is a contradiction. Hence $u^*(t_0)=u^*(s)$, and since $u^*$ is decreasing, we obtain that $u^*$ is constant in $(t_0,s)$, therefore $u^*(s)\in A_u$. This completes the proof.
\end{proof}

The following lemma will be useful when we will consider the Lorentz seminorm on disjoint subsets of our domain.

\begin{lemma}\label{NormOnDisjoint0}
Let $\Omega\subseteq\bR^n$ be open, and let $f\in L^{p,1}(\Omega)$ for some $1\leq p<\infty$. If $X,Y\subseteq\Omega$ with $X,Y\neq\emptyset$ and $X\cap Y=\emptyset$, then
\[
\|f\|_{L^{p,1}(\Omega)}^p\geq\|f\|_{L^{p,1}(X)}^p+\|f\|_{L^{p,1}(Y)}^p.
\]
\end{lemma}
\begin{proof}
Let $\mu_f,\mu_f^X,\mu_f^Y$ be the distribution functions of $f,f|_X,f|_Y$, respectively. If $t>0$, then
\[
\{x\in\Omega:|f(x)|>t\}\supseteq\{x\in X:|f(x)|>t\}\cup \{x\in Y:|f(x)|>t\},
\]
and the last two sets are disjoint; hence, $\mu_f\geq \mu_f^X+\mu_f^Y$. Therefore, from the reverse Minkowski inequality (since $1/p<1$), we obtain
\[
\left(p\int_0^{\infty}\left(\mu_f^X(s)+\mu_f^Y(s)\right)^{\frac{1}{p}}\,ds\right)^p\geq\left( p\int_0^{\infty}\mu_f^X(s)^{\frac{1}{p}}\,ds\right)^p+\left(p\int_0^{\infty}\mu_f^Y(s)^{\frac{1}{p}}\,ds\right)^p,
\]
and combining with \eqref{eq:EquivalentForInf} completes the proof.
\end{proof}

We will also need the following lemma.

\begin{lemma}\label{NormOnDisjoint}
Let $\Omega\subseteq\bR^n$ be open, and let $f\in L^{p,q}(\Omega)$, for some $p\in(1,\infty)$ and $q\in[1,\infty)$. If $(A_m)$ is a sequence of subsets of $\Omega$ with $\chi_{A_m}\to 0$ almost everywhere, then
\[
\|f\|_{L^{p,q}(A_m)}\xrightarrow[m\to\infty]{}0.
\]
\end{lemma}
\begin{proof}
Using the terminology of page 14 in \cite{BennettSharpley}, the assumption we have on $(A_m)$ is stated as $A_m\to\emptyset$. Then, the proof follows combining Definitions I-3.1 and IV-4.1, and Theorems IV-4.7 and IV-4.8 in \cite{BennettSharpley}.
\end{proof}

The next lemma shows that Sobolev functions have decreasing rearrangements that are locally absolutely continuous in $(0,\infty)$. This fact will be crucial in some technical steps.

\begin{lemma}\label{AbsoluteContinuity}
Let $\Omega\subseteq\bR^n$ be a bounded domain, and let $u\in W_0^{1,2}(\Omega)$. Then, the decreasing rearrangement $u^*$ is absolutely continuous in every interval of the form $(\e,M)$ for $0<\e<M<\infty$.
\end{lemma}
\begin{proof}
Consider the function $u^*$ defined in (2) of \cite{BrothersZiemer} (note that this $u^*$ is not the same as the decreasing rearrangement in \eqref{eq:DecrRearr}!). We will use this function to define the function $\tilde{u}$ as on page 154 of \cite{BrothersZiemer}; that is, we set
\[
\tilde{u}(s)=\sup\{t>0:\mu_u(t)>\alpha(n)s^n\}=\inf\{t>0:\mu_u(t)\leq\alpha(n)s^n\},
\]
where $\alpha(n)$ is the volume of the unit ball in $\bR^n$. Then, from Corollary 2.6 in \cite{BrothersZiemer}, $\tilde{u}$ is absolutely continuous in every interval of the form $(\e,M)$. Hence,
\[
u^*(s)=\inf\{t>0:\mu_u(t)\leq s\}=\tilde{u}\left(\sqrt[n]{\frac{s}{\alpha(n)}}\right)
\]
is absolutely continuous in every interval $(\e,M)$, which completes the proof.
\end{proof}

Finally, we will need the following version of Gronwall's inequality (see for example Proposition 2.1 in \cite{DelVecchioPosteraroExistenceMeasure}).

\begin{lemma}\label{Gronwall}
Let $a\geq 0$ and suppose that $f,g_1,g_2,g_3$ are measurable functions defined in $(a,\infty)$, with $g_2,g_3\geq 0$, and $g_3g_1,g_3g_2,g_3f\in L^1(a,\infty)$. If, for almost every $t>a$,
\[
f(t)\leq g_1(t)+g_2(t)\int_t^{\infty}g_3(\tau)f(\tau)\,d\tau,
\]
then, if $\exp$ is the exponential function, for almost every $t>a$, 
\[
f(t)\leq g_1(t)+g_2(t)\int_t^{\infty}g_1(\tau)g_3(\tau)\exp\left(\int_t^{\tau}g_2(\rho)g_3(\rho)\,d\rho\right)\,d\tau.
\]
\end{lemma}

\subsection{Derivatives of compositions}
In this subsection we prove basic lemmas about derivatives of compositions. We start with the following decomposition.
\begin{lemma}\label{Splitting}
Let $\Omega\subseteq\bR^n$ be a domain, and $u\in W_0^{1,2}(\Omega)$. Then we can split
\[
(0,\infty)=G_u\cup D_u\cup N_u,
\]
where the sets $G_u,D_u$ and $N_u$ are disjoint, such that the following properties hold.
\begin{enumerate}[i)]
\item If $x\in G_u$, then $u^*$ is differentiable at $x$, $\mu_u$ is differentiable at $u^*(x)$, and $(u^*)'(x)\neq 0$. Moreover,
\begin{equation}\label{eq:xGuFormulas}
\mu_u(u^*(x))=x\quad\text{and}\quad\mu_u'(u^*(x))=\frac{1}{u^*(x)},\quad\text{for all}\quad x\in G_u.
\end{equation}
\item If $x\in D_u$, then $u^*$ is differentiable at $x$, with $(u^*)'(x)=0$.
\item $N_u$ is a null set.
\end{enumerate}
\end{lemma}
\begin{proof}
Set $\tilde{N}_u$ to be the set of $y\in(0,\infty)$ such that $\mu_u$ is not differentiable at $y$. Let also $N^*_u$ to be the set of $x\in(0,\infty)$ such that $u^*$ does not have a finite derivative at $x$. Since $\mu_u,u^*$ are decreasing, we obtain that $|\tilde{N}_u|=|N^*_u|=0$. We also set $Z_u$ to be the set of $x$ such that $(u^*)'(x)=0$.
	
We now define
\[
E_u=(u^*)^{-1}((0,\infty)\setminus\tilde{N}_u),\quad F_u=(u^*)^{-1}(\tilde{N}_u).
\]
We further split
\[
E_u(1)=E_u\cap N_u^*,\quad E_u(2)=(E_u\setminus Z_u)\setminus N_u^*,\quad E_u(3)=(E_u\cap Z_u)\setminus N_u^*,
\]
and
\[
F_u(1)=F_u\cap N_u^*,\quad F_u(2)=(F_u\setminus Z_u)\setminus N_u^*,\quad F_u(3)=(F_u\cap Z_u)\setminus N_u^*.
\]
We then have that the sets $E_u(i),F_u(j)$ for $i,j=1,2,3$ are pairwise disjoint, and also
\begin{equation}\label{eq:Split}
(0,\infty)=\bigcup_{i=1}^3E_u(i)\cup\bigcup_{j=1}^3F_u(j).
\end{equation}
Note that $|E_u(1)|=0$ and $|F_u(1)|=0$. In addition, $u^*$ has a finite derivative everywhere in $F_u(2)$, and $u^*(F_u(2))\subseteq u^*(F_u)\subseteq\tilde{N}_u$, so $|u^*(F_u(2))|=0$. Therefore, Theorem 1 in \cite{SerrinVarberg} shows that $(u^*)'(y)=0$ for almost every $y\in F_u(2)$, and since $F_u(2)\cap Z_u=\emptyset$, we have that $|F_u(2)|=0$. So, if we define $N_u=E_u(1)\cup F_u(1)\cup F_u(2)$, then $|N_u|=0$.
	
Set now $D_u=E_u(3)\cup F_u(3)$. From the definition of $Z_u$, we then have that $u^*$ is differentiable at every $y\in D_u$, with $(u^*)'(y)=0$. Finally, set $G_u=E_u(2)$; then $u^*$ is differentiable at every $x\in G_u$, and $(u^*)'(x)\neq 0$. Moreover, if $x\in G_u$, then $u^*(x)\in(0,\infty)\setminus\tilde{N}_u$, therefore $\mu_u$ is differentiable at $u^*(x)$. Then, we obtain that the sets $G_u,D_u,N_u$ are disjoint, and \eqref{eq:Split} shows that
\[
(0,\infty)=G_u\cup D_u\cup N_u.
\]
It remains to show the formulas for $x\in G_u$. For this, note that if $x\in G_u$, then $(u^*)'(x)\neq 0$ and since $u^*$ is decreasing, we obtain that $u^*(x)\notin A_u$, where $A_u$ is defined in Lemma~\ref{MuInverse}. Therefore, from the same lemma, we obtain that $\mu_u(u^*(x))=x$.

To show the second formula note that $u^*$ is continuous in $(0,\infty)$ from Lemma~\ref{AbsoluteContinuity}. Since $A_u$ is at most countable, for every $x\in G_u$ there exists a sequence $h_n\to 0$ such that $u^*(x+h_n)\notin A_u$. Then, from Lemma~\ref{MuInverse}, $\mu_u(u^*(x+h_n))=x+h_n$, which implies that $u^*(x+h_n)\neq u^*(x)$. Since now $\mu_u$ is differentiable at $u^*(x)$ and $u^*(x+h_n)\to u^*(x)$,
\[
\mu_u'(u^*(x))=\lim_{n\to\infty}\frac{\mu_u(u^*(x+h_n))-\mu_u(u^*(x))}{u^*(x+h_n)-u^*(x)}=\lim_{n\to\infty}\frac{h_n}{u^*(x+h_n)-u^*(x)}=\frac{1}{(u^*)'(x)},
\]
which completes the proof.
\end{proof}

The following lemma follows from Theorem 2 in \cite{SerrinVarberg}, and will be used when we will have to differentiate the composition $F\circ u$ in the case that $u$ is a monotone function and $F$ is an integrable function. 

\begin{lemma}\label{Decomp}
Let $M>0$, suppose that $g:(0,\infty)\to\mathbb(0,\infty)$ is a monotone function, and let $F:(0,\infty)\to\bR$ be a locally absolutely continuous function. Then, the function $v=F\circ g$ is differentiable almost everywhere, and for almost every $x\in(0,M)$,
\[
v'(x)=f(g(x))\cdot g'(x),
\]
where $f$ is a function with $F'=f$ almost everywhere.
\end{lemma}

As a corollary, we obtain the next change of variables inequality.

\begin{cor}\label{ChangeOfVariables}
Let $g:(0,\infty)\to[0,\infty)$ be a nonnegative decreasing function. Let also $f$ be a nonnegative function in $(0,\infty)$ with $f\in L^1(g(b),g(a))$ for some $0<a<b$. Then, the function $f(g(x))\cdot g'(x)$ is measurable in $(a,b)$, and
\[
\int_a^bf(g(x))\cdot g'(x)\,dx\geq-\int_{g(b)}^{g(a)}f.
\]
\end{cor}
\begin{proof}
Set $\displaystyle F(x)=\int_{g(b)}^xf$. Then $F$ is Lipschitz continuous, with $F'=f$ for almost every $x$. Hence, setting $v=F\circ g$, Lemma~\ref{Decomp} shows that $v$ is differentiable almost everywhere, and also $v'(x)=f(g(x))\cdot g'(x)$ for almost every $x\in(a,b)$. This shows that $(f\circ g)\cdot g'$ is measurable in $(a,b)$. Note now that $v$ is decreasing, therefore Corollary 3.29 in \cite{AmbrosioFuscoPallara} implies that
\[
\int_a^bf(g(x))\cdot g'(x)\,dx=\int_a^bv'(x)\,dx\geq v(b)-v(a)=F(g(b))-F(g(a))=-\int_{g(b)}^{g(a)}f,
\]
which completes the proof.
\end{proof}

\subsection{Symmetrization}
An important construction that we will use is the pseudo-rearrangement of a function $f$ with respect to some function $u$. For this, we first need the following construction from \cite{AlvinoTrombetti78} and \cite{AlvinoTrombetti81} (see also, for example, page 65 in \cite{BettaMercaldoComparisonAndRegularity}, and page 856 in \cite{FeroneVolpicelliSomeRelations}): for $u:\Omega\to\mathbb R$ measurable, there exists a set valued map $s\mapsto \Omega_u(s)\subseteq\Omega$, such that
\[
\left\{\begin{array}{c l}|\Omega_u(s)|=s, &0\leq s\leq |\Omega| \\
\Omega_u(s_1)\subseteq \Omega_u(s_2),& s_1\leq s_2\\
\Omega_u(s)=\left\{x:|u(x)|>t\right\}, & \text{if}\,\,\,s=\mu_u(t),\,\,\,\text{for some}\,\,\,t\geq 0.
\end{array}\right.
\]

We now let $u\in L^1(\Omega)$, and define the pseudo-rearrangement of a function $f\in L^1(\Omega)$ with respect to $u$ as
\begin{equation}\label{eq:DDfn}
\Psi_uf(s)=\frac{d}{ds}\int_{\Omega_u(s)}|f|.
\end{equation}
The fact that $\Psi_uf$ is well defined follows from the absolute continuity of $\displaystyle s\mapsto\int_{\Omega_u(s)}|f|$. Moreover, if $f$ is bounded, then $\Psi_uf$ is bounded as well, since for every $s<t$,
\[
\left|\int_{\Omega_u(t)}|f|-\int_{\Omega_u(s)}|f|\right|=\int_{\Omega_u(t)\setminus\Omega_u(s)}|f|\leq\|f\|_{\infty}\left|\Omega_u(t)\setminus\Omega_u(s)\right|=\|f\|_{\infty}|t-s|,
\]
which implies that $\displaystyle s\mapsto\int_{\Omega_u(s)}|f|$ is Lipschitz. Hence $\Psi_uf$ is bounded.

The following lemma is similar to Lemma 3.5 in \cite{Buccheri}.

\begin{lemma}\label{ChangeDerivative}
Let $u,f\in L^1(\Omega)$. Then, for almost every $s$,
\begin{equation}\label{eq:D_uToLevelSets}
-\frac{d}{ds}\int_{|u|>s}|f|=\Psi_uf(\mu_u(s))(-\mu_u'(s)),
\end{equation}
where we interpret $\Psi_uf(\mu_u(s))(-\mu_u'(s))$ as $0$ when $\mu_u'(s)=0$.
\end{lemma}
\begin{proof}
Since  $\displaystyle s\mapsto\int_{\Omega_u(s)}|f|$ is absolutely continuous, and $\displaystyle\int_{\Omega_u(\mu_u(s))}|f|=\int_{|u|>s}|f|$ for every $s>0$, the proof follows from Lemma~\ref{Decomp} after differentiating with respect to $s$.
\end{proof}

Based on the $L^p$ boundedness of the operator $\Psi_u$, we obtain the following estimate in the setting of Lorentz spaces.

\begin{lemma}\label{LorentzEstimate}
Let $\Omega\subseteq\bR^n$ be a bounded domain and $u\in L^1(\Omega)$. Then there exists $C=C_n$ such that $\displaystyle\|\Psi_uf\|_{L^{\frac{n}{2},\frac{1}{2}}(0,|\Omega|)}\leq C\|f\|_{L^{\frac{n}{2},\frac{1}{2}}(\Omega)}$ for all $f\in L^{\frac{n}{2},\frac{1}{2}}(\Omega)$.
\end{lemma}
\begin{proof}
The proof is based on interpolation and Marcinkiewicz's theorem. Note first that, from Lemma 1.2 in \cite{BettaMercaldoComparisonAndRegularity}, $\Psi_u:L^{4/3}(\Omega)\to L^{4/3}(0,|\Omega|)$ and $\Psi_u:L^n(\Omega)\to L^n(0,|\Omega|)$. Moreover, $\Psi_u$ is subadditive: if $f,g\in L^1(\Omega)$, then for any $s\in(0,|\Omega|)$ and $h>0$ small enough,
\[
\frac{1}{h}\int_{\Omega(s+h)\setminus\Omega(s)}|f+g|\leq\frac{1}{h}\int_{\Omega(s+h)\setminus\Omega(s)}|f|+\frac{1}{h}\int_{\Omega(s+h)\setminus\Omega(s)}|g|,
\]
and letting $h\to 0$, we then obtain that $|\Psi_u(f+g)|\leq |\Psi_uf|+|\Psi_ug|$ for almost every $s\in(0,|\Omega|)$. Since $\frac{4}{3}<\frac{n}{2}<n$, the proof is concluded by applying the off diagonal Marcinkiewicz interpolation theorem (Theorem 1.4.19, page 56 in \cite{Grafakos}) to $\Psi_u$.
\end{proof}

Finally, the following estimate ((40) in \cite{TalentiElliptic}) will be crucial.
\begin{lemma}\label{TalentiEstimate}
Let $\Omega\subseteq\bR^n$ be a bounded domain, and $u\in W_0^{1,2}(\Omega)$. Then, for almost every $t>0$,
\[
1\leq C_n\mu_u(t)^{\frac{2}{n}-2}(-\mu_u'(t))\left(-\frac{d}{dt}\int_{|u|>t}|\nabla u|^2\right).
\]
\end{lemma}

\section{Estimates}
\subsection{An estimate on the derivative}
The following lemma is an analog of the Cacciopoli inequality, but here we bound the $L^2$ gradient of a subsolution in terms of the $L^{2^*}$ norm of the subsolution, which suffices in order to deduce our subsequent estimates. The usual Cacciopoli estimate appears in \cite{MourgoglouRegularity}.

\begin{lemma}\label{DerivativeByU}
Let $\Omega\subseteq\bR^n$ be a domain. Suppose that $A$ is uniformly elliptic, $b,c\in L^{n,\infty}(\Omega)$, $d\in L^{\frac{n}{2},\infty}(\Omega)$, and either $d\geq\dive b$ or $d\geq\dive c$. Let also $2^*=\frac{2n}{n-2}$, $2_*=\frac{2n}{n+2}$ and $f\in L^{2_*}(\Omega)$, $g\in L^2(\Omega)$, and suppose that $u\in Y^{1,2}(\Omega)$ is a nonnegative subsolution to
\[
-\dive(A\nabla u+bu)+c\nabla u+du\leq-\dive g+f.
\]
Then, for any $\phi\in C_c^{\infty}(\bR^n)$ with $\phi\geq 0$ such that $u\phi\in Y_0^{1,2}(\Omega)$,
\[
\int_{\Omega}|\phi\nabla u|^2\leq C\|f\phi\|_{2_*}^2+C\|g\phi\|_2^2+C\|u\phi\|_{2^*,2}^2+C\|u\nabla\phi\|_2^2,
\]
where $C$ depends on $n,\lambda,\|A\|_{\infty}$ and $\|b-c\|_{n,\infty}$.

Moreover, if $b-c\in L^n(\Omega)$, we can replace $\|u\phi\|_{2^*,2}$ above by $\|u\phi\|_{2^*}$, and then $C$ depends on $\|b-c\|_n$ also.
\end{lemma}
\begin{proof}
Assume first that $d\geq\dive b$. Since $u\phi\in Y_0^{1,2}(\Omega)$, $u\phi\in L^{2^*,2}(\Omega)$ from Lemma~\ref{ImprovedSobolev}. Hence,  H{\"o}lder's estimate for Lorentz norms (from \cite{Grafakos}, Theorem 1.4.17) shows that
\[
\nabla(u^2\phi^2)=u\phi\nabla(u\phi)+u\phi\nabla(u\phi)\in L^{\frac{n}{n-1},1}(\Omega),\quad u^2\phi^2\in L^{\frac{n}{n-2},1}(\Omega).
\]
Therefore, \eqref{eq:Div} shows that
\begin{equation*}
\int_{\Omega}b\nabla(u^2\phi^2)+du^2\phi^2\geq 0\Rightarrow \int_{\Omega}b\nabla(u\phi^2)\cdot u+du^2\phi^2\geq-\int_{\Omega}b\nabla u\cdot u\phi^2.
\end{equation*}
Hence, using $u\phi^2$ as a test function and combining with the last estimate, we obtain that
\begin{align}\nonumber
\lambda\int_{\Omega}|\phi\nabla u|^2&\leq\int_{\Omega}A\nabla u\nabla u\cdot\phi^2=\int_{\Omega}A\nabla u\nabla(u\phi^2)-2A\nabla u\nabla\phi\cdot u\phi\\
\label{eq:B}
&\leq\int_{\Omega}fu\phi^2+g\nabla(u\phi^2)+\int_{\Omega}(b-c)\nabla u\cdot u\phi^2-2\int_{\Omega}A\nabla u\nabla\phi\cdot u\phi=I_1+I_2+I_3.
\end{align}
To bound $I_1$, we use H{\"o}lder's inequality (since $2^*$, $2_*$ are conjugate exponents), the Cauchy-Schwartz inequality and the Cauchy inequality with $\delta$, to obtain
\begin{align}\nonumber
I_1&\leq\|f\phi\|_{2_*}\|u\phi\|_{2^*}+\|g\phi\|_2\|\phi\nabla u\|_2+2\|g\phi\|_2\|u\nabla\phi\|_2\\
\label{eq:I1}
&\leq \|f\phi\|_{2_*}\|u\phi\|_{2^*}+\frac{1}{\lambda}\|g\phi\|_2^2+\frac{\lambda}{4}\|\phi\nabla u\|_2+2\|g\phi\|_2\|u\nabla\phi\|_2.
\end{align}
For $I_2$, using H{\"o}lder's estimate for Lorentz norms, we estimate
\begin{equation}\label{eq:I2}
I_2\leq C_n\|b-c\|_{n,\infty}\|u\phi\|_{2^*,2}\|\phi\nabla u\|_2\leq\frac{1}{\lambda}\|b-c\|_{n,\infty}^2\|u\phi\|_{2^*,2}^2+\frac{C_n^2\lambda}{4}\|\phi\nabla u\|_2^2.
\end{equation}
Note that, if $b-c\in L^n(\Omega)$, we can replace the $\|u\phi\|_{2^*,2}$ norm by the $\|u\phi\|_{2^*}$, replacing also $\|b-c\|_{n,\infty}$ by $\|b-c\|_n$. Moreover, for $I_3$,
\begin{equation}\label{eq:I3}
I_3\leq 2\|A\|_{\infty}\|\phi\nabla u\|_2\|u\nabla\phi\|_2\leq\frac{4}{\lambda}\|A\|_{\infty}^2\|u\nabla\phi\|_2^2+\frac{\lambda}{4}\|\phi\nabla u\|_2^2.
\end{equation}
Substituting \eqref{eq:I1}, \eqref{eq:I2} and \eqref{eq:I3} in \eqref{eq:B} shows the estimate in the case $d\geq\dive b$. 

In the case $d\geq\dive c$, using a similar argument to the above, we note that \eqref{eq:B} becomes
\[
\lambda\int_{\Omega}|\phi\nabla u|^2\leq\int_{\Omega}fu\phi^2+g\nabla(u\phi^2)\cdot u+\int_{\Omega}(c-b)\nabla(u\phi^2)\cdot u-2\int_{\Omega}A\nabla u\nabla\phi\cdot u\phi,
\]
and we proceed as above to conclude the proof.
\end{proof}

\subsection{Scale invariant estimates}
In this section we will show scale invariant estimates for subsolutions to the equations we are considering. To achieve more generality, we will show those estimates assuming less regularity than what we will need in the construction of Green's function.

We begin with the maximum principle. Under slightly weaker hypotheses, this has appeared in \cite{MourgoglouDraft} (see also \cite{MourgoglouRegularity}).

\begin{prop}\label{MaxPrinciple}
Let $\Omega\subseteq\bR^n$ be a domain. Assume that $A$ is uniformly bounded and elliptic, and $b,c\in L^{n,\infty}(\Omega)$, $d\in L^{\frac{n}{2},\infty}(\Omega)$, with $b-c\in L^{n,q}(\Omega)$ for some $q<\infty$ and $d\geq\dive b$. Let also $u\in Y^{1,2}(\Omega)$ be a subsolution to $-\dive(A\nabla u+bu)+c\nabla u+du\leq 0$ in $\Omega$. Then, $\displaystyle \sup_{\Omega}u\leq \sup_{\partial\Omega}u^+$.
\end{prop}
\begin{proof}
Set $l=\sup_{\partial\Omega}u^+$. Since the inequality is true if $l=\infty$, we assume that $l<\infty$. Moreover, by considering $u-l$ and using that $d\geq\dive b$, we can assume that $l=0$, so $u^+\in Y_0^{1,2}(\Omega)\subseteq L^{2^*,2}(\Omega)$, by Lemma~\ref{ImprovedSobolev}. In addition, from \eqref{eq:LorentzNormsRelations}, we can assume that $q>n$.

We follow the proof of Theorem 8.1 in \cite{Gilbarg}. Assume that $\sup_{\Omega}u^+=l'>0$, and let $k\in(0,l')$. Set $u_k=(u-k)^+$, then $u_k\in Y_0^{1,2}(\Omega)$, and $uu_k=u^+u_k\geq 0$. Moreover, using Lemma~\ref{ImprovedSobolev}, we obtain that $\nabla(uu_k)\in L^{\frac{n}{n-1},1}(\Omega)$ and $uu_k\in L^{\frac{n}{n-2},1}(\Omega)$, therefore the assumption $d\geq\dive b$ implies that
\[
\int_{\Omega}b\nabla u_k\cdot u+duu_k+\int_{\Omega}b\nabla u\cdot u_k=\int_{\Omega}b\nabla(uu_k)+duu_k\geq 0,
\]
hence
\begin{equation*}
\int_{\Omega}A\nabla u\nabla u_k+(c-b)\nabla u\cdot u_k\leq\int_{\Omega}A\nabla u\nabla u_k+b\nabla u_k\cdot u+c\nabla u\cdot u_k+duu_k\leq 0.
\end{equation*}
If now $p/2$ is the conjugate exponent to $\frac{q}{2}>\frac{n}{2}$, then $p=\frac{2q}{q-2}>2$. Then, if $D_k$ is the support of $\nabla u_k$, using H{\"o}lder's inequality for Lorentz norms (from Theorem 1.4.17 in \cite{Grafakos}) we have that
\begin{align}\nonumber
\lambda\|\nabla u_k\|_{L^2(\Omega)}^2&\leq\int_{\Omega}A\nabla u_k\nabla u_k\leq\|b-c\|_{L^{n,q}(D_k)}\|\nabla u_k\|_{L^2(\Omega)}\|u_k\|_{L^{2^*,p}(\Omega)}\\
\label{eq:Forb-c}
&\leq C\|b-c\|_{L^{n,q}(D_k)}\|\nabla u_k\|_{L^2(\Omega)}\|u_k\|_{L^{2^*,2}(\Omega)}\leq C\|b-c\|_{L^{n,q}(D_k)}\|\nabla u_k\|_{L^2(\Omega)}^2,
\end{align}
for some $C$ depending only on $n,q$, where we used that $p>2$ and \eqref{eq:LorentzNormsRelations} for the second to last estimate, and  \eqref{eq:ImprovedSobolev} for the last estimate. If $\|\nabla u_k\|_2=0$, the fact that $\Omega$ is connected implies $u_k$ is a constant. But, $u_k\in Y_0^{1,2}(\Omega)$, so $u_k\equiv 0$, therefore $u\leq k$ in $\Omega$, which is a contradiction with $k<l'$. Hence $\|\nabla u_k\|_2\neq 0$, and then \eqref{eq:Forb-c} shows that, for every $0<k<l'$,
\begin{equation}\label{eq:NormNotGoingTo0}
\|b-c\|_{L^{n,q}(D_k)}\geq C_{n,q,\lambda}.
\end{equation}
Let now $(k_m)$ be an increasing sequence with $k_m\in(0,l')$ and $k_m\to l'$. Then the sequence $(D_{k_m})$ is decreasing, and also $\nabla u_k=0$ almost everywhere on $[u=l']$; hence,
\[
\bigcap_{m=1}^{\infty}D_{k_m}\subseteq \bigcap_{m=1}^{\infty}[u>k_m]\setminus[u=l']=[u> l']\,\,\,\Rightarrow\,\,\, \left|\bigcap_{m=1}^{\infty}D_{k_m}\right|\leq\left|[u> l']\right|=0
\]
where we used that $l'=\sup_{\Omega}u^+$ in the last equality. Therefore $\chi_{D_{k_m}}\to 0$ almost everywhere, and then Lemma~\ref{NormOnDisjoint} shows that $\|b-c\|_{L^{n,q}(D_{k_m})}\to 0$ as $m\to\infty$. However, this contradicts \eqref{eq:NormNotGoingTo0}, and this completes the proof.
\end{proof}

As a corollary, we obtain uniqueness of $Y_0^{1,2}$ solutions.

\begin{prop}\label{Uniqueness}
Let $\Omega\subseteq\bR^n$ be a domain. Assume that $A$ is uniformly bounded and elliptic, with ellipticity $\lambda$, and $b,c\in L^{n,\infty}(\Omega)$, $d\in L^{\frac{n}{2},\infty}(\Omega)$, with $b-c\in L^{n,q}(\Omega)$ for some $q<\infty$ and $d\geq\dive b$. If $u\in Y_0^{1,2}(\Omega)$ is a solution to the equation
\[
-\dive(A\nabla u+bu)+c\nabla u+du=0
\]
in $\Omega$, then $u\equiv 0$.
\end{prop}

The next estimate is a refinement of Lemma 3.13 in \cite{KimSak}, in which we recover the correct way in which the constant depends on $b-c$. This proposition has appeared in \cite{MourgoglouDraft} (see \cite{MourgoglouRegularity}), but here we present a different proof which is based on the maximum principle.

\begin{prop}\label{SupByIntegral}
Let $B_r\subseteq\bR^n$ be a ball of radius $r$. Let $A$ be uniformly bounded and elliptic, with ellipticity $\lambda$, and $b,c\in L^{n,\infty}(\Omega)$, $d\in L^{\frac{n}{2},\infty}(\Omega)$, with $b-c\in L^n(\Omega)$ and $d\geq\dive b$. Assume that $u\in W^{1,2}(B_r)$ is a nonnegative subsolution to the equation $-\dive(A\nabla u+bu)+c\nabla u+du=0$ in $B_r$. Then,
\[
\sup_{B_{r/2}}u\leq C\fint_{B_r}u,
\]
where $C$ depends on $n,\lambda,\|A\|_{\infty}$ and $\|b-c\|_n$.
\end{prop}
\begin{proof}
Since the estimate we want to show is scale invariant, we will assume that $B_r=B_1$. In the following, $B_s$ will denote the ball with radius $s$, centered at $0$, and $\|b-c\|_n=\|b-c\|_{L^n(B_1)}$.

Set $\e_0=\frac{\lambda}{6C_n}$, where $C_n$ is the constant in the Sobolev embedding $W_0^{1,2}(\bR^n)\hookrightarrow L^{2^*}(\bR^n)$. We will show inductively that, for all $m\in\mathbb N$,
\begin{equation}\label{eq:Induction}
\sup_{B_{1/2}}u\leq 8^{(m-1)n}C_0\fint_{B_1}u,\quad\text{if}\quad \|b-c\|_{L^n(B_1)}^n\e_0^{-n}\leq m,
\end{equation}
where $C_0$ only depends on $n,\lambda$ and $\|A\|_{\infty}$.
	
Assume first that $\|b-c\|_{L^n(B_1)}\leq\e_0$. As in \cite{KimSak}, (2.2), we define
\[
R_{b-c}(t)=\left(\int_{|b-c|>t}|b-c|^n\right)^{1/n},\quad r_{b-c}(\e)=\inf\{t>0: R_{b-c}(t)<\e\}.
\]
Then, $R_{b-c}(t)\leq \|b-c\|_{L^n(B_1)}<\frac{\lambda}{3C_n}$ for all $t>0$, therefore $r_{b-c}\left(\frac{\lambda}{3C_n}\right)=0$. Note that the proof of Lemma 3.13 in \cite{KimSak} gives the same result if we assume that $u$ is a nonnegative subsolution; hence there exists $C_0>0$, depending only on $n,\lambda$ and $\|A\|_{\infty}$, such that $\displaystyle\sup_{B_{1/2}}u\leq C_0\fint_{B_1}u$. So, \eqref{eq:Induction} holds for $m=1$. 

Let now $m\geq 1$, and suppose that
\begin{equation}\label{eq:InductiveHypothesis}
\sup_{B_{1/2}}u\leq 8^{(m-1)n}C_0\fint_{B_1}u,\quad\text{if}\quad \|b-c\|_n^n\e_0^{-n}\leq m.
\end{equation}
Suppose that $b,c$ are such that $m<\|b-c\|_n^n\e_0^{-n}\leq m+1$. We distinguish between two cases: $\|b-c\|_{L^n(B_{3/4})}^n\e_0^{-n}\leq m$, and $\|b-c\|_{L^n(B_{3/4})}^n\e_0^{-n}>m$.
	
In the first case, for any $x$ with $|x|<\frac{1}{2}$, $B_{1/4}(x)\subseteq B_{3/4}$, therefore  $\|b-c\|_{L^n(B_{1/4}(x))}^n\e_0^{-n}\leq m$. Then, from the inductive hypothesis \eqref{eq:InductiveHypothesis} and a scaling argument,
\[
\sup_{B_{1/8}(x)}u\leq 8^{(m-1)n}C_0\fint_{B_{1/4}(x)}u\leq 8^{(m-1)n}C_04^n\fint_{B_1}u\leq 8^{mn}C_0\fint_{B_1}u.
\]
Since this estimate holds for any $x$ with $|x|<\frac{1}{2}$, we obtain that
\begin{equation}\label{eq:Case1}
\sup_{B_{1/2}}u\leq\sup_{|x|\leq 1/2}\left(\sup_{B_{1/8}(x)}u\right)\leq 8^{mn}C_0\fint_{B_1}u.
\end{equation}
In the second case, we have that $\|b-c\|_{L^n(B_{3/4})}^n>\e_0^nm$, therefore
\[
\|b-c\|_{L^n(B_1\setminus B_{3/4})}=\left(\int_{B_1}|b-c|^n-\int_{B_{3/4}}|b-c|^n\right)^{1/n}\leq\left(\e_0^n(m+1)-\e_0^nm\right)^{1/n}=\e_0.
\]
Now, for any $y$ with $|y|=\frac{7}{8}$, we have that $B_{1/8}(y)\subseteq B_1\setminus B_{3/4}$, therefore $\|b-c\|_{L^n(B_{1/8}(y))}\leq\e_0$. So, from \eqref{eq:Induction} for $m=1$ and a scaling argument, we obtain that
\[
\sup_{B_{1/16}(y)}u\leq C_0\fint_{B_{1/8}(y)}u\leq 8^nC_0\fint_{B_1}u.
\]
This shows that, in a neighborhood of the sphere $\partial B_{7/8}$, $\displaystyle u\leq 8^nC_0\fint_{B_1}u$ almost everywhere. Then, the maximum principle (Proposition \ref{MaxPrinciple}) shows that
\begin{equation}\label{eq:Case2}
\sup_{B_{1/2}}u\leq\sup_{B_{7/8}}u\leq\sup_{\partial B_{7/8}}u\leq 8^nC_0\fint_{B_1}u\leq 8^{mn}C_0\fint_{B_1}u.
\end{equation}
Hence, in all cases, \eqref{eq:Case1} and \eqref{eq:Case2} show that, if $m<\|b-c\|_n^n\e_0^{-n}\leq m+1$, then
\begin{equation}\label{eq:InductionEnd}
\sup_{B_{1/2}}u\leq 8^{mn}C_0\fint_{B_1}u.
\end{equation}
If now $\|b-c\|_n^n\e_0^{-n}\leq m$, then \eqref{eq:InductiveHypothesis} shows that \eqref{eq:InductionEnd} holds in this case as well; therefore, \eqref{eq:InductionEnd} holds whenever $\|b-c\|_n^n\e_0^{-n}\leq m+1$. Inductively, this shows that \eqref{eq:Induction} holds for any $m\in\mathbb N$.

Now, if $b-c\in L^n$, there exists $m\in\mathbb N$ such that $m-1\leq\|b-c\|_n^n\e_0^{-n}\leq m$. Then,
\begin{equation*}
\sup_{B_{1/2}}u\leq 8^{(m-1)n}C_0\fint_{B_1}u\leq 8^{\|b-c\|_n^n\e_0^{-n}}C_0\fint_{B_1}u,
\end{equation*}
which completes the proof.
\end{proof}

\section{Estimates on approximate Green's functions}
\subsection{Estimates for \texorpdfstring{$G$}{G}}
We now turn to the main estimates for approximate Green's functions. Those will be solutions to our equations with right hand sides being approximations to the Dirac mass at a point in $\Omega$.

Assume that $\Omega\subseteq\bR^n$ is a bounded domain, and $A$ is uniformly elliptic and bounded in $\Omega$. Assume also that $b,c$ are Lipschitz continuous in $\Omega$, and $d\in L^{\infty}(\Omega)$, with $d\geq\dive b$ in $\Omega$. For $y\in\Omega$ fixed and $m>\frac{2}{\delta(y)}$, as right before (5.4) in \cite{KimSak}, there exists $G_y^m\in W_0^{1,2}(\Omega)$ such that
\begin{equation}\label{eq:G_y^m}
-\dive(A\nabla G_y^m+bG_y^m)+c\nabla G_y^m+dG_y^m=h_m:=\frac{1}{|B_{1/m}(y)|}\chi_{B_{1/m}(y)}.
\end{equation}
From the maximum principle (Proposition~\ref{MaxPrinciple}), we then have that $G_y^m\geq 0$ in $\Omega$.

The next lemma will be used at a technical step in the first main estimate for approximate Green's functions.
\begin{lemma}\label{AbsCts}
Let $\Omega\subseteq\bR^n$ be a bounded domain. Assume that $g\in W_0^{1,2}(\Omega)$, $g\geq 0$ and $b\in\Lip(\Omega)$. Then, the function $\displaystyle s\mapsto\int_{[g>s]}b\nabla g$ is Lipschitz in $(0,\infty)$.
\end{lemma}
\begin{proof}
Set $d=\dive b\in L^{\infty}(\Omega)$. For $s>0$, $(g-s)^+\in W_0^{1,2}(\Omega)$. Therefore, integrating by parts,
\[
\int_{[g>s]}b\nabla g=\int_{[g>s]}b\nabla(g-s)^+=\int_{\Omega}b\nabla(g-s)^+=-\int_{\Omega}d(g-s)^+=-\int_{[g>s]}d(g-s).
\]
Hence, if $s,h>0$,
\begin{align*}
\left|\int_{[g>s+h]}b\nabla g-\int_{[g>s]}b\nabla g\right|&=\left|-\int_{[g>s+h]}d(g-s-h)+\int_{[g>s]}d(g-s)\right|\\
&\leq h\int_{[g>s+h]}|d|+\int_{[s\leq g<s+h]}|d||g-s|\leq h\int_{[g\geq s]}|d|\leq \|d\|_{\infty}|\Omega|h,
\end{align*}
which completes the proof.
\end{proof}

We now show a weak type estimate for $G_y^m$.

\begin{lemma}\label{WeakForG}
Let $\Omega\subseteq\bR^n$ be a bounded domain. Let $A$ be uniformly elliptic and bounded in $\Omega$, with ellipticity $\lambda$, and let $b,c\in\Lip(\Omega)$, $d\in L^{\infty}(\Omega)$, with $d\geq\dive b$. For any $y\in\Omega$ and $m\in\mathbb N$ with $m>\frac{2}{\delta(y)}$, the function $G_y^m$ in \eqref{eq:G_y^m} satisfies the estimate 
\[
\|G_y^m\|_{L^{\frac{n}{n-2},\infty}(\Omega)}\leq C,
\]
where $C$ depends on $n,\lambda$ and $\|b-c\|_{n,1}$ only.
\end{lemma}
\begin{proof}
Fix $y\in\Omega$. We follow the proof of Lemma 3.1 in \cite{DelVecchioPosteraroExistenceMeasure} (see also Theorem 3.1 in \cite{DelVecchioPosteraroNoncoercive}). First, for $t,h>0$, consider the function
\begin{equation}\label{eq:Tth}
T_{t,h}(s)=\left\{\begin{array}{c l}0, & |s|<t \\ s-t\sgn(s), & t\leq |s|< t+h \\ h\sgn(s), & |s|\geq t+h.\end{array}\right.
\end{equation}
We use $\phi=T_{t,h}(G_y^m)$ as a test function, to obtain that
\[
\int_{\Omega}A\nabla G_y^m\nabla\phi+b\nabla\phi\cdot G_y^m+c\nabla G_y^m\cdot\phi+dG_y^m\phi=\fint_{B_{1/m}(y)}\phi.
\]
We have that $sT_{t,h}(s)\geq 0$ for all $s\in\mathbb R$, so $G_y^m\phi\geq 0$. Hence, the assumption $d\geq\dive b$ implies that
\begin{equation}\label{eq:ToUseAbsCts}
\int_{\Omega}A\nabla G_y^m\nabla\phi\leq\fint_{B_{1/m}(y)}\phi+\int_{\Omega}(b-c)\nabla G_y^m\cdot\phi.
\end{equation}
Note now that $|\phi|\leq h$ and $\phi$ is supported on $[G_y^m>t]$. Moreover, $\nabla\phi=\nabla G_y^m$ if $t<G_y^m\leq t+h$, and $\nabla\phi=0$ otherwise. Hence, from the ellipticity of $A$,
\[
\lambda\int_{[t<G_y^m\leq t+h]}|\nabla G_y^m|^2\leq h+h\int_{[G_y^m>t]}|b-c||\nabla G_y^m|.
\]
Since $\Omega$ is bounded, we have that $b-c\in L^2(\Omega)$. Moreover, $\nabla G_y^m\in L^2(\Omega)$, therefore the previous estimate shows that
\begin{equation}\label{eq:FIsLipschitz}
t\mapsto H_m(t):=\int_{[G_y^m>t]}|\nabla G_y^m|^2\quad\text{is Lipschitz in}\quad(0,\infty).
\end{equation}
We now return to \eqref{eq:ToUseAbsCts}. Using the definition of $\phi$ and dividing by $h$, we estimate
\[
\frac{\lambda}{h}\int_{[t<G_y^m\leq t+h]}|\nabla G_y^m|^2\leq 1+\int_{[t<G_y^m\leq t+h]}|b-c||\nabla G_y^m|+\int_{[G_y^m>t+h]}(b-c)\nabla G_y^m.
\]
So, by letting $h\to 0$, and since $\nabla G_y^m=0$ almost everywhere in $[G_y^m=t]$, we obtain that for almost every $t>0$,
\begin{equation}\label{eq:ToDropSquare}
-\frac{d}{dt}\int_{[G_y^m>t]}|\nabla G_y^m|^2\leq C_{\lambda}+C_{\lambda}\int_{[G_y^m>t]}(b-c)\nabla G_y^m.
\end{equation}
Let $\mu_m$ denote the distribution function of $G_y^m$, and set $\nu_m(t)=\mu_m(t)^{\frac{1}{n}-1}(-\mu_m'(t))^{1/2}$. Using Lemma~\ref{TalentiEstimate}, we then obtain that
\begin{equation}\label{eq:g_mDfn}
\gamma_m(t):=\left(-\frac{d}{dt}\int_{[G_y^m>t]}|\nabla G_y^m|^2\right)^{1/2}\leq C_n\nu_m(t)\left(-\frac{d}{dt}\int_{[G_y^m>t]}|\nabla G_y^m|^2\right),
\end{equation}
therefore, plugging in \eqref{eq:ToDropSquare} and using Lemma~\ref{AbsCts}, we obtain that
\begin{align}\label{eq:ToPlugDerivative}
\gamma_m(t)&\leq C\nu_m(t)+C\nu_m(t)\int_t^{\infty}\left(-\frac{d}{ds}\int_{[G_y^m>s]}(b-c)\nabla G_y^m\right)\,ds,
\end{align}
where $C$ depends on $n$ and $\lambda$ only.

We now write, for any $s>0$ and $h>0$ small,
\begin{equation}\label{eq:CSForDerivativeLater}
\frac{1}{h}\int_{[s<G_y^m\leq s+h]}(b-c)\nabla G_y^m\leq\left(\frac{1}{h}\int_{[s<G_y^m\leq s+h]}|b-c|^2\right)^{1/2}\left(\frac{1}{h}\int_{[s<G_y^m\leq s+h]}|\nabla G_y^m|^2\right)^{1/2},
\end{equation}
which implies that, for almost every $s>0$,
\[
-\frac{d}{ds}\int_{[G_y^m>s]}(b-c)\nabla G_y^m\leq\left(-\frac{d}{ds}\int_{[G_y^m>s]}|b-c|^2\right)^{1/2}\gamma_m(s):=\beta_m(s)\gamma_m(s),
\]
where $\gamma_m$ is defined in \eqref{eq:g_mDfn}. Plugging the last estimate in \eqref{eq:ToPlugDerivative}, we obtain that
\begin{equation}\label{eq:ForGronwall}
\gamma_m(t)\leq C\nu_m(t)+C\nu_m(t)\int_t^{\infty}\beta_m(s)\gamma_m(s)\,ds.
\end{equation}
If $\Psi_{G_y^m}$ is as in \eqref{eq:DDfn} and since $b$ and $c$ are bounded, the definitions of $\nu_m,\beta_m$ and \eqref{eq:D_uToLevelSets} show that
\begin{align*}
\nu_m(s)\beta_m(s)&=\mu_m(s)^{\frac{1}{n}-1}(-\mu_m'(s))^{1/2}\left(-\frac{d}{ds}\int_{[G_y^m>s]}|b-c|^2\right)^{1/2}\\
&=\mu_m(s)^{\frac{1}{n}-1}(-\mu_m'(s))\sqrt{\Psi_{G_y^m}(|b-c|^2)(\mu_m(s))}.
\end{align*}
Set now $f(t)=t^{\frac{1}{n}-1}\sqrt{\Psi_{G_y^m}(|b-c|^2)(t)}$ for $t>0$. Then $f\geq 0$ and $f\in L^1(0,\infty)$, since, from \eqref{eq:Hardy} and the fact that $t^{\frac{1}{n}-1}$ is decreasing, we obtain that
\begin{align}\nonumber
\int_0^{\infty}f&\leq\int_0^{\infty}t^{\frac{1}{n}-1}\sqrt{\Psi_{G_y^m}(|b-c|)(t)}\,dt\leq\int_0^{\infty}t^{\frac{1}{n}-1}\left(\sqrt{\Psi_{G_y^m}(|b-c|)}\right)^*(t)\,dt\\
\nonumber
&\leq C\left\|\sqrt{\Psi_{G_y^m}(|b-c|^2)}\right\|_{L^{n,1}(0,|\Omega|)}= C\left\|\Psi_{G_y^m}(|b-c|^2)\right\|_{L^{\frac{n}{2},\frac{1}{2}}(0,|\Omega|)}^{1/2}\\
\label{eq:ToGronwall}
&\leq C\left\||b-c|^2\right\|_{L^{\frac{n}{2},\frac{1}{2}}(\Omega)}^{1/2}=C\|b-c\|_{L^{n,1}(\Omega)},
\end{align}
where $C$ only depends on $n$, and where we have also used Lemma~\ref{ChangeDerivative} and \eqref{eq:LorentzNormPowers}. Then, since $\nu_m\beta_m=(f\circ \mu_m)\cdot(-\mu_m')$, we apply Corrolary~\ref{ChangeOfVariables} and the last estimate to obtain that
\begin{equation}\label{eq:ToGronwall2}
\int_0^{\infty}\nu_m\beta_m=-\int_0^{\infty}f(\mu_m(s))\mu_m'(s)\,ds\leq\int_0^{\infty}f\leq C\|b-c\|_{L^{n,1}(\Omega)}.
\end{equation}
This shows that $\nu_m\beta_m\in L^1(0,\infty)$. Moreover, $\nu_m,\beta_m\geq 0$, and from \eqref{eq:FIsLipschitz} and since $b,c$ are Lipschitz, $\beta_m\gamma_m$ is bounded in $(0,|\Omega|)$ and it vanishes in $(|\Omega|,\infty)$, hence $\beta_m\gamma_m\in L^1(0,\infty)$. Hence, all the hypotheses of Gronwall's inequality (Lemma~\ref{Gronwall}) are satisfied, therefore \eqref{eq:ForGronwall} shows that
\begin{align*}
\gamma_m(t)&\leq C\nu_m(t)+C\nu_m(t)\int_t^{\infty}C\nu_m(\tau)\beta_m(\tau)\exp\left(\int_t^{\tau}\nu_m(s)\beta_m(s)\,ds\right)\,d\tau\\
&\leq C\nu_m(t)+C\nu_m(t)\int_0^{\infty}\nu_m(\tau)\beta_m(\tau)\,d\tau\cdot\exp\left(\int_0^{\infty}\nu_m(\rho)\beta_m(\rho)\,d\rho\right)\leq C\nu_m(t),
\end{align*}
since $\nu_m$ and $\beta_m$ are nonnegative, where $\exp$ is the exponential function, and where we used \eqref{eq:ToGronwall2} in the last estimate. Hence, using the definitions of $\gamma_m$ and $\nu_m$, we obtain that
\begin{equation}\label{eq:UseLater}
\left(-\frac{d}{dt}\int_{[G_y^m>t]}|\nabla G_y^m|^2\right)^{1/2}\leq C\mu_m(t)^{\frac{1}{n}-1}(-\mu_m'(t))^{1/2},
\end{equation}
where $C$ depends on $n,\lambda$ and $\|b-c\|_{n,1}$. Therefore, from Proposition~\ref{TalentiEstimate} we obtain that, for almost every $t>0$, $\mu_m(t)^{\frac{2}{n}-2}(-\mu_m'(t))\geq C$. Therefore, this shows that, for almost every $t>0$, $\mu_m^{\frac{2}{n}-1}$ is differentiable at $t$, and also
\[
\left(\mu_m^{\frac{2}{n}-1}\right)'(t)=\left(\frac{2}{n}-1\right)\mu_m^{\frac{2}{n}-2}(t)\mu_m'(t)=\left(1-\frac{2}{n}\right)\mu_m^{\frac{2}{n}-2}(t)(-\mu_m'(t))\geq\left(1-\frac{2}{n}\right)C.
\]
Since the function $\mu_m(t)^{\frac{2}{n}-1}$ is increasing and nonnegative, the last estimate and Corollary 3.29 in \cite{AmbrosioFuscoPallara} imply that, for $t>0$,
\[
\mu_m(t)^{\frac{2}{n}-1}\geq\mu_m(t/2)^{\frac{2}{n}-1}+\int_{t/2}^t\left(\mu_m(t)^{\frac{2}{n}-1}\right)'\geq\int_{t/2}^t\left(\mu_m(t)^{\frac{2}{n}-1}\right)'\geq\int_{t/2}^tC=Ct.
\]
Hence, $t\mu_m(t)^{1-\frac{2}{n}}\leq C$, and combining with \eqref{eq:EquivalentForInf} we complete the proof.
\end{proof}

The next lemma shows a weak bound on $\nabla G_y^m$.

\begin{lemma}\label{WeakForDG}
Under the same assumptions as in Lemma~\ref{WeakForG}, for any $y\in\Omega$ and $m>\frac{2}{\delta(y)}$,
\[
\|\nabla G_y^m\|_{L^{\frac{n}{n-1},\infty}(\Omega)}\leq C,
\]
where $C$ depends on $n,\lambda$ and $\|b-c\|_{n,1}$ only.
\end{lemma}
\begin{proof}
Fix $y\in\Omega$ and consider the function $H_m$ from \eqref{eq:FIsLipschitz}; then $H_m$ is Lipschitz and increasing in $(0,\infty)$. Set $u_m=(G_y^m)^*$ and let $\mu_m$ be the distribution function of $u_m$. Define also $\Phi_m=H_m\circ u_m$. Since $u_m$ is decreasing, Lemma~\ref{Decomp} shows that, for almost every $s\in(0,\infty)$,
\begin{equation}\label{eq:H}
\Phi_m'(s)=\left(H_m\circ u_m\right)'(s)=H_m'(u_m(s))\cdot (u_m)'(s),
\end{equation}
where we interpret $\Phi_m'(s)$ as $0$ whenever $(u_m)'(s)=0$. Consider now the decomposition $(0,\infty)=G_{G_y^m}\cup D_{G_y^m}\cup N_{G_y^m}$ from Lemma~\ref{Splitting}, and define $B_m$ to be the set of $t>0$ such that \eqref{eq:UseLater} holds. Then $B_m$ has full measure in $(0,\infty)$, hence from Theorem 1 in \cite{SerrinVarberg} and Lemma~\ref{Splitting}, $u_m'(s)=0$ for almost every $s\in G_{G_y^m}\cap u_m^{-1}((0,\infty)\setminus B_m)$. Since $u_m'(s)\neq 0$ for all $s\in G_{G_y^m}$, this shows that $u_m(s)\in B_m$ for almost every $s\in G_{G_y^m}$. Then, for those $s$, using \eqref{eq:UseLater} and \eqref{eq:xGuFormulas} we obtain that
\begin{align*}
\Phi_m'(s)&=(H_m\circ u_m)'(s)=-H_m'(u_m(s))\cdot (-u_m'(s))\\
&\leq C\mu_m(u_m(s))^{\frac{2}{n}-2}(-\mu_m'(u_m(s))(-u_m'(s))=Cs^{\frac{2}{n}-2}.
\end{align*}
In addition, for almost every $s\in D_{G_y^m}$, $u_m'(s)=0$, hence $\Phi_m'(s)=0$ almost everywhere in $D_{G_y^m}$, from \eqref{eq:H}. Therefore, for almost every $s\in(0,\infty)$,
\begin{equation}\label{eq:Phi}
\frac{d}{ds}\int_{[G_y^m>(G_y^m)^*(s)]}|\nabla G_y^m|^2=\Phi_m'(s)\leq Cs^{\frac{2}{n}-2}
\end{equation}
which corresponds to (3.11) in \cite{AlvinoFeroneTrombetti}.

We now fix $m\in\mathbb N$ with $m>\frac{2}{\delta(y)}$ and we follow the proof of Lemma 3.3 in \cite{AlvinoFeroneTrombetti}, to construct a sequence $g_j$ of functions in $L^1(0,\infty)$, such that $g_j^*=|\nabla G_y^m|^*$, and for all $\phi$ which are Lipschitz and compactly supported in $(0,|\Omega|]$,
\begin{equation}\label{eq:j}
\int_0^{|\Omega|}g_j^2\phi\xrightarrow[j\to\infty]{}\int_0^{|\Omega|}\Phi_m'\phi.
\end{equation}
Then, we proceed as in Theorem 3.2 in \cite{AlvinoFeroneTrombetti}: we fix $s>0$, and for $0<\e<s$, we set
\[
\phi_{\e}(t)=\left\{\begin{array}{c l} 0, &0\leq t<\e \\ \displaystyle \frac{s}{s-\e}(t-\e), & \e\leq t<s \\ s, & t\geq s.\end{array}\right.
\]
Then, from \eqref{eq:HardyLower}, and since $g_j^*=|\nabla G_y^m|^*\geq 0$ and $\phi_{\e}\geq 0$ is increasing,
\begin{align*}
\int_0^{\infty}g_j^2\phi_{\e}&\geq\int_0^{\infty}(g_j^2)^*\phi_{\e}=\int_0^{\infty}\left(|\nabla G_y^m|^*\right)^2\phi_{\e}\geq\int_0^s\left(|\nabla G_y^m|^*\right)^2\phi_{\e}\\
&\geq\left(|\nabla G_y^m|^*(s)\right)^2\int_{\e}^s\frac{s}{s-\e}(t-\e)\,dt=\left(|\nabla G_y^m|^*(s)\right)^2\frac{s(s-\e)}{2}.
\end{align*}
Note now that $\phi_{\e}(t)\leq t$ for $t\in(0,s)$. Therefore, letting $j\to\infty$ and using \eqref{eq:j},
\[
\frac{s(s-\e)}{2}\left(|\nabla G_y^m|^*(s)\right)^2\leq\int_0^{\infty}\Phi_m'\phi_{\e}\leq \int_0^st\Phi_m'(t)\,dt+s\int_s^{\infty}\Phi_m'(t)\,dt,
\]
so, letting $\e\to 0$ and using \eqref{eq:Phi}, we obtain that
\begin{equation}\label{eq:UseSimilarLater}
\frac{s^2}{2}\left(|\nabla G_y^m|^*(s)\right)^2\leq \int_0^st\Phi_m'(t)\,dt+s\int_s^{\infty}\Phi_m'(t)\,dt\leq Cs^{\frac{2}{n}}.
\end{equation}
This completes the proof.
\end{proof}

We now improve the weak type bound in Lemma~\ref{WeakForG} to a pointwise bound, using the maximum principle and the weak $L^{\frac{n}{n-2},\infty}$ bound on $G_y^m$.

\begin{lemma}\label{Pointwise}
Under the same assumptions as in Lemma~\ref{WeakForG}, for every $x,y\in\Omega$ and $m\in\mathbb N$  with $m>\frac{2}{\delta(y)}$ and $|x-y|>\frac{2}{m}$, the functions $G_y^m$ satisfy the estimate
\[
G_y^m(x)\leq C|x-y|^{2-n},
\]
where $C$ depends on $n,\lambda,\|A\|_{\infty}$ and $\|b-c\|_{n,1}$ only.
\end{lemma}
\begin{proof}
The proof is similar to the proof of Proposition 6.1 in \cite{KimSak}, which is based on an argument in \cite{Gruter}: consider $y\in\Omega$, and set $r=\frac{1}{4}|x-y|$, $B_s=B_s(x)$ for $s>0$. Then, the assumption $m>\frac{2}{|x-y|}$ shows that $B_{2r}\cap B_{1/m}(y)=\emptyset$. We now distinguish between two cases: $B_{2r}\subseteq\Omega$, and $B_{2r}\not\subseteq\Omega$.
	
In the first case: since $G_y^m\in W^{1,2}(B_{2r})$ and $B_{2r}\cap B_{1/m}(y)=\emptyset$, \eqref{eq:G_y^m} shows that $G_y^m$ is a $W^{1,2}(B_{2r})$ solution to the equation
\[
-\dive(A\nabla G_y^m+bG_y^m)+c\nabla G_y^m+dG_y^m=0
\]
in $B_{2r}$. In addition, since $b,c$ and $d$ are bounded, Theorem 8.22 in \cite{Gilbarg} shows that $G_y^m$ is H{\"o}lder continuous in $B_r$. Therefore, from Proposition~\ref{SupByIntegral} and H{\"o}lder's inequality,
\[
G_y^m(x)\leq\sup_{B_{r/2}}G_y^m\leq C\fint_{B_r}G_y^m\leq Cr^{-n}\|G_y^n\|_{L^{\frac{n}{n-2},\infty}(B_r)}\|\chi_{B_r}\|_{L^{\frac{n}{2},1}(B_r)}\leq Cr^{2-n}=C|x-y|^{2-n},
\]
where $C$ depends on $n,\lambda$, $\|A\|_{\infty}$ and $\|b-c\|_n$. Since $\|b-c\|_n\leq C_n\|b-c\|_{n,1}$, we obtain that $C$ depends on $n,\lambda$, $\|A\|_{\infty}$ and $\|b-c\|_{n,1}$.
	
In the second case: we consider the solution $\widetilde{G}_y^m\in W_0^{1,2}(\Omega)$ to the equation
\[
-\dive\left(A\nabla\widetilde{G}_y^m\right)+(c-b)\nabla\widetilde{G}_y^m=h_m
\]
in $\Omega$, where $h_m$ is as in \eqref{eq:G_y^m}. Then $\widetilde{G}_y^m\geq 0$, and $\tilde{v}=\widetilde{G}_y^m-G_y^m\in W_0^{1,2}$ is a supersolution to $-\dive(A\nabla\tilde{v}+b\tilde{v})+c\nabla\tilde{v}+d\tilde{v}\geq 0$ in $\Omega$; hence Proposition~\ref{MaxPrinciple} shows that $\tilde{v}\geq 0$ in $\Omega$. Therefore $G_y^m\leq\widetilde{G}_y^m$.

Denote by $\overline{A}$ the extension of $A$ by $\lambda I$ outside $\Omega$. Using also Theorem 3 on page 174 of \cite{SteinSingular}, we construct Lipschitz extensions $\overline{b},\overline{c}$ of $b,c$ in $\bR^n$, such that $\|\overline{b}-\overline{c}\|_{L^{n,1}(\bR^n)}\leq 2\|b-c\|_{n,1}$. We also consider $\overline{G}_y^m\in W_0^{1,2}(\Omega\cup B_{2r})$ that solves the equation
\[
-\dive\left(\overline{A}\nabla\overline{G}_y^m\right)+\left(\overline{c}-\overline{b}\right)\nabla\overline{G}_y^m=h_m
\]
in $\Omega\cup B_{2r}$. Then, $\overline{v}=\overline{G}_y^m-\widetilde{G}_y^m$ is a $W^{1,2}(\Omega)$ solution to the equation $-\dive\left(A\nabla\overline{v}\right)+(c-b)\nabla\overline{v}=0$ in $\Omega$. From Proposition~\ref{MaxPrinciple}, $\overline{G}_y^m\geq 0$ in $\Omega\cup B_{2r}$, therefore $\overline{v}\geq 0$ on $\partial\Omega$. Hence Proposition~\ref{MaxPrinciple} shows that $\overline{v}\geq 0$ in $\Omega$, therefore $\widetilde{G}_y^m\leq\overline{G}_y^m$ in $\Omega$. Since $G_y^m\leq\widetilde{G}_y^m$, we have that $G_y^m\leq\overline{G}_y^m$. Moreover, since $h_m$ vanishes in $B_{2r}$, we apply Lemma~\ref{WeakForG} and the argument of the first case to $\overline{G}_y^m$ to obtain that
\[
G_y^m(x)\leq\overline{G}_y^m(x)\leq\sup_{B_{r/2}}\overline{G}_y^m\leq C\fint_{B_r}\overline{G}_y^m\leq Cr^{-n}\|\overline{G}_y^m\|_{L^{\frac{n}{n-2},\infty}(B_r)}\|\chi_{B_r}\|_{L^{\frac{n}{2},1}(B_r)}\leq C|x-y|^{2-n},
\]
where $C$ depends on $n,\lambda$, $\|\overline{A}\|_{\infty}$ and $\|\overline{b}-\overline{c}\|_{n,1}$; hence, $C$ depends on $n,\lambda$, $\|A\|_{\infty}$ and $\|b-c\|_{n,1}$ only. This completes the proof.
\end{proof}

\subsection{Estimates for \texorpdfstring{$g$}{g}}
In this subsection we will show estimates for approximate Green's function for the adjoint equation
\[
-\dive(A^t\nabla u+cu)+b\nabla u+du=0
\]
in a bounded domain $\Omega$, where $b,c\in\Lip(\Omega)$, $d\in L^{\infty}(\Omega)$, and $d\geq\dive b$. Under these assumptions, we fix $x\in\Omega$, and for $k>\frac{2}{\delta(x)}$ consider the approximate Green's function $g_x^k\in W_0^{1,2}(\Omega)$, which solves the equation
\begin{equation}\label{eq:g_x^k}
-\dive(A^t\nabla g_x^k+cg_x^k)+b\nabla g_x^k+dg_x^k=f_k:=\frac{1}{|B_{1/k}(x)|}\chi_{B_{1/k}(x)}.
\end{equation}
The existence of these functions follows as right before (5.2) in \cite{KimSak}.

The next lemma follows using $g_x^k$ as a test function in \eqref{eq:G_y^m} and $G_y^m$ as a test function in \eqref{eq:g_x^k}.
\begin{lemma}\label{Symmetry}
Under the same assumptions as in Lemma~\ref{WeakForG}, and if $G_y^m,g_x^k$ are as in \eqref{eq:G_y^m} and \eqref{eq:g_x^k}, respectively, then, for $m>\frac{2}{\delta(y)}$ and $k>\frac{2}{\delta(x)}$,
\[
\fint_{B_{1/m}(y)}g_x^k=\fint_{B_{1/k}(x)}G_y^m.
\]
In particular, $g_x^k\geq 0$ in $\Omega$.
\end{lemma}
The following lemma shows an $L^{\frac{n}{n-2},\infty}(\Omega)$ estimate for $g_x^k$. We also deduce preliminary bounds that will lead to a $L^{\frac{n}{n-1},\infty}$ type estimate for $\nabla g_x^k$.

\begin{lemma}\label{WeakForGt}
Let $\Omega\subseteq\bR^n$ be a bounded domain. Let $A$ be uniformly elliptic and bounded in $\Omega$, and let $b,c\in\Lip(\Omega)$, $d\in L^{\infty}(\Omega)$, with $d\geq\dive b$. For any $x\in\Omega$ fixed and $k\in\mathbb N$ with $k>\frac{2}{\delta(y)}$ consider the function $g_x^k$ in \eqref{eq:g_x^k} and let $\mu_k$ be its distribution function. Then, the function $\displaystyle\tilde{H}_k(t)=\int_{[g_x^k>t]}|\nabla g_x^k|^2$ is Lipschitz in $(0,\infty)$. In addition, if $\Psi_{g_x^k}$ is as in \eqref{eq:DDfn}, for almost every $t>0$,
\begin{equation}\label{eq:UseLatert}
-\tilde{H}'_k(t)\leq C+Ct^2\Psi_{g_x^k}(|b-c|^2)(\mu_k(t))(-\mu_k'(t)). 
\end{equation}
Finally, if $R_k(t)=Ct^{\frac{1}{n}-1}\sqrt{\Psi_{g_x^k}(|b-c|^2)(t)}$ and $v_k(t)=(g_x^k)^*(t)$, then 
\begin{equation}\label{eq:hBounds}
\int_0^{\infty}R_k(s)\,ds\leq C,\,\,\,\text{and}\,\,\,v_k(t)\leq Ct^{\frac{2}{n}-1},\,\,\,-v_k'(t)\leq Ct^{\frac{2}{n}-2}+Ct^{\frac{2}{n}-1}R_k(t)\,\,\,\text{for almost every}\,\,\,t.
\end{equation}
In particular, $\|g_x^k\|_{L^{\frac{n}{n-2},\infty}(\Omega)}\leq C$, where $C$ depends on $n,\lambda$ and $\|b-c\|_{n,1}$ only.
\end{lemma}
\begin{proof}
Let $T_{t,h}$ be as in \eqref{eq:Tth}, and use $\phi=T_{t,h}(g_x^k)$ as a test function. Since $sT_{t,h}(s)\geq 0$ and $d\geq\dive b$, we obtain that
\begin{equation}\label{eq:At}
\int_{\Omega}A^t\nabla g_x^k\nabla\phi\leq\fint_{B_{1/k}(x)}\phi+\int_{\Omega}(c-b)\nabla\phi\cdot g_x^k.
\end{equation}
For the last integral, we integrate by parts and use that $g_x^k\in W_0^{1,2}(\Omega)$ and $|\phi|\leq h$, to estimate
\begin{align*}
\int_{\Omega}(c-b)\nabla\phi\cdot g_x^k&=\int_{\Omega}(c-b)\nabla(\phi g_x^k)-\int_{\Omega}(c-b)\nabla g_x^k\cdot\phi=-\int_{\Omega}\dive(c-b)\cdot\phi g_x^k-\int_{\Omega}(c-b)\nabla g_x^k\cdot\phi\\
&\leq\left(\|\dive(b-c)\|_{\infty}\|g_x^k\|_{L^1(\Omega)}+\|b-c\|_{L^2(\Omega)}\|\nabla g_x^k\|_{L^2(\Omega)}\right)h\leq\tilde{C}h,
\end{align*}
for some $\tilde{C}>0$. Plugging the last estimate in \eqref{eq:At} and using that $|\phi|\leq h$, $\nabla\phi=\nabla g_x^k$ in $[t<g_x^k\leq t+h]$ and $\nabla\phi=0$ otherwise, we obtain that
\[
\lambda\int_{[t<g_x^k\leq t+h]}|\nabla g_x^k|^2\leq\int_{\Omega}A^t\nabla g_x^k\nabla\phi\leq h+\tilde{C}h,
\]
therefore $\tilde{H}$ is Lipschitz continuous.

We now return to \eqref{eq:At}, and we estimate
\[
\lambda\int_{[t<g_x^k\leq t+h]}|\nabla g_x^k|^2\leq h+\int_{[t<g_x^k\leq t+h]}(c-b)\nabla g_x^k\cdot g_x^k\leq h+(t+h)\int_{[t<g_x^k\leq t+h]}|b-c||\nabla g_x^k|.
\]
Hence, after dividing by $h$, using the Cauchy-Schwartz inequality as in \eqref{eq:CSForDerivativeLater} and letting $h\to 0$, we obtain that
\begin{align*}
-\lambda\frac{d}{dt}\int_{[g_x^k>t]}|\nabla g_x^k|^2&\leq 1+t\left(-\frac{d}{dt}\int_{[g_x^k>t]}|b-c|^2\right)^{1/2}\left(-\frac{d}{dt}\int_{[g_x^k>t]}|\nabla g_x^k|^2\right)^{1/2}\\
&\leq 1+\frac{t^2}{2\lambda}\left(-\frac{d}{dt}\int_{[g_x^k>t]}|b-c|^2\right)+\frac{\lambda}{2}\left(-\frac{d}{dt}\int_{[g_x^k>t]}|\nabla g_x^k|^2\right).
\end{align*}
Hence, for almost every $t>0$,
\begin{align*}
-\frac{d}{dt}\int_{[g_x^k>t]}|\nabla g_x^k|^2&\leq C+Ct^2\left(-\frac{d}{dt}\int_{[g_x^k>t]}|b-c|^2\right)\leq C+Ct^2D_{g_x^k}(|b-c|^2)(\mu_k(t))(-\mu_k'(t)),
\end{align*}
where we used Lemma~\ref{ChangeDerivative} for the last estimate, and $C$ depends on $\lambda$ only. This shows \eqref{eq:UseLatert}.

We now multiply both sides of the last estimate with $\mu_k^{\frac{2}{n}-2}(-\mu_k')$ and we apply Lemma~\ref{TalentiEstimate}, to obtain that
\begin{equation}\label{eq:Bk}
1\leq C\mu_k(t)^{\frac{2}{n}-2}(-\mu_k'(t))+Ct^2\mu_k(t)^{\frac{2}{n}-2}\Psi_{g_x^k}(|b-c|^2)(\mu_k(t))(-\mu_k'(t))^2,
\end{equation}
in a set $B_k\subseteq(0,\infty)$ with full measure.

Consider the decomposition $(0,\infty)=G_{g_x^k}\cup D_{g_x^k}\cup N_{g_x^k}$ from Lemma~\ref{Splitting}. From  Lemma~\ref{Splitting}, $v_k$ is differentiable at $s$ for every $s\in G_{g_x^k}$, therefore Theorem 1 in \cite{SerrinVarberg} shows that $v_k'(s)=0$ for almost every $s\in (v_k)^{-1}((0,\infty)\setminus B_k)$. Since $v_k'(s)\neq 0$ for every $s\in G_{g_x^k}$, this shows that $v_k(s)\in B_k$ for almost every $s\in G_{g_x^k}$. For these $s$, plugging $v_k(s)$ in \eqref{eq:Bk}, we obtain that
\[
1\leq C\mu_k(v_k(s))^{\frac{2}{n}-2}(-\mu_k'(v_k(s)))+Cv_k(s)^2\mu_k(v_k(s))^{\frac{2}{n}-2}\Psi_{g_x^k}(|b-c|^2)(\mu_k(v_k(s)))(-\mu_k'(v_k(s)))^2.
\]
Then, using the formulas in \eqref{eq:xGuFormulas} and the last estimate we obtain that
\[
1\leq Cs^{\frac{2}{n}-2}\left(-\frac{1}{v_k'(s)}\right)+Cv_k(s)^2s^{\frac{2}{n}-2}\Psi_{g_x^k}(|b-c|^2)(s)\left(-\frac{1}{v_k'(s)}\right)^2.
\]
Hence, after multiplying with $(-v_k'(s))^2$ we obtain that for almost every $s\in G_{g_x^k}$,
\[
(-v_k'(s))^2\leq Cs^{\frac{2}{n}-2}\left(-v_k'(s)\right)+CR_k(s)^2v_k(s)^2\leq\frac{C^2}{2}s^{\frac{4}{n}-4}+\frac{1}{2}\left(-v_k'(s)\right)^2+CR_k(s)^2v_k(s)^2,
\]
and after rearranging and taking square roots, this implies that
\[
-v_k'(s)\leq Cs^{\frac{2}{n}-2}+CR_k(s)v_k(s)
\]
for almost every $s\in G_{g_x^k}$. Moreover, for every $s\in D_{g_x^k}$, $-v_k'(t)=0$. Hence, we obtain that, for almost every $s>0$, and for some $C>0$ that only depends on $\lambda$,
\begin{equation}\label{eq:AfterT}
-v_k'(s)\leq Cs^{\frac{2}{n}-2}+CR_k(s)v_k(s).
\end{equation}
Fix $t>0$, and note that $v_k$ is absolutely continuous in $[t,|\Omega|]$, from Lemma~\ref{AbsoluteContinuity}. Therefore, integrating the last inequality in $[t,|\Omega|]$, we obtain that
\begin{align*}
v_k(t)&=v_k(t)-v_k(|\Omega|)=\int_t^{|\Omega|}-v_k'(\tau)\,d\tau\leq C\int_t^{\infty}\left(\tau^{\frac{2}{n}-2}+R_k(\tau)v_k(\tau)\right)\,d\tau\\
&\leq Ct^{\frac{2}{n}-1}+C\int_t^{\infty}R_k(\tau)v_k(\tau)\,d\tau.
\end{align*}
Note now that, similarly to \eqref{eq:ToGronwall}, using \eqref{eq:Hardy} and Lemma~\ref{LorentzEstimate}, we obtain
\begin{equation}\label{eq:Rk}
\int_0^{\infty}R_k(\tau)\,d\tau\leq C\int_0^{\infty}\tau^{\frac{1}{n}-1}\sqrt{\Psi_{g_x^k}(|b-c|^2)(\tau)}\,d\tau\leq C\|b-c\|_{L^{n,1}(\Omega)},
\end{equation}
therefore $R_k$ is integrable in $(0,\infty)$; this also shows the first estimate in \eqref{eq:hBounds}. Now, for $t_0>0$ fixed, the function $t^{\frac{2}{n}-1}R_k(t)$ is integrable in $(t_0,\infty)$, and $R_kv_k\in L^1(t_0,\infty)$ since $v_k(t)\leq v_k(t_0)$ in $(t_0,\infty)$. Therefore all the hypotheses in Gronwall's inequality (Lemma~\ref{Gronwall}) are satisfied, and we obtain that, for any $t>t_0$,
\[
v_k(t)\leq Ct^{\frac{2}{n}-1}+C\int_t^{\infty}\left(\tau^{\frac{2}{n}-1}R_k(\tau)\exp\left(\int_t^{\tau}R_k(\rho)\,d\rho\right)\right)\,d\tau\leq Ct^{\frac{2}{n}-1},
\]
where we used \eqref{eq:Rk} for the last estimate, and where $C$ depends on $n,\lambda$ and $\|b-c\|_{n,1}$. This shows the second estimate in \eqref{eq:hBounds}. Finally, to show the third estimate in \eqref{eq:hBounds} we plug the last estimate back to \eqref{eq:AfterT}, and this completes the proof.
\end{proof}

We now show a weak type bound for $\nabla g_x^k$, which is the analog of Lemma~\ref{WeakForDG} for $g_x^k$.

\begin{lemma}\label{WeakForDGt}
Under the same assumptions as in Lemma~\ref{WeakForGt}, for any $x\in\Omega$ and $k>\frac{2}{\delta(x)}$,
\[
\|\nabla g_x^k\|_{L^{\frac{n}{n-1},\infty}(\Omega)}\leq C,
\]
where $C$ depends on $n,\lambda$ and $\|b-c\|_{n,1}$ only.
\end{lemma}
\begin{proof}
The proof is similar to the proof of Lemma~\ref{WeakForDG}. Fix $x\in\Omega$ and consider the functions $\tilde{H}_k$, $\mu_k$ and $v_k$ from Lemma~\ref{WeakForGt}, then $\tilde{H}_k$ is Lipschitz and increasing in $(0,\infty)$.  Define $\tilde{\Phi}_k=\tilde{H}_k\circ v_k$, and consider the decomposition $(0,\infty)=G_{g_x^k}\cup D_{g_x^k}\cup N_{g_x^k}$ from Lemma~\ref{Splitting}. Then, since $\tilde{H}_k$ is Lipschitz and $v_k$ is decreasing, we apply Lemma~\ref{Decomp} and then Theorem 1 in \cite{SerrinVarberg} (as right after \eqref{eq:Bk}) to obtain \eqref{eq:UseLatert} for $t=v_k(s)$ for almost every $s\in G_{g_x^k}$; then, \eqref{eq:xGuFormulas} shows that, for almost every $s\in G_{g_x^k}$,
\begin{align*}
\tilde{\Phi}_k'(s)&=(\tilde{H}_k\circ v_k)'(s)=-\tilde{H}_k'(v_k(s))\cdot (-v_k'(s))\\
&\leq C\left(1+v_k(s)^2\Psi_{g_x^k}(|b-c|^2)(s)\cdot(-\mu_k'(v_k(s)))\right)\cdot (-v_k'(s))\\
&\leq C(-v_k'(s))+Cv_k(s)^2\Psi_k(s)\leq Cs^{\frac{2}{n}-2}+Cs^{\frac{2}{n}-1}R_k(s)+Cs^{\frac{4}{n}-2}\Psi_k(s),
\end{align*}
where $\Psi_k=\Psi_{g_x^k}(|b-c|^2)$, and where we used \eqref{eq:hBounds} for the last estimate. In addition, for every $s\in D_{g_x^k}$, $v_k'(s)=0$, hence $\tilde{\Phi}_k'(s)=0$ for almost every $s\in D_{g_x^k}$, from Lemma~\ref{Decomp}. Therefore, for almost every $s\in(0,\infty)$,
\begin{equation}\label{eq:PhiTilde}
\frac{d}{ds}\int_{[g_x^k>(g_x^k)^*(s)]}|\nabla g_x^k|^2=\tilde{\Phi}_k'(s)\leq Cs^{\frac{2}{n}-2}+Cs^{\frac{2}{n}-1}R_k(s)+Cs^{\frac{4}{n}-2}\Psi_k(s).
\end{equation}
Now, as in the proof of Lemma~\ref{WeakForDG}, we fix $k$ and we construct a sequence $G_j$ of functions in $L^1(0,\infty)$, such that $G_j^*=|\nabla g_x^k|^*$, and for all $\phi$ which are Lipschitz and compactly supported in $(0,|\Omega|]$,
\[
\int_0^{|\Omega|}G_j^2\phi\xrightarrow[j\to\infty]{}\int_0^{|\Omega|}\tilde{\Phi}_k'\phi.
\]
Using a procedure as in Lemma~\ref{WeakForDG}, we then obtain the analog of the first estimate in \eqref{eq:UseSimilarLater} for $\tilde{\Phi}_k'$: that is, for any $s>0$,
\begin{equation}\label{eq:nablagxkbound}
\frac{s^2}{2}\left(|\nabla g_x^k|^*(s)\right)^2\leq\int_0^st\tilde{\Phi}_k'(t)\,dt+s\int_s^{\infty}\tilde{\Phi}_k'(t)\,dt=I_1+I_2.
\end{equation}
To bound $I_1$, using \eqref{eq:PhiTilde}, we estimate
\begin{align}\nonumber
I_1&\leq C\int_0^s\left(t^{\frac{2}{n}-1}+t^{\frac{2}{n}}R_k(t)+t^{\frac{4}{n}-1}\Psi_k(t)\right)\,dt\leq Cs^{\frac{2}{n}}\left(1+\|R_k\|_{L^1(0,\infty)}+\int_0^{\infty}t^{\frac{2}{n}-1}\Psi_k(t)\,dt\right)\\
\nonumber 
&\leq Cs^{\frac{2}{n}}\left(1+\|R_k\|_{L^1(0,\infty)}+\int_0^{\infty}t^{\frac{2}{n}-1}\Psi_k^*(t)\,dt\right)=Cs^{\frac{2}{n}}\left(1+\|R_k\|_{L^1(0,\infty)}+\|\Psi_k\|_{L^{\frac{n}{2},1}(0,\infty)}\right)\\
\label{eq:DNormBound}
&\leq Cs^{\frac{2}{n}}\left(1+\|R_k\|_{L^1(0,\infty)}+\|\Psi_k\|_{L^{\frac{n}{2},\frac{1}{2}}(0,\infty)}\right)\leq Cs^{\frac{2}{n}},
\end{align}
where $C$ depends only on $n,\lambda$ and $\|b-c\|_{n,1}$, and where we used \eqref{eq:Hardy} for the third estimate, \eqref{eq:LorentzDfn} for the first equality, \eqref{eq:LorentzNormsRelations} for the fourth estimate, and \eqref{eq:hBounds} and Lemma~\ref{LorentzEstimate} for the fifth estimate. Similarly, for $I_2$, if $t\geq s$ and $\frac{2}{n}-1<0$ we have that $t^{\frac{2}{n}-1}\leq s^{\frac{2}{n}-1}$, therefore
\begin{align*}
I_2&\leq Cs\int_s^{\infty}\left(t^{\frac{2}{n}-2}+t^{\frac{2}{n}-1}R_k(t)+t^{\frac{4}{n}-2}\Psi_k(t)\right)\,dt\leq Cs^{\frac{2}{n}}+Cs^{\frac{2}{n}}\int_s^{\infty}\left(R_k(t)+t^{\frac{2}{n}-1}\Psi_k(t)\right)\,dt\\
&\leq Cs^{\frac{2}{n}}\left(1+\|R_k\|_{L^1(0,\infty)}+\int_0^{\infty}t^{\frac{2}{n}-1}\Psi_k(t)\,dt\right)\leq Cs^{\frac{2}{n}},
\end{align*}
where the last estimate follows as in \eqref{eq:DNormBound}. Then, plugging the last estimate together with \eqref{eq:DNormBound} to \eqref{eq:nablagxkbound} completes the proof.
\end{proof}

\section{Constructions}
\subsection{A preliminary construction}
In this subsection we pass to the limit as $m\to\infty$ for $G_y^m$ and as $k\to\infty$ for $g_x^k$ to construct Green's functions in the case that $\Omega$ is bounded, $b,c$ are Lipschitz and $d$ is bounded.

Fix $q_0\in\left(1,\frac{n}{n-1}\right)$. Under the same assumptions as in Lemma~\ref{WeakForG}, for $y\in\Omega$ fixed and $m>\frac{2}{\delta(y)}$, the functions $(G_y^m)$ have uniformly bounded $W_0^{1,q_0}(\Omega)$ norm. Moreover, if $r>0$, then for $m>\frac{2}{r}$, $G_y^m$ is a $W^{1,2}(\Omega\setminus B_r(y))$ solution of the equation
\[
-\dive(A\nabla G_y^m+bG_y^m)+c\nabla G_y^m+dG_y^m=0
\]
in $\Omega\setminus B_r(y)$. In addition, Lemma~\ref{Pointwise} shows that, for $m$ sufficiently large, $G_y^m(x)\leq C'|x-y|^{2-n}$ in $\Omega\setminus B_r(y)$ , where $C'$ depends on $n,\lambda,\|A\|_{\infty}$ and $\|b-c\|_{n,1}$. Hence, choosing $\phi$ to be a smooth cutoff function, with $\phi\equiv 1$ outside $B_r(y)$, $\phi\equiv 0$ in $B_{r/2}(y)$ and $|\nabla\phi|\leq\frac{C}{r}$, and $m$ large enough so that $h_m$ vanishes in $B_r$, Lemma \ref{DerivativeByU} and the pointwise bound on $G_y^m$ shows that
\begin{equation}\label{eq:boundL2}
\int_{\Omega\setminus B_r(y)}|\nabla G_y^m|^2\leq C'\left(\int_{\Omega\setminus B_{r/2}(y)}|G_y^m|^{2^*}\right)^{\frac{2}{2^*}}+\frac{C'}{r^2}\int_{\Omega\setminus B_{r/2}(y)}|G_y^m|^2\leq C'r^{2-n},
\end{equation}
where $C'$ depends on $n,\lambda,\|A\|_{\infty}$ and $\|b-c\|_{n,1}$.

Since $L^{\frac{n}{n-2},\infty}(\Omega)$, $L^{\frac{n}{n-1},\infty}(\Omega)$ are the dual spaces of $L^{\frac{n}{2},1}(\Omega)$ and $L^{n,1}(\Omega)$, respectively, Lemmas~\ref{WeakForG} and \ref{WeakForDG} show that there exists a subsequence $G_y^{i_m}$ and $G_y\in W_0^{1,q_0}(\Omega)$, such that $G_y\in L^{\frac{n}{n-2},\infty}(\Omega)$, $\nabla G_y\in L^{\frac{n}{n-1},\infty}(\Omega)$, and also
\begin{equation}\label{eq:WeakAllBounded}
\begin{gathered}
G_y^{i_m}\rightharpoonup G_y\quad\text{weakly}^*\,\,\text{in}\,\,\,L^{\frac{n}{n-2},\infty}(\Omega),\quad\nabla G_y^{i_m}\rightharpoonup \nabla G_y\quad\text{weakly}^*\,\,\text{in}\,\,\,L^{\frac{n}{n-1},\infty}(\Omega),\\
G_y^{i_m}\rightharpoonup G_y\quad\text{weakly}\,\,\text{in}\,\,\,W_0^{1,q_0}(\Omega),\quad G_y^{i_m}\to G_y\quad\text{almost everywhere}.
\end{gathered}
\end{equation}
Moreover, from \eqref{eq:boundL2}, $\nabla G_y^{i_m}\rightharpoonup \nabla G_y$, weakly in $L^2(\Omega\setminus B_r)$. So, using Lemmas~\ref{WeakForG}, \ref{WeakForDG}, \ref{Pointwise} and \eqref{eq:boundL2}, we obtain that
\begin{equation}\label{eq:GPrelim}
\|G_y\|_{L^{\frac{n}{n-2},\infty}(\Omega)}+\|\nabla G_y\|_{L^{\frac{n}{n-1},\infty}(\Omega)}\leq C,\quad G_y(x)\leq C'|x-y|^{2-n},\quad \|\nabla G_y\|_{L^2(\Omega\setminus B_r)}^2\leq C'r^{2-n},
\end{equation}
where $C$ depends on $n,\lambda$ and $\|b-c\|_{n,1}$, and $C'$ depends on $n,\lambda,\|A\|_{\infty}$ and $\|b-c\|_{n,1}$.

Let now $\delta(y)>\e_2>\e_1>0$ and consider any function $\psi\in C^{\infty}(\bR^n)$ with $\psi\equiv 0$ in $B_{\e_1}(y)$ and $\psi\equiv 1$ outside $B_{\e_2}(y)$. Then $\nabla(G_y\psi)=\psi\nabla G_y+G_y\nabla\psi$, hence \eqref{eq:GPrelim} shows that
\begin{equation}\label{eq:GlobalL2AllBounded}
G_y\psi\in W_0^{1,2}(\Omega).
\end{equation}
In addition, since $\Omega$ and $b,c,d$ are bounded, Lemma 4.4 in \cite{KimSak} shows that, for any $f\in L^{\infty}(\Omega)$, there exists a unique $u\in W_0^{1,2}(\Omega)$ that solves the equation $-\dive(A^t\nabla u+cu)+b\nabla u+du=f$ in $\Omega$. Then, since $u,G_y^m\in W_0^{1,2}(\Omega)$, using $u$ as a test function in \eqref{eq:G_y^m} and $G_y^m$ as a test function in the variational definition of $u$, we obtain that
\[
\fint_{B_{1/m}(y)}u=\int_{\Omega}A\nabla G_y^m\nabla u+c\nabla G_y^m\cdot u+b\nabla u\cdot G_y^m+dG_y^mu=\int_{\Omega}G_y^mf.
\]
From Theorem 8.22 in \cite{Gilbarg}, $u$ is continuous in $\Omega$. Hence, letting $m\to\infty$ and using \eqref{eq:WeakAllBounded}, we obtain that
\begin{equation}\label{eq:FormulaAllBounded}
u(y)=\int_{\Omega}G_yf\,\,\,\text{is the}\,\,\,W_0^{1,2}(\Omega)\,\,\,\text{solution to}\,\,\,-\dive(A^t\nabla u+cu)+b\nabla u+du=f\,\,\,\text{in}\,\,\,\Omega.
\end{equation}
We also note that, from \eqref{eq:WeakAllBounded} and \eqref{eq:G_y^m}, for any $\phi\in C_c^{\infty}(\Omega)$ and any $y\in\Omega$,
\begin{equation}\label{eq:DiracForBoundedBounded}
\int_{\Omega}A\nabla G_y\nabla\phi+b\nabla\phi\cdot G_y+c\nabla G_y\cdot\phi+dG_y\phi=\phi(y).
\end{equation}

We now turn to Green's function for the adjoint equation. Under the assumptions of Lemma~\ref{Symmetry} note that, if $y\neq x$, then the function $g_x^k$ is continuous at $y$ from Theorem 8.22 in \cite{Gilbarg}. Hence, letting $m\to\infty$ in Lemma~\ref{Symmetry} and using \eqref{eq:GPrelim}, we obtain that
\begin{equation}\label{eq:Sym}
g_x^k(y)=\lim_{m\to\infty}\fint_{B_{1/m}(y)}g_x^k=\lim_{m\to\infty}\fint_{B_{1/k}(x)}G_y^m=\fint_{B_{1/k}(x)}G_y\leq C\fint_{B_{1/k}(x)}|z-y|^{2-n}\,dz,
\end{equation}
which implies that $g_x^k(y)\leq C|x-y|^{2-n}$ if $k>\frac{2}{|x-y|}$. Using Lemma~\ref{DerivativeByU}, the bound on $\|g_x^k\|_{L^{\frac{n}{n-2},\infty}}$ from Lemma~\ref{WeakForGt} instead of Lemma~\ref{WeakForG}, and Lemma~\ref{WeakForDGt} instead of Lemma~\ref{WeakForDG}, an argument identical to the one before \eqref{eq:GPrelim} implies the existence of a subsequence $g_x^{j_k}$ and a function $g_x\in W_0^{1,q_0}(\Omega)$, with $g_x\in L^{\frac{n}{n-2},\infty}(\Omega)$, $\nabla g_x\in L^{\frac{n}{n-1},\infty}(\Omega)$, $\nabla g_x\in L^2(\Omega\setminus B_r)$ for any $r>0$ fixed, such that $g_x^{j_k}\to g_x$ almost everywhere, and also
\begin{equation}\label{eq:gPrelim}
\|g_x\|_{L^{\frac{n}{n-2},\infty}(\Omega)}+\|\nabla g_x\|_{L^{\frac{n}{n-1},\infty}(\Omega)}\leq C,\quad g_x(y)\leq C'|x-y|^{2-n},\quad \|\nabla g_x\|_{L^2(\Omega\setminus B_r)}^2\leq C'r^{2-n},
\end{equation}
where $C$ depends on $n,\lambda$ and $\|b-c\|_{n,1}$, and $C'$ depends on $n,\lambda,\|A\|_{\infty}$ and $\|b-c\|_{n,1}$. Moreover, for every $\delta(y)>\e_2>\e_1>0$ and and $\psi\in C^{\infty}(\bR^n)$ with $\psi\equiv 0$ in $B_{\e_1}(y)$ and $\psi\equiv 1$ outside $B_{\e_2}(y)$, we obtain that
\begin{equation}\label{eq:GlobalL2AllBoundedt}
g_x\psi\in W_0^{1,2}(\Omega).
\end{equation}
Using Lemma 4.2 in \cite{KimSak} and an argument similar to the one before \eqref{eq:FormulaAllBounded} we obtain that, for all $f\in L^{\infty}(\Omega)$,
\begin{equation}\label{eq:FormulaAllBoundedt}
v(y)=\int_{\Omega}g_xf\,\,\,\text{is the}\,\,\,W_0^{1,2}(\Omega)\,\,\,\text{solution to}\,\,\,-\dive(A\nabla v+bv)+c\nabla v+dv=f\,\,\,\text{in}\,\,\,\Omega.
\end{equation}
In addition, for any $\phi\in C_c^{\infty}(\Omega)$ and any $x\in\Omega$,
\begin{equation}\label{eq:DiracForBoundedBoundedt}
\int_{\Omega}A^t\nabla g_x\nabla\phi+c\nabla\phi\cdot g_x+b\nabla g_x\cdot\phi+dg_x\phi=\phi(y).
\end{equation}
Now, if $x\neq y$ and $k>\frac{2}{|x-y|}$, \eqref{eq:Sym} shows that $g_x^k$ is uniformly bounded in $B_{2\e}(y)$ for $\e>0$ sufficiently small. So, Theorem 8.22 in \cite{Gilbarg} shows that $g_x^k$ is equicontinuous in $B_{\e}(y)$; hence a subsequence $(g_x^{j_k'})$ of $(g_x^{j_k})$ converges uniformly to $g_x$ in $U$. Also, if $k$ is large enough, \eqref{eq:GPrelim} shows that $G_y\in W^{1,2}(B_{1/k}(x))$, and it solves the equation $-\dive(A\nabla G_y+bG_y)+c\nabla G_y+dG_y=0$ in $B_{1/k}$; so, from Theorem 8.22 in \cite{Gilbarg}, $G_y$ is continuous at $x$. Hence, from the equalities in \eqref{eq:Sym},
\begin{equation}\label{eq:Symmetry}
g_x(y)=\lim_{k\to\infty}g_x^{j_k'}(y)=\lim_{k\to\infty}\fint_{B_{1/j_k'}(x)}G_y=G_y(x).
\end{equation}
We now set $G(x,y)=G_y(x)$ and $g_x(y)=g(y,x)$ for $(x,y)\in\Omega^2\setminus\Delta$, where $\Delta=\{(x,x):x\in\Omega\}$. Then, from Theorem 8.22 in \cite{Gilbarg},
\begin{equation}
\label{eq:JointCont}
G,g\,\,\,\text{are continuous in}\,\,\,\Omega^2\setminus\Delta.
\end{equation}
This also shows that $G,g$ are measurable in $\Omega^2$.

Finally, consider a set $\mathcal{U}\subseteq\Omega^2$ with smooth boundary. Then, for any $\Phi\in C_c^{\infty}(\mathcal{U})$ and $x,y$ fixed, the functions $\Phi_y^1(z)=\Phi(z,y)$ and $\Phi_x^2(z)=\Phi(x,z)$ belong to $C_c^{\infty}(\Omega)$. Hence, for $i\in\{1,\dots n\}$,
\[
\int_{\mathcal{U}}G\cdot\partial_i\Phi=\int_{\Omega}\left(\int_{\Omega}G(z,w)\partial_i\Phi_w^1(z)\,dz\right)\,dw=-\int_{\Omega}\left(\int_{\Omega}\partial_iG_w(z)\cdot\Phi_w^1(z)\,dz\right)\,dw.
\]
So, $\partial_iG(z,w)=-\partial_iG_w(z)$, therefore if $U_1,U_2$ are the projections of $\mathcal{U}\subseteq\Omega^2$ in the first and second component, using the weak $L^{\frac{n}{n-2},\infty}$ bound from \eqref{eq:GPrelim}, we obtain that
\[
\|\partial_iG\|_{L^{q_0}(\mathcal{U})}^{q_0}=\int_{U_2}\left(\int_{U_1}|\partial_iG_w|^{q_0}\right)\,dw\leq\tilde{C},
\]
where $\tilde{C}$ depends on $n,\lambda$, $\|b-c\|_{n,1}$ and $\mathcal{U}$. Moreover, \eqref{eq:Symmetry} shows that, for $i=n+1,\dots 2n$,
\begin{equation}\label{eq:FirstGrad}
\int_{\mathcal{U}}G\cdot\partial_i\Phi=\int_{\Omega}\left(\int_{\Omega}g(w,z)\partial_i\Phi_z^2(w)\,dw\right)\,dz=-\int_{\Omega}\left(\int_{\Omega}\partial_ig_z(w)\cdot\Phi_z^2(w)\,dw\right)\,dz.
\end{equation}
So, $\partial_iG(z,w)=-\partial_ig_z(w)$, and similarly to \eqref{eq:FirstGrad}, we finally obtain that the $2n$-dimensional gradient of $G$ is uniformly bounded in $L^{q_0}(\mathcal{U})$. Hence,
\begin{equation}\label{eq:JointBoundedness}
\|G\|_{W^{1,q_0}(\mathcal{U})}\leq\tilde{C},\,\,\,\text{and}\,\,\,\|g\|_{W^{1,q_0}(\mathcal{U})}\,\leq\tilde{C},
\end{equation}
where $\tilde{C}$ depends on $n,\lambda$, $\|b-c\|_{n,1}$ and $\mathcal{U}$, and the second estimate follows from the first one and \eqref{eq:Symmetry}.

\subsection{Constructions in general domains}
In this subsection we will construct Green's function for general domains $\Omega\subseteq\bR^n$ and coefficients that are not necessarily bounded. This will be done in two steps: we will first assume that $\Omega$ is bounded and we will drop the Lipschitz assumption on the lower order coefficients, and we will then drop the boundedness assumption on $\Omega$.

To pass to unbounded coefficients, we will need the following lemma, which is in the same spirit as Lemma 6.9 in \cite{KimSak}.
\begin{lemma}\label{ToApproximate}
Let $\Omega\subseteq\bR^n$ be a bounded domain. Suppose that $b\in L^{p,q}(\Omega)$ for some $p\in(1,\infty)$ and $q\in[1,\infty)$, $d\in L^{p,\infty}(\Omega)$ with $d\geq\dive b$. Let $\psi_j(x)=j^n\psi(jx)$ be a mollifier, and define
\[
b_j=(b\chi_{\Omega})*\psi_j,\quad d_j=(d\chi_{\Omega})*\psi_j.
\]
Then $d_j$ and $b_j$ are Lipschitz continuous in $\Omega$, and also
\[
\|b_j\|_{L^{n,q}(\Omega)}\leq \|b\|_{L^{n,q}(\Omega)},\quad \|d_j\|_{L^{\frac{n}{2},\infty}(\Omega)}\leq\|d\|_{L^{\frac{n}{2},\infty}(\Omega)}.
\]
Moreover, $b_j\to b$ in $L^{n,q}(\Omega)$, and $d_j\chi_{\Omega_m}\to d$ weakly-* in $L^{\frac{n}{2},\infty}(\Omega)$. In addition, if we set $\Omega_j=\left\{x\in\Omega:\dist(x,\partial\Omega)>\frac{1}{j}\right\}$, then $d_j\geq\dive b_j$ in $\Omega_j$. Finally, if $d\in L^{\frac{n}{2},1}(\Omega)$, then $d_j\to d$ in $L^{\frac{n}{2},1}(\Omega)$.
\end{lemma}
\begin{proof}
First, if $x,y\in\Omega$, then
\begin{align*}
|d_j(x)-d_j(y)|&=\left|\int_{B_{1/j}(x)}d(z)(\psi_j(x-z)-\psi_j(y-z))\,dz\right|\leq\|\nabla \psi_j\|_{\infty}|x-y|\int_{B_{1/k}(x)}|d|\\
&\leq \|\nabla \psi_j\|_{\infty}\|d\|_{\frac{n}{2},\infty}\|\chi_{B_{1/j}(x)}\|_{\frac{n}{n-2},1}|x-y|\leq C_j|x-y|,
\end{align*}
where $C_j$ depends on $\psi_j$ and $\|d\|_{\frac{n}{2},\infty}$. Therefore $d_j\in\Lip(\Omega)$. Similarly, $b_j\in\Lip(\Omega)$.

We now use part (i) of Theorem V4 in \cite{CosteaThesis}, to obtain that
\[
\|b_j\|_{L^{n,q}(\Omega)}\leq \|b\|_{L^{n,q}(\Omega)},\quad \|d_j\|_{L^{\frac{n}{2},\infty}(\Omega)}\leq \|d\|_{L^{\frac{n}{2},\infty}(\Omega)}.
\]
Note now that, with the terminology of \cite{CosteaThesis} (or, Definition I-3.1 in \cite{BennettSharpley}) and Lemma~\ref{NormOnDisjoint}, every $b\in L^{n,q}(\Omega)$ has absolutely continuous norm. Hence, from part (ii) of Theorem V4 in \cite{CosteaThesis}, $b_j\to b\chi_{\Omega}$ in $L^{n,q}(\bR^n)$, hence $b_j\to b$ in $L^{n,q}(\Omega)$. Similarly, if $d\in L^{\frac{n}{2},1}(\Omega)$, then $d_j\to d$ in $L^{\frac{n}{2},1}(\Omega)$.

Now, since $\Omega$ is bounded and $d_m$ is a mollification of $d\chi_{\Omega}$, we obtain that $d_m\to d$ in $L^{4/3}(\Omega)$. Hence, for any $\phi\in C_c^{\infty}(\Omega)$,
\[
\int_{\Omega}d_m\phi\xrightarrow[m\to\infty]{}\int_{\Omega}d\phi.
\]
Since $(d_m)$ is bounded in $L^{\frac{n}{2},\infty}(\Omega)$, $L^{\frac{n}{n-2},1}(\Omega)$ is the predual of $L^{\frac{n}{2},\infty}(\Omega)$ and $C_c^{\infty}(\Omega)$ is dense in $L^{\frac{n}{n-2},1}(\Omega)$ (from Theorem 1.4.13 in \cite{Grafakos}) , we obtain that $d_m\to d$, weakly-* in $L^{\frac{n}{2},\infty}(\Omega)$.

Finally, to show that $d_m\geq\dive b_m$ in $\Omega_m$ in the sense of distributions, we follow the same argument as in the proof of Lemma 6.9 in \cite{KimSak}. 
\end{proof}

We will now drop the assumption that the lower order coefficients are Lipschitz to construct Green's functions. In order to obtain the symmetry relation $G(x,y)=g(y,x)$ for almost every $(x,y)\in\Omega^2$, we will have to consider convergent subsequences for functions defined in the product space $\Omega^2$. For this reason, we construct $G$ and $g$ concurrently in the next lemma.

\begin{lemma}\label{DropBound}
Let $\Omega\subseteq\bR^n$ be a bounded domain. Let $A$ be uniformly elliptic and bounded in $\Omega$, with ellipticity $\lambda$, and suppose that $b,c\in L^{n,q}(\Omega)$ for some $q\in[1,\infty)$, $d\in L^{\frac{n}{2},\infty}(\Omega)$, with $b-c\in L^{n,1}(\Omega)$ and $d\geq\dive b$. Then, for every $x,y\in\Omega$ there exist nonnegative functions $G_y(z)=G(z,y)$, $g_x(z)=g(z,x)$, where $G,g$ are measurable in $\Omega^2$, with $G_y,g_x\in L^1_{\loc}(\Omega)$ such that, for all $f\in L^{\infty}(\Omega)$, if we define
\[
v^f(y)=\int_{\Omega}G(z,y)f(z)\,dz,\quad u^f(x)=\int_{\Omega}g(z,x)f(z)\,dz,\quad x,y\in\Omega,
\]
then $u,v\in W_0^{1,2}(\Omega)$, and they solve the equations
\[
-\dive(A^t\nabla v^f+cv^f)+b\nabla v^f+dv^f=f,\quad -\dive(A\nabla u^f+bu^f)+c\nabla u^f+du^f=f
\]
in $\Omega$. Moreover, for any $x,y$,
\begin{equation}\label{eq:WeakBounded}
\|G_y\|_{L^{\frac{n}{n-2},\infty}(\Omega)}+\|\nabla G_y\|_{L^{\frac{n}{n-1},\infty}(\Omega)}+\|g_x\|_{L^{\frac{n}{n-2},\infty}(\Omega)}+\|\nabla g_x\|_{L^{\frac{n}{n-1},\infty}(\Omega)}\leq C,
\end{equation}
where $C$ depends on $n,\lambda$ and $\|b-c\|_{n,1}$ only. In addition, if $x,y$ are fixed, then
\begin{equation}\label{eq:BoundedForBounded}
\begin{gathered}
G_y(z)\leq C'|z-y|^{2-n}\,\,\,\text{for almost every}\,\,\,z\in\Omega,\,\,\,\text{and}\\
g_x(w)\leq C'|w-x|^{2-n},\,\,\,\text{for almost every}\,\,\,w\in\Omega,
\end{gathered}
\end{equation}
where $C'$ depends on $n,\lambda,\|A\|_{\infty}$ and $\|b-c\|_{n,1}$ only. Furthermore, for any $\delta(x)>\e_2>\e_1>0$, $\delta(y)>\e_2'>\e_1'>0$ and any two functions $\psi,\psi'\in C^{\infty}(\bR^n)$ with $\psi\equiv 0$ in $B_{\e_1}(y)$, $\psi'\equiv 0$ in $B_{\e_1}(x)$, and $\psi\equiv 1$ outside $B_{\e_2}(y)$, $\psi'\equiv 1$ outside $B_{\e_2'}(x)$, we have that $G_y\psi,g_x\psi'\in W_0^{1,2}(\Omega)$ and, for any $r>0$,
\[
\int_{\Omega\setminus B_r(y)}|\nabla G_y|^2+\int_{\Omega\setminus B_r(x)}|\nabla g_x|^2\leq Cr^{2-n},
\]
where $C$ depends on $n,\lambda,\|A\|_{\infty}$ and $\|b-c\|_{n,1}$ only. Moreover, the functions $G(z,w)$ and $g(w,z)$, for $z,w\in\Omega^2$ are measurable in $\Omega^2$, and
\begin{equation}\label{eq:SymmetryBounded}
G(z,w)=g(w,z)\,\,\,\text{for almost every}\,\,\,(z,w)\in\Omega^2.
\end{equation}
Finally, in the case where $d\in L^{\frac{n}{2},1}(\Omega)$, we have that, for every $y\in\Omega$ and every $\phi\in C_c^{\infty}(\Omega)$,
\begin{equation}\label{eq:DiracForBounded}
\int_{\Omega}A\nabla G_y\cdot\phi+b\nabla\phi\cdot G_y+c\nabla G_y\cdot\phi+dG_y\phi=\phi(y),
\end{equation}
and also, for every $x\in\Omega$ and every $\phi\in C_c^{\infty}(\Omega)$,
\begin{equation}\label{eq:DiracForBoundedt}
\int_{\Omega}A^t\nabla g_x\cdot\phi+c\nabla\phi\cdot g_x+b\nabla g_x\cdot\phi+dg_x\phi=\phi(x).
\end{equation}
\end{lemma}
\begin{proof}
Fix $x_0\in\Omega$, and let $\psi_j(x)=j^n\psi(jx)$ be a mollifier. For any $j\in\mathbb N$ we define the mollifications $b_j=(b\chi_{\Omega})*\psi_j$, $c_j=(c\chi_{\Omega})*\psi_j$, $d_j=(d\chi_{\Omega})*\psi_j$, and we also let $\tilde{\Omega}_j$ to be the connected component of $\left\{x\in\Omega:\dist{x,\partial\Omega}>\frac{1}{j}\right\}$ that contains $x_0$. From Lemma~\ref{ToApproximate}, $b_j,c_j$ and $d_j$ are Lipschitz continuous in $\tilde{\Omega}_j$ and $d_j\geq\dive b_j$ in $\tilde{\Omega}_j$ in the sense of distributions.

Let $y\in\Omega$, and assume that $j$ is large enough, so that $y\in\tilde{\Omega}_j$. From \eqref{eq:GPrelim}, there exists $G_y^j$ defined in $\tilde{\Omega}_j$, continuous in $\tilde{\Omega}_j\setminus\{y\}$, such that
\begin{equation}\label{eq:GyjBounds}
\|G_y^j\|_{L^{\frac{n}{n-2},\infty}(\tilde{\Omega}_j)}+\|\nabla G_y^j\|_{L^{\frac{n}{n-1},\infty}(\tilde{\Omega}_j)}\leq C,\quad G_y^j(z)\leq C'|z-y|^{2-n},\quad\|\nabla G_y^j\|_{L^2(\tilde{\Omega}_j\setminus B_r)}^2\leq Cr^{2-n},
\end{equation}
where the second to last estimate holds for all $z\in\tilde{\Omega}_j$ with $z\neq y$, and where $C$ depends on $n,\lambda$ and $\|b_j-c_j\|_{n,1}$; hence, from Lemma~\ref{ToApproximate}, $C$ depends on $n,\lambda$ and $\|b-c\|_{n,1}$. Also, we obtain that $C'$ depends on $n,\lambda,\|A\|_{\infty}$ and $\|b-c\|_{n,1}$. Moreover, from \eqref{eq:FormulaAllBounded}, for any $f\in L^{\infty}(\Omega)$, the function
\begin{equation}\label{eq:vjDfn}
v_j^f(y)=\int_{\tilde{\Omega}_j}G_y^j(z)f(z)\,dz,\,\,\,y\in\tilde{\Omega_j}
\end{equation}
is the unique $W_0^{1,2}(\tilde{\Omega_j})$ solution to $-\dive(A^t\nabla v_j^f+c_jv_j^f)+b_j\nabla v_j^f+d_jv_j^f=f$ in $\tilde{\Omega}_j$.

For the adjoint equation, let $x\in\Omega$. From \eqref{eq:gPrelim}, for $j$ large enough, there exists $g_x^j$ defined in $\tilde{\Omega}_j$ which is continuous in $\tilde{\Omega}_j\setminus\{x\}$, such that
\begin{equation}\label{eq:gxjBounds}
\|g_x^j\|_{L^{\frac{n}{n-2},\infty}(\tilde{\Omega}_j)}+\|\nabla g_x^j\|_{L^{\frac{n}{n-1},\infty}(\tilde{\Omega}_j)}\leq C,\quad g_x^j(z)\leq C'|z-x|^{2-n},\quad\|\nabla g_x^j\|_{L^2(\tilde{\Omega}_j\setminus B_r)}^2\leq Cr^{2-n},
\end{equation}
where the second to last estimate holds for all $z\in\tilde{\Omega}_j$ with $z\neq x$, and where $C$ depends on $n,\lambda$ and $\|b-c\|_{n,1}$, and $C'$ depends on $n,\lambda,\|A\|_{\infty}$ and $\|b-c\|_{n,1}$. Moreover, from \eqref{eq:FormulaAllBoundedt}, for any $f\in L^{\infty}(\Omega)$, the function
\begin{equation}\label{eq:ujDfn}
u_j^f(x)=\int_{\tilde{\Omega}_j}g_x^j(y)f(y)\,dy,\,\,\,x\in\tilde{\Omega_j}
\end{equation}
is the unique $W_0^{1,2}(\tilde{\Omega_j})$ solution to $-\dive(A\nabla u_j^f+b_ju_j^f)+c_j\nabla u_j^f+d_ju_j^f=f$ in $\tilde{\Omega}_j$.

We now set
\[
g^j(z,x)=g_x^j(z),\quad G^j(w,y)=G_y^j(w),
\]
for $x,y,z,w\in\tilde{\Omega}_j$; those functions are well defined for any $z\neq x$ and $w\neq y$, from \eqref{eq:JointCont}.

From \eqref{eq:JointBoundedness}, for any $\mathcal{U}\subseteq\Omega^2$ with smooth boundary, $\|G^j\|_{W^{1,q_0}(\mathcal{U})}\leq C$, where $C$ depends on $n,\lambda,\|b-c\|_{n,1}$ and $\mathcal{U}$. Since the embedding $W^{1,q_0}(\mathcal{U})\hookrightarrow L^{q_0}(\mathcal{U})$ is compact, there exists $G_{\mathcal{U}}\in W^{1,q_0}(\mathcal{U})$ such that, for a subsequence $(G^{j_i})$ of $(G^j)$, $G^{j_i}\to G_{\mathcal{U}}$ weakly in $W^{1,q_0}(\mathcal{U})$, strongly in $L^{q_0}(\mathcal{U})$ and almost everywhere in $\mathcal{U}$. Using a diagonalization argument, there exists $G\in W^{1,q_0}_{\loc}(\Omega^2)$ such that, for a subsequence $(G^{j_i})$ of $(G^j)$, $G^{j_i}\to G_{\mathcal{U}}$ weakly in $W^{1,q_0}_{\loc}(\Omega^2)$, strongly in $L^{q_0}_{\loc}(\Omega^2)$ and almost everywhere in $\Omega^2$.

Now, with an argument similar to the above, we have that there exists $g\in W^{1,q_0}_{\loc}(\Omega^2)$ such that, for a subsequence $(g^{j_i^1})$ of $(g^{j_i})$, $g^{j_i^1}\to g$ weakly in $W^{1,q_0}_{\loc}(\Omega^2)$, strongly in $L^{q_0}_{\loc}(\Omega^2)$ and almost everywhere in $\Omega^2$. From \eqref{eq:Symmetry}, $G^j(x,y)=g^j(y,x)$ for every $x,y\in\Omega$ with $x\neq y$ from \eqref{eq:Symmetry}. Since $g^j\to\tilde{g}$ and $G^j\to\tilde{G}$ almost everywhere in $\Omega^2$, there exists $\mathcal{F}\subseteq\Omega^2$, with full measure in $\Omega^2$, such that
\begin{equation}\label{eq:gtildeGtilde}
g(y,x)=G(x,y)\,\,\,\text{for every}\,\,\,(x,y)\in\mathcal{F}.
\end{equation}
Fix now $y\in\Omega$. Since $L^{\frac{n}{n-2},\infty}$ and $L^{\frac{n}{n-1},\infty}$ are the dual spaces of $L^{\frac{n}{2},1}$ and $L^{n,1}$, respectively, \eqref{eq:GyjBounds} and the Banach-Alaoglou theorem imply that there exists a subsequence $(G_y^{j_i^2})$ of $(G_y^{j_i^1})$ (which depends on $y$) and $G_y\in W_0^{1,q}(\Omega)$, such that
\begin{equation}\label{eq:WeakConv,n1}
\begin{gathered}
G_y^{j_i^2}\rightharpoonup G_y\quad\text{weakly}^*\,\,\text{in}\,\,\,L^{\frac{n}{n-2},\infty}(\Omega),\quad\nabla G_y^{j_i^2}\rightharpoonup \nabla G_y\quad\text{weakly}^*\,\,\text{in}\,\,\,L^{\frac{n}{n-1},\infty}(\Omega),\\
G_y^{j_i^2}\rightharpoonup g_x\quad\text{weakly}^*\,\,\text{in}\,\,\,W_0^{1,q}(\Omega),\quad G_y^{j_i^2}\to G_y\quad\text{almost everywhere}.
\end{gathered}
\end{equation}
If $F_y\subseteq\Omega$ is the set of $z$ for which $(G_y^{j_i^2}(z))$ converges, then $F_y$ is measurable and has full measure in $\Omega$. Then, we define 
\begin{equation}\label{eq:Fy}
G_y(z)=\lim_{i\to\infty}G_y^{j_i^2}(z),
\end{equation}
so that $G_y$ is defined for every $z\in F_y$. Also, combining with \eqref{eq:GyjBounds}, we obtain that
\[
\|G_y\|_{L^{\frac{n}{n-2},\infty}(\Omega)}+\|\nabla G_y\|_{L^{\frac{n}{n-1},\infty}(\Omega)}\leq C,\quad G_y(z)\leq C'|z-x|^{2-n},\quad\|\nabla G_y\|_{L^2(\Omega\setminus B_r)}^2\leq C'r^{2-n},
\]
where the second to last estimate holds for almost every $z\in\Omega$, $C$ depends on $n,\lambda$ and $\|b-c\|_{n,1}$, and $C'$ depends on $n,\lambda,\|A\|_{\infty}$ and $\|b-c\|_{n,1}$. In addition, if $\psi$ is as in the statement of the lemma, using \eqref{eq:GlobalL2AllBounded} for $G_y^j$ we obtain that $G_y\psi\in W_0^{1,2}(\Omega)$. Note also that, if $(x,y)\in\mathcal{F}$, then $(G^{j_i^2}(x,y))$ converges; hence $x\in F_y$, and, from \eqref{eq:Fy},
\begin{equation}\label{eq:gEqualsTildeg}
G_y(x)=\lim_{i\to\infty}G_y^{j_i^2}(x)=\lim_{i\to\infty}G^{j_i^2}(x,y)=G(x,y).
\end{equation}
We now fix $x\in\Omega$. Then, using \eqref{eq:gxjBounds} and proceeding as above, there exists a subsequence $(g_x^{j_i^3})$ of $(g_x^{j_i^2})$ (which depends on $x$) and $g_x\in W_0^{1,q_0}(\Omega)$, such that
\begin{equation}\label{eq:WeakConv,n1g}
\begin{gathered}
g_x^{j_i^3}\rightharpoonup g_x\quad\text{weakly}^*\,\,\text{in}\,\,\,L^{\frac{n}{n-2},\infty}(\Omega),\quad\nabla g_x^{j_i^3}\rightharpoonup \nabla g_x\quad\text{weakly}^*\,\,\text{in}\,\,\,L^{\frac{n}{n-1},\infty}(\Omega),\\
g_x^{j_i^3}\rightharpoonup g_x\quad\text{weakly}^*\,\,\text{in}\,\,\,W_0^{1,q_0}(\Omega),\quad g_x^{j_i^3}\to g_x\quad\text{almost everywhere}.
\end{gathered}
\end{equation}
If $F^x\subseteq\Omega$ is the set of $z$ for which $(g_x^{j_i^3}(z))$ converges, then $F^x$ is measurable and has full measure in $\Omega$. Then, we define
\begin{equation}\label{eq:Fx}
g_x(z)=\lim_{i\to\infty}g_x^{j_i^3}(z),
\end{equation}
so that $g_x(z)$ is defined for every $z\in F^x$. Then, from \eqref{eq:gxjBounds}, we obtain that
\[
\|g_x\|_{L^{\frac{n}{n-2},\infty}(\Omega)}+\|\nabla g_x\|_{L^{\frac{n}{n-1},\infty}(\Omega)}\leq C,\quad g_x(z)\leq C'|z-y|^{2-n},\quad\|\nabla g_x\|_{L^2(\Omega\setminus B_r)}^2\leq C'r^{2-n},
\]
where the second to last estimate holds for almost every $z\in\Omega$, $C$ depends on $n,\lambda$ and $\|b-c\|_{n,1}$, and $C'$ depends on $n,\lambda,\|A\|_{\infty}$ and $\|b-c\|_{n,1}$. In addition, if $\psi'$ is as in the statement of the lemma using \eqref{eq:GlobalL2AllBoundedt} for $g_x^j$, we obtain that $g_x\psi'\in W_0^{1,2}(\Omega)$.

Note now that, if $(x,y)\in\mathcal{F}$, then $(g^{j_i^3}(y,x))$ converges; hence $y\in F^x$, and, from \eqref{eq:Fx},
\begin{equation}\label{eq:GEqualsTildeG}
g_x(y)=\lim_{i\to\infty}g_x^{j_i^3}(y)=\lim_{i\to\infty}g^{j_i^3}(y,x)=g(y,x),
\end{equation}
Combining \eqref{eq:gEqualsTildeg} and \eqref{eq:GEqualsTildeG} with \eqref{eq:gtildeGtilde}, we obtain that
\[
G_x(y)=G(x,y)=g(y,x)=g_x(y)\,\,\,\text{for almost every}\,\,\,(x,y)\in\Omega^2.
\]

We now note that, from \eqref{eq:vjDfn} and \eqref{eq:WeakConv,n1}, for almost every $y\in\tilde{\Omega}_j$,
\[
|v_j^f(y)|\leq C\|G_y^j\|_{L^1(\tilde{\Omega}_j)}\|f\|_{\infty}\leq C\|G_y^j\|_{L^{\frac{n}{n-2},\infty}(\Omega)}\|f\|_{\infty}|\Omega|^{\frac{2}{n}}\leq\tilde{C},
\]
where $\tilde{C}$ depends on $n,\lambda$, $\|b-c\|_{n,1}$, $\|f\|_{\infty}$ and $|\Omega|$. Hence, using the last estimate and Lemma~\ref{DerivativeByU} (choosing $\phi\equiv 1$),
\[
\int_{\tilde{\Omega}_j}|\nabla v_j^f|^2\leq C\|f\|_{2_*}^2+C\|v_j^f\|_{2^*}^2\leq\tilde{C}.
\]
The last two estimates show that, extending $v_j^f$ by $0$ in $\Omega\setminus\tilde{\Omega}_j$, $(v_j^f)$ is bounded in $W_0^{1,2}(\Omega)$. Then, for a subsequence $(v_{j_i^4}^f)$ of $(v_{j_i^3}^f)$, we obtain that $(v_{j_i^4}^f)$ converges to a function $v^f_0\in W_0^{1,2}(\Omega)$ weakly in $W_0^{1,2}(\Omega)$.

Let now $\phi\in C_c^{\infty}(\Omega)$, and let $U$ be an open and bounded set that contains the support of $\phi$ (we could use $U=\Omega$, but the following argument will also be useful in the proof of Theorem~\ref{Green}). We then have that a further subsequence $(v_{j_i^5}^f)$ converges to $v^f_0$ strongly in $L^{\frac{3n}{2n-4}}(U)$ (since $1<\frac{3n}{2n-4}<2^*$ and $U$ is bounded), and almost everywhere in $U$. Moreover, for any $\phi\in C_c^{\infty}(U)$, if $j_i^5$ is large enough so that the support of $\phi$ is contained in $\tilde{\Omega}_{j_i^4}$, we obtain that
\begin{equation}\label{eq:ToPlugWeak}
\int_UA^t\nabla v_{j_i^5}^f\nabla\phi+c_{j_i^5}\nabla\phi\cdot v_{j_i^5}^f+b_{j_i^5}\nabla v_{j_i^5}^f\cdot\phi+d_{j_i^5}v_{j_i^5}^f\phi=\int_Uf\phi.
\end{equation}
Since $b_j\to b$ and $c_j\to c$ in $L^{n,q}(\Omega)$ from Lemma~\ref{ToApproximate}, and also $\nabla v_{j_i^5}^f\to \nabla v^f_0$ weakly in $L^2(\Omega)$ and $v_{j_i^5}^f\to v^f_0$ in $L^2(U)$, we obtain that 
\begin{equation}\label{eq:Weak1}
\int_UA^t\nabla v_{j_i^5}^f\nabla\phi+c_{j_i^5}\nabla\phi\cdot v_{j_i^5}^f+b_{j_i^5}\nabla v_{j_i^5}^f\cdot\phi\xrightarrow[i\to\infty]{}\int_UA\nabla v^f_0\nabla\phi+c\nabla\phi\cdot v^f_0+b\nabla v^f_0\cdot\phi.
\end{equation}
Moreover, $d_j\to d$ weakly-* in $L^{\frac{n}{2},\infty}(\Omega)$ from Lemma~\ref{ToApproximate}, and $v_{j_i^5}^f\to v^f_0$ strongly in $L^{\frac{3n}{2n-4}}(U)$. Since $\frac{3n}{2n-4}>\frac{n}{n-2}$ and $U$ is bounded, we obtain that $v_{j_i^5}^f\to v^f_0$ strongly in $L^{\frac{n}{n-2},1}(U)$, which is the predual of $L^{\frac{n}{2},\infty}(U)$. Therefore,
\begin{equation}\label{eq:Weak2}
\int_Ud_{j_i^5}v_{j_i^5}^f\phi\xrightarrow[i\to\infty]{}\int_Udv^f_0\phi.
\end{equation}
Hence, plugging \eqref{eq:Weak1} and \eqref{eq:Weak2} to \eqref{eq:ToPlugWeak}, we obtain that $v^f_0$ is a $W_0^{1,2}(\Omega)$ solution to the equation $-\dive(A^t\nabla v^f+cv^f)+b\nabla v^f+dv^f=f$ in $\Omega$. Then, letting $j\to\infty$ in \eqref{eq:vjDfn} and using \eqref{eq:WeakConv,n1} and that $v_{j_i^5}^f\rightharpoonup v^f_0$ weakly in $W_0^{1,2}(\Omega)$, we obtain that $\displaystyle v^f_0(y)=v^f(y)=\int_{\Omega}G(z,y)f(z)\,dz$ for almost every $y\in\Omega$.

Using \eqref{eq:ujDfn} and \eqref{eq:WeakConv,n1g}, an argument similar to the above shows that $u^f$ is a $W_0^{1,2}(\Omega)$ solution to the equation $-\dive(A\nabla u^f+bu^f)+c\nabla u^f+du^f=f$ in $\Omega$.

It only remains to show \eqref{eq:DiracForBounded} and \eqref{eq:DiracForBoundedt}. For this, we first note that the definition of $G_y^j$ and \eqref{eq:DiracForBoundedBounded} show that, for any $y\in\Omega$ and any $\phi\in C_c^{\infty}(\Omega)$,
\begin{equation}\label{eq:ji}
\int_{\Omega}A\nabla G_y^{j_i^3}\nabla\phi+b_{j_i^3}\nabla\phi\cdot G_y^{j_i^3}+c_{j_i^3}\nabla G_y^{j_i^3}\cdot\phi+d_{j_i^3}G_y^{j_i^3}\phi=\phi(y)
\end{equation}
From Lemma~\ref{ToApproximate}, $b_j\to b$ and $c_j\to c$ strongly in $L^{n,1}(\Omega)$, and $d_j\to d$ strongly in $L^{\frac{n}{2},1}(\Omega)$. Hence, letting $i\to\infty$ in \eqref{eq:ji} and using \eqref{eq:WeakConv,n1}, we obtain \eqref{eq:DiracForBounded}.

Using \eqref{eq:DiracForBoundedBoundedt}, the proof of \eqref{eq:DiracForBoundedt} is similar, and this completes the proof.
\end{proof}

We now drop the boundedness assumption on $\Omega$, and we construct Green's function in arbitrary domains.

\begin{thm}\label{Green}
Let $\Omega\subseteq\bR^n$ be a domain. Let $A$ be uniformly elliptic and bounded in $\Omega$, with ellipticity $\lambda$, and suppose that $b,c\in L^{n,q}(\Omega)$ for some $q\in[1,\infty)$, $d\in L^{\frac{n}{2},\infty}(\Omega)$, with $b-c\in L^{n,1}(\Omega)$ and $d\geq\dive b$. Then, for every $x,y\in\Omega$ there exist nonnegative functions $G_y(z)=G(z,y)$, $g_x(z)=g(z,x)$, where $G,g$ are measurable in $\Omega^2$, with $G_y,g_x\in L^1_{\loc}(\Omega)$ such that, for all $f\in L_c^{\infty}(\Omega)$, if we define
\[
v^f(y)=\int_{\Omega}G(z,y)f(z)\,dz,\quad u^f(x)=\int_{\Omega}g(z,x)f(z)\,dz,\quad x,y\in\Omega,
\]
then $u,v\in Y_0^{1,2}(\Omega)$, and they are solutions to the equations
\[
-\dive(A^t\nabla v^f+cv^f)+b\nabla v^f+dv^f=f,\quad -\dive(A\nabla u^f+bu^f)+c\nabla u^f+du^f=f
\]
in $\Omega$. Moreover, for any $x,y$,
\[
\|G_y\|_{L^{\frac{n}{n-2},\infty}(\Omega)}+\|\nabla G_y\|_{L^{\frac{n}{n-1},\infty}(\Omega)}+\|g_x\|_{L^{\frac{n}{n-2},\infty}(\Omega)}+\|\nabla g_x\|_{L^{\frac{n}{n-1},\infty}(\Omega)}\leq C,
\]
where $C$ depends on $n,\lambda$ and $\|b-c\|_{n,1}$ only. In addition, if $x,y$ are fixed, then
\[
\begin{gathered}
G_y(z)\leq C'|z-y|^{2-n}\,\,\,\text{for almost every}\,\,\,z\in\Omega,\,\,\,\text{and}\\
g_x(w)\leq C'|w-x|^{2-n},\,\,\,\text{for almost every}\,\,\,w\in\Omega,
\end{gathered}
\]
where $C'$ depends on $n,\lambda,\|A\|_{\infty}$ and $\|b-c\|_{n,1}$ only. Furthermore, for any $\delta(x)>\e_2>\e_1>0$, $\delta(y)>\e_2'>\e_1'>0$ and any two functions $\psi,\psi'\in C^{\infty}(\bR^n)$ with $\psi\equiv 0$ in $B_{\e_1}(y)$, $\psi'\equiv 0$ in $B_{\e_1}(x)$, and $\psi\equiv 1$ outside $B_{\e_2}(y)$, $\psi'\equiv 1$ outside $B_{\e_2'}(x)$, we have that $G_y\psi,g_x\psi'\in W_0^{1,2}(\Omega)$ and, for any $r>0$,
\[
\int_{\Omega\setminus B_r(y)}|\nabla G_y|^2+\int_{\Omega\setminus B_r(x)}|\nabla g_x|^2\leq Cr^{2-n},
\]
where $C$ depends on $n,\lambda,\|A\|_{\infty}$ and $\|b-c\|_{n,1}$ only. Moreover, the functions $G(z,w)$ and $g(w,z)$, for $z,w\in\Omega^2$ are measurable in $\Omega^2$, and
\[
G(z,w)=g(w,z)\,\,\,\text{for almost every}\,\,\,(z,w)\in\Omega^2.
\]
Finally, in the case where $d\in L^{\frac{n}{2},1}(\Omega)$, we have that, for every $y\in\Omega$ and every $\phi\in C_c^{\infty}(\Omega)$,
\begin{equation}\label{eq:Dirac1}
\int_{\Omega}A\nabla G_y\cdot\phi+b\nabla\phi\cdot G_y+c\nabla G_y\cdot\phi+dG_y\phi=\phi(y),
\end{equation}
and also, for every $x\in\Omega$ and every $\phi\in C_c^{\infty}(\Omega)$,
\begin{equation}\label{eq:Dirac2}
\int_{\Omega}A^t\nabla g_x\cdot\phi+c\nabla\phi\cdot g_x+b\nabla g_x\cdot\phi+dg_x\phi=\phi(x).
\end{equation}
\end{thm}
\begin{proof}
Fix $x_0\in\Omega$, and for $j\in\mathbb N$, let $\Omega_j$ to be the connected component of $\Omega\cap B_j(x_0)$. Let now $G^j,g^j$ be Green's functions for the operators $\mathcal{L}u=-\dive(A\nabla u+bu)+c\nabla u+du$ and $\mathcal{L}^tu=-\dive(A^t\nabla u+cu)+b\nabla u+du$ in $\Omega_j$, in the sense of Lemma 5.2. Then $G^j$ and $g^j$ are measurable in $\Omega_j^2$, and, from Lemma~\ref{DropBound}, $G^j(x,y)=g^j(y,x)$ for almost every $(x,y)\in\Omega_j$. Extending $G^j,g^j$ by $0$ in $\Omega^2\setminus\Omega_j^2$ and using \eqref{eq:SymmetryBounded}, we obtain that $G^j(x,y)=g^j(y,x)$ for every $(x,y)\in\mathcal{F}_j$, where $\mathcal{F}_j\subseteq\Omega^2$ has full measure. Also, for any $x,y\in\Omega$, \eqref{eq:WeakBounded} shows that
\begin{equation}\label{eq:ToDropj}
\|G_y^j\|_{L^{\frac{n}{n-2},\infty}(\Omega)}+\|\nabla G_y^j\|_{L^{\frac{n}{n-1},\infty}(\Omega)}+\|g_x^j\|_{L^{\frac{n}{n-2},\infty}(\Omega)}+\|\nabla g_x^j\|_{L^{\frac{n}{n-1},\infty}(\Omega)}\leq C,
\end{equation}
where $C$ depends on $n,\lambda$ and $\|b-c\|_{n,1}$.

Fix now $\mathcal{U}\subseteq\Omega^2$ with smooth boundary. As in the proof of Lemma~\ref{DropBound}, there exist $G_{\mathcal{U}},g_{\mathcal{U}}\in W^{1,q_0}(\mathcal{U})$ such that, for subsequences $(G^{j_i}),(g^{j_i})$, $G^{j_i}\to G_{\mathcal{U}}$ and $g^{j_i}\to g_{\mathcal{U}}$ almost everywhere in $\mathcal{U}$. Using a diagonalization argument, we obtain that there exist measurable functions $G,g$ defined in $\Omega^2$ and a set $\mathcal{F}\subseteq\Omega^2$ with full measure, such that, for subsequences $(G^{j_i^1}),(g^{j_i^1})$,
\[
G^{j_i^1}(x,y)\to G(x,y),\quad g^{j_i^1}(y,x)\to g(y,x),\quad\text{for every}\,\,\,(x,y)\in\mathcal{F}.
\]
Setting $\mathcal{F}_0=\mathcal{F}\cap\bigcap_{j=1}^{\infty}\mathcal{F}_j$, we obtain that $\mathcal{F}_0\subseteq\Omega^2$ has full measure, and for all $(x,y)\in\mathcal{F}_0$,
\begin{equation}\label{eq:ToEqualTildes}
G(x,y)=\lim_{i\to\infty}G^{j_i^1}(x,y)=\lim_{i\to\infty}g^{j_i^1}(y,x)=g(y,x).
\end{equation}
We now fix $x,y\in\Omega$. Using \eqref{eq:ToDropj}, we obtain that there exist subsequences $(G_y^{j_i^2}), (g_x^{j_i^2})$ (depending on $y,x$ respectively) and functions $G_y,g_x$ defined in $\Omega$, such that
\begin{equation*}
\begin{gathered}
G_y^{j_i^2}\rightharpoonup G_y\quad\text{weakly}^*\,\,\text{in}\,\,\,L^{\frac{n}{n-2},\infty}(\Omega),\quad\nabla G_y^{j_i^2}\rightharpoonup \nabla G_y\quad\text{weakly}^*\,\,\text{in}\,\,\,L^{\frac{n}{n-1},\infty}(\Omega),\\
G_y^{j_i^2}\to G_y\quad\text{almost everywhere},
\end{gathered}
\end{equation*}
and also
\begin{equation*}
\begin{gathered}
g_x^{j_i^2}\rightharpoonup g_y\quad\text{weakly}^*\,\,\text{in}\,\,\,L^{\frac{n}{n-2},\infty}(\Omega),\quad\nabla g_x^{j_i^2}\rightharpoonup \nabla g_x\quad\text{weakly}^*\,\,\text{in}\,\,\,L^{\frac{n}{n-1},\infty}(\Omega),\\
g_x^{j_i^2}\to g_x\quad\text{almost everywhere}.
\end{gathered}
\end{equation*}
If $F_y,F^x$ are the sets in which $(G_y^{j_i^2})$, $(g_x^{j_i^2})$ converge pointwise, we explicitly define
\begin{equation}\label{eq:EqualTildes}
G_y(z)=\lim_{i\to\infty}G_y^{j_i^2}(z)\,\,\,\text{for every}\,\,\,z\in F_y,\quad g_x(w)=\lim_{i\to\infty}g_x^{j_i^2}(w)\,\,\,\text{for every}\,\,\,w\in F^x.
\end{equation}
If now $(x,y)\in\mathcal{F}_0$, we obtain that $x\in F_y$ and $y\in F^x$. Hence, \eqref{eq:ToEqualTildes} and \eqref{eq:EqualTildes} show that
\[
G_y(x)=G(x,y)=g(y,x)=g_x(y)\,\,\,\text{for almost every}\,\,\,(x,y)\in\Omega^2.
\]
Let now $f\in L_c^{\infty}(\Omega)$, and set
\[
v_j^f(y)=\int_{\Omega_j}G_y^j(z)f(z)\,dz.
\]
From Lemma~\ref{DropBound}, $v_j^f\in W_0^{1,2}(\Omega_j)$ for every $j$, and it solves the equation $-\dive(A^t\nabla u+cu)+b\nabla u+du=f$ in $\Omega_j$. Now, from \eqref{eq:BoundedForBounded}, $|v_j^f|$ is bounded above by a constant multiple of the Riesz potential $|I_2f|$ (as on page 117 in \cite{SteinSingular}). Since the exponents $2^*,2_*$ satisfy the relation $\frac{1}{2^*}=\frac{1}{2_*}-\frac{2}{n}$, Theorem 1 on page 119 in \cite{SteinSingular} shows that 
\[
\|v_j^f\|_{L^{2^*}(\Omega)}\leq \|I_2f\|_{L^{2^*}(\bR^n)}\leq C\|f\|_{L^{2_*}(\bR^n)}\leq C'\|f\|_{L^{2_*}(\Omega)},
\]
where $C'$ depends on $n,\lambda,\|A\|_{\infty}$ and $\|b-c\|_{n,1}$. Then, using Lemma~\ref{DerivativeByU} in $\Omega_j$ (with $\phi\equiv 1$), we obtain that
\[
\int_{\Omega_j}|\nabla v_j^f|^2\leq C\|f\|_{L^{2^*}(\Omega_j)}^2+C\|u\|_{L^{2^*}(\Omega_j)}^2\leq\tilde{C},
\]
hence $(v_j^f)$ is uniformly bounded in $Y_0^{1,2}(\Omega_j)$. Extending $(v_j^f)$ by $0$ in $\Omega\setminus\Omega_j$, we obtain that $(v_j^f)$ is uniformly bounded in $Y_0^{1,2}(\Omega)$; hence, reflexivity of $Y_0^{1,2}(\Omega)$ implies that there exists a subsequence of $(v_{{j_i}^2}^f)$ that converges weakly to some $v^f_0\in Y_0^{1,2}(\Omega)$. Using an argument similar to the proof of Lemma~\ref{DropBound}, we have that
\[
v^f(y)=\int_{\Omega}G(z,y)f(z)\,dz
\]
is a $Y_0^{1,2}(\Omega)$ solution to the equation $-\dive(A^t\nabla v^f+cv^f)+b\nabla v^f+dv^f=f$ in $\Omega$. Similarly, we show that $u^f$ is a $Y_0^{1,2}(\Omega)$ solution to the equation $-\dive(A\nabla u^f+bu^f)+c\nabla u^f+du^f=f$ in $\Omega$.

The rest of the proof is similar to the proof of Lemma~\ref{DropBound}, where to show \eqref{eq:Dirac1} and \eqref{eq:Dirac2} we also use that $b\chi_{\Omega_j}\to b$ in $L^{n,1}(\Omega)$ and $d\chi_{\Omega_j}\to d$ in $L^{\frac{n}{2},1}(\Omega)$ (if $d\in L^{\frac{n}{2},1}(\Omega)$). This completes the proof.
\end{proof}

Note that the previous theorem asserts existence of solutions to the equation $\mathcal{L}u=f$ for $f\in L_c^{\infty}(\Omega)$. Using those solutions, we can show uniqueness for solutions to the adjoint equation.

\begin{prop}\label{Uniquenesst}
Under the same assumptions as in Theorem~\ref{Green}, if $v\in Y_0^{1,2}(\Omega)$ is a solution to the equation
\[
-\dive(A^t\nabla v+cv)+b\nabla v+dv=0,
\]
then $v\equiv 0$.
\end{prop}
\begin{proof}
Let $f\in L_c^{\infty}(\Omega)$. From Theorem~\ref{Green}, there exists a solution $u^f\in Y_0^{1,2}(\Omega)$ to the equation $-\dive(A\nabla u^f+bu^f)+c\nabla u^f+du^f=f$ in $\Omega$. Therefore, for all $f\in L_c^{\infty}(\Omega)$,
\[
\int_{\Omega}fv=\int_{\Omega}A\nabla u^f\nabla v+b\nabla v\cdot u^f+c\nabla u^f\cdot v+du^fv=0,
\]
which implies that $v\equiv 0$. This completes the proof.
\end{proof}

Combining Propositions~\ref{Uniqueness} and \ref{Uniquenesst}, we can now show that the solutions $u^f$ and $v^f$ in Theorem~\ref{Green} are unique.

\begin{prop}\label{UniquenessOfSolutions}
Under the same assumptions as in Theorem~\ref{Green}, for any $f\in L^{\infty}_c(\Omega)$, the functions
\[
v^f(y)=\int_{\Omega}G(z,y)f(z)\,dz,\quad u^f(x)=\int_{\Omega}g(z,x)f(z)\,dz,\quad x,y\in\Omega,
\]
are the unique $Y_0^{1,2}(\Omega)$ solutions to the equations
\[
-\dive(A\nabla u+bu)+c\nabla u+du=f,\quad -\dive(A^t\nabla v+cv)+b\nabla v+dv=f,\,\,\,\text{in}\,\,\,\Omega.
\]
\end{prop}

Finally, we show uniqueness of Green's functions.

\begin{prop}\label{GreenUniqueness}
Under the same assumptions as in Theorem~\ref{Green}, suppose that $G^*(x,y)=G^*_y(x)$ is a function such that $G^*_y\in L^1_{\loc}(\Omega)$ for almost every $y\in\Omega$, and for all $f\in L_c^{\infty}(\Omega)$, the function
\[
v_*^f(y)=\int_{\Omega}G^*(z,y)f(z)\,dz
\]
is a $Y_0^{1,2}(\Omega)$ solution to the equation $-\dive(A^t\nabla v_*^f+cv_*^f)+b\nabla v_*^f+dv_*^f=f$ in $\Omega$. If $G$ is Green's function constructed in Theorem~\ref{Green}, then for almost every $y\in\Omega$, $G^*_y=G_y$ almost everywhere in $\Omega$. The analogous statement holds for $g$.
\end{prop}
\begin{proof}	
Let $f\in L_c^{\infty}(\Omega)$ and consider the function $v^f(y)=\int_{\Omega}G(z,y)f(z)\,dz$ from Theorem~\ref{Green}. Then, from Proposition~\ref{UniquenessOfSolutions}, $v^f=v_*^f$ almost everywhere in $\Omega$. Set now $F_f$ to be the set of $y\in\Omega$ with $G_y^*\in L^1_{\loc}(\Omega)$ and $v^f(y)=v_*^f(y)$, then $F_f$ has full measure in $\Omega$. Let $(f_m)$ be an enumeration of the characteristic functions of $B_{r_j}(q_i)$, where $q_i\in\Omega\cap\mathbb Q^n$ and $r_j>0$ with $r_j\in\mathbb Q$, and set $F=\bigcap_mF_{f_m}$. Then $F$ has full measure in $\Omega$, and for all $y\in F$, $q_i\in\Omega\cap\mathbb Q^n$ and $r_j>0$ with $r_j\in\mathbb Q$,
\begin{equation}\label{eq:ABEquality}
\int_{B_{r_j}(q_i)}(G^*(z,y)-G(z,y))\,dz=0.
\end{equation}
If we fix $y\in F$, then the set of Lebesgue points of $G^*(\cdot,y)-G(\cdot,y)$ has full measure in $\Omega$, hence \eqref{eq:ABEquality} shows that $G^*(z,y)=G(z,y)$ for almost every $z\in\Omega$.

The analogous statement for $g$ is proved in the same way, and this completes the proof.
\end{proof}

\section{Counterexamples}
In this section we show that, in the setting of Lorentz spaces, the space $L^{n,1}$ is optimal in order to deduce pointwise bounds as in Theorem~\ref{Green}. This will be done using the function $c(x)$ defined in \eqref{eq:CFormula} in the Introduction, for which we first show the following lemma.

\begin{lemma}\label{bNorms}
Let $B=B_{1/e}$ be the ball with radius $1/e$, centered at $0$, and consider the function $c$ in \eqref{eq:CFormula}. Then $c\in L^{n,q}(B)$ for all $q>1$, but $c\notin L^{n,1}(B)$.
\end{lemma}
\begin{proof}
Let $C_n$ be the volume of the unit ball in $\bR^n$. It is straightforward to check that, if $B$ is a ball centered at $0$ and $f(x)=f(|x|):B\setminus\{0\}\to\mathbb R$, with $f\geq 0$ decreasing and continuous, then $f^*(s)=f\left(C_n^{-1/n}s^{1/n}\right)$, so $c^*(s)=C_n^{1/n}s^{-1/n}\left(-\ln(C_n^{-1/n}s^{1/n})\right)^{-1}$. Hence, for $q>1$, if $s=\sigma^nC_n$,
\[
\int_0^{|B|}s^{q/n-1}c^*(s)^q\,ds=\int_0^{C_ne^{-n}}\frac{C_n^{q/n}}{s(-\ln( C_n^{-1/n}s^{1/n}))^q}\,ds=\int_0^{1/e}\frac{nC_n^{q/n}}{\sigma\left(-\ln\sigma\right)^q}\,d\sigma<\infty.
\]
This shows that $c\in L^{n,q}(B)$ for any $q>1$. A similar calculation shows that $c\notin L^{n,1}(B)$.
\end{proof}

Using the previous lemma, Definition~\ref{GreenDfn} and the comment after it, we can follow the argument in Proposition 7.5 in \cite{KimSak} to obtain the next proposition for Green's function for $-\Delta u+c\nabla u=0$.
\begin{prop}\label{NoPointwiseG}
Let $B=B_{1/e}$. Let $q>1$ and $\delta>0$ and set $c_{\delta}=\delta c$, where $c$ is as in \eqref{eq:CFormula}. Then $c_{\delta}\in L^{n,q}(B)$, $\|c_{\delta}\|_{n,q}\leq \|c\|_{L^{n,q}}\delta$, and if Green's function $G_y^{\delta}(x)=G^{\delta}(x,y)$ for the operator $-\Delta u+c_{\delta}\nabla u$ exists, then it cannot belong to $L^1(B)$ uniformly in $y$. In particular, the bounds
\begin{equation*}
G^{\delta}(x,y)\leq C|x-y|^{2-n},\quad \|G^{\delta}(\cdot,y)\|_{L^{\frac{n}{n-2},\infty}(\Omega)}\leq C
\end{equation*}
for some $C>0$, for almost every $x,y\in B$, \textbf{cannot} hold.
\end{prop}

Hence, even assuming that $c$ has small $L^{n,q}$ norm for some $q>1$ does not necessarily imply the pointwise bounds for Green's function. In view of Theorem~\ref{Green}, the previous proposition shows that the consideration of $L^{n,1}$ is both necessary and optimal in order to deduce weak type bounds and pointwise bounds in the case $d\geq\dive b$: 

We now turn to the equation $-\Delta u-\dive(cu)=0$. Since $c\in L^n(B)$, Theorem 7.2 in \cite{KimSak} shows that, for any $\delta>0$, Green's function $g^{\delta}_x(y)=g^{\delta}(y,x)$ for the operator $-\Delta u-\dive(\delta cu)$ in $B$ exists, and also satisfies the bounds
\[
\|g_x^{\delta}\|_{L^{\frac{n}{n-2},\infty}(B)}+\|\nabla g_x^{\delta}\|_{L^{\frac{n}{n-1},\infty}(B)}\leq C,
\]
where $C$ depends on $n$, $r_{\delta b}\left(\frac{\lambda}{9C_n}\right)$ and $\tilde{r}_{\delta b}\left(\frac{\lambda}{9C_n}\right)$ (defined in (2.2) and (2.6) in \cite{KimSak}). However, the following counterexample shows that the pointwise bound $g^{\delta}(y,x)\leq C|y-x|^{2-n}$ fails for $g^{\delta}$.

\begin{prop}\label{NoPointwiseGt}
Let $B=B_{1/e}$. Le $q>1$ and $\delta>0$ and set $b_{\delta}=\delta c$, where $c$ is as in \eqref{eq:CFormula}. Then $c_{\delta}\in L^{n,q}(B)$, $\|c_{\delta}\|_{n,q}\leq\|c\|_{n,q}\delta$, and if $g_x^{\delta}(y)=g^{\delta}(y,x)$ is Green's function for the operator $-\Delta u-\dive(c_{\delta}u)$ in $B$, then, the bound
\[
g^{\delta}(y,x)\leq C|y-x|^{2-n}
\]
for some $C>0$, for almost every $x,y\in B$, \textbf{cannot} hold.
\end{prop}
\begin{proof}
We will show a stronger fact: for any $0<\e_1<\e_2<\e_3<e^{-1}$, the inequality
\[
g^{\delta}(y,x)\leq C\,\,\,\text{for almost every}\,\,\,y\in B_{\e_1},\,x\in B_{\e_3}\setminus B_{\e_2},
\]
\textbf{cannot} hold, for any $C>0$. To show this, let $\e\in(0,e^{-1})$, and set
\[
G_{\delta,\e}(\rho)=\left\{\begin{array}{c l}\displaystyle\int_0^{\rho}\sigma^{n-1}(-\ln\sigma)^{\delta}\,d\sigma,& 0\leq\rho\leq\e \\ \displaystyle\int_0^{\e}\sigma^{n-1}(-\ln\sigma)^{\delta}\,d\sigma,& \e<\rho\leq e^{-1}.\end{array}\right.
\]
Then $G_{\delta,\e}$ is continuous in $[0,e^{-1}]$, $G_{\delta,\e}'(\rho)=\rho^{n-1}(-\ln\rho)^{\delta}$ for $0<\rho<\e$, and $G_{\delta,\e}'(\rho)=0$ for $\e<\rho<e^{-1}$. Moreover, $G_{\delta,\e}\geq 0$, and also
\begin{equation}\label{eq:BoundGe}
G_{\delta,\e}(\rho)\leq C_n\rho^n(-\ln\rho)^{\delta}\,\,\,\text{for}\,\,\,0<\rho<e^{-1}, \quad G_{\delta,\e}(\e)\geq C_n(-\ln\e)^{\delta}\e^n.
\end{equation}
We now set
\[
u_{\delta,\e}(r)=\int_r^{e^{-1}}\frac{G_{\delta,\e}(\rho)}{\rho^{n-1}(-\ln\rho)^{\delta}}\,d\rho.
\]
From \eqref{eq:BoundGe}, $u_{\delta,\e}$ is Lipschitz in $(0,e^{-1})$. So, if $u_{\delta,\e}(x)=u_{\delta,\e}(|x|)$, then $u_{\delta,\e}\in\Lip(B)\cap W_0^{1,2}(B)$ and
\[
-\Delta u_{\delta,\e}+\delta c\nabla u_{\delta,\e}=-u''_{\delta,\e}-\frac{n-1}{r}u'_{\delta,\e}-\frac{\delta}{r\ln r}u'_{\delta,\e}=\chi_{(0,\e)}(|x|),
\]
hence $u_{\delta,\e}$ is the $W_0^{1,2}(B)$ solution to the equation $-\Delta u_{\delta,\e}+\delta c\nabla u_{\delta,\e}=\chi_{B_{\e}}$ in $B$.

Let now $0<\e_1<\e_2<\e_3<e^{-1}$, and suppose that $g^{\delta}(y,x)\leq C$ for almost every $y\in B_{\e_1}$ and $x\in B_{\e_3}\setminus B_{\e_2}$, for some $C>0$. Then, for $\e<\e_1$ and almost every $x$ with $\e_2<|x|<\e_3$, we should have that
\begin{equation}\label{eq:BoundudeltaAbove}
u_{\delta,\e}(x)=\int_Bg^{\delta}(y,x)\chi_{B_{\e}}(y)\,dy\leq C\int_{B_{\e}}\,dy\leq C\e^n.
\end{equation}
Now, from the definition of $G_{\delta,\e}$, $G_{\delta,\e}(\rho)=G_{\delta,\e}(\e)$ for $\rho>\e$. So, if $\e_2<|x|<\e_3$, then
\begin{align}\nonumber
u_{\delta,\e}(x)&=\int_{|x|}^{e^{-1}}\frac{G_{\delta,\e}(\rho)}{\rho^{n-1}(-\ln\rho)^{\delta}}\,d\rho\geq\int_{\e_3}^{e^{-1}}\frac{G_{\delta,\e}(\rho)}{\rho^{n-1}(-\ln\rho)^{\delta}}\,d\rho\\
\label{eq:BoundudeltaBelow}
&=G_{\delta,\e}(\e)\int_{\e_3}^{e^{-1}}\frac{1}{\rho^{n-1}(-\ln\rho)^{\delta}}\,d\rho\geq C_n(-\ln\e)^{\delta}\e^n\int_{\e_3}^{e^{-1}}\frac{1}{\rho^{n-1}(-\ln\rho)^{\delta}}\,d\rho,
\end{align}
where we used \eqref{eq:BoundGe} for the last estimate. So, \eqref{eq:BoundudeltaAbove} and \eqref{eq:BoundudeltaBelow} show that
\[
C_n(-\ln\e)^{\delta}\e^n\int_{\e_3}^{e^{-1}}\frac{1}{\rho^{n-1}(-\ln\rho)^{\delta}}\,d\rho\leq C\e^n\Rightarrow (-\ln\e)^{\delta}\int_{\e_3}^{e^{-1}}\frac{1}{\rho^{n-1}(-\ln\rho)^{\delta}}\,d\rho\leq C,
\]
which is a contradiction by letting $\e\to 0$. This completes the proof.
\end{proof}

So, in the case $d\geq\dive c$, we see again that the space $L^{n,1}$ is necessary and optimal in order to deduce pointwise bounds for Green's function in Theorem~\ref{Green}: from the previous proposition, even assuming smallness of an $L^{n,q}$ norm for some $q>1$ is not enough to guarantee those bounds.

\begin{remark} In the case when $d,c$ are not identically $0$, it might seem that the assumption $d\geq\dive c$ allows us to consider lower regularity than $c\in L^{n,1}(\Omega)$ in order to deduce pointwise bounds in Theorem~\ref{Green}. However, this is not the case: if $c$ is as in \eqref{eq:CFormula}, then
\[
\dive c=\frac{1-(n-2)\ln r}{r^2\ln^2r}\in L^{\frac{n}{2},q}(B)\,\,\,\text{for all}\,\,\,q>1.
\]
So, setting $\tilde{c}=-c$, $\tilde{d}=\dive\tilde{c}$ and $b=0$, we obtain that $\tilde{d}\in L^{\frac{n}{2},q}$ for all $q>1$, $\tilde{d}\geq\dive\tilde{c}$, and then the operator $\mathcal{L}u=-\Delta u+\tilde{c}\nabla u+\tilde{d}u$ is equal to $\mathcal{L}u=-\Delta u-\dive(cu)$. Then, from Proposition~\ref{NoPointwiseGt}, Green's function for $\mathcal{L}$ in $u$ cannot satisfy pointwise bounds. So, the assumption $c\in L^{n,1}(\Omega)$ is necessary in this case as well.
\end{remark}

\section{Applications}
\subsection{Global boundedness}
As a first application of our results, we will show a maximum principle (with a constant) for subsolutions. In order to do this, we first show a pointwise estimate for solutions that vanish on the boundary to equations with right hand sides.

\begin{lemma}\label{BoundedSolutions}
Let $\Omega\subseteq\bR^n$ be a domain with $|\Omega|<\infty$. Let $A$ be uniformly elliptic and bounded in $\Omega$, and suppose that $b,c\in L^{n,q}(\Omega)$ for some $q\in[1,\infty)$, $d\in L^{\frac{n}{2},\infty}(\Omega)$, with $b-c\in L^{n,1}(\Omega)$, and either $d\geq\dive b$, or $d\geq\dive c$. If $f\in L^{n,1}(\Omega)$ and $g\in L^{\frac{n}{2},1}(\Omega)$, the solution $u\in W_0^{1,2}(\Omega)$ that solves the equation $-\dive(A\nabla u+bu)+c\nabla u+du=-\dive f+g$ in $\Omega$ is bounded, with
\[
\|u\|_{\infty}\leq C\|f\|_{n,1}+C\|g\|_{\frac{n}{2},1},
\]
where $C$ depends only on $n,\lambda$ and $\|b-c\|_{n,1}$.
\end{lemma}
\begin{proof}
Existence and uniqueness of such a solution is guaranteed from Lemmas 4.2 and 4.4 in \cite{KimSak}. Let now $(f_j),(g_j)$ be compactly supported Lipschitz continuous functions in $\Omega$, with $f_j\to f$ in $L^{n,1}(\Omega)$ and $g_j\to g$ in $L^{\frac{n}{2},1}(\Omega)$. In the case that $d\geq\dive c$, let $G$ be Green's function for the equation $-\dive(A^t\nabla u+cu)+b\nabla u+du=0$ in $\Omega$ from Theorem~\ref{Green}, and define
\begin{equation}\label{eq:umExists}
u_j(y)=\int_{\Omega}G_y(x)(-\dive f_j(x)+g_j(x))\,dx=\int_{\Omega}\nabla_xG_y(x)\cdot f_j(x)\,dx+\int_{\Omega}G_y(x)g_j(x)\,dx.
\end{equation}
From Theorem~\ref{Green}, $u_j$ is a $W_0^{1,2}(\Omega)$ solution to the equation
\begin{equation*}
-\dive(A\nabla u_j+bu_j)+c\nabla u_j+du_j=-\dive f_j+g_j
\end{equation*}
in $\Omega$. Moreover, since $n>2$ and $\frac{n}{2}>2_*$, we have that $f_j\to f$ in $L^2(\Omega)$ and $g_j\to g$ in $L^{2_*}(\Omega)$, therefore, from (4.5) in \cite{KimSak} we obtain that $u_j\to u$ in $W_0^{1,2}(\Omega)$. Hence, for a subsequence $(u_{j_m})$, $u_{j_m}\to u$ almost everywhere in $\Omega$, from the Rellich compactness theorem. Note now that, from Theorem~\ref{Green}, $\|G_y\|_{L^{\frac{n}{n-2},\infty}(\Omega)}+\|\nabla G_y\|_{L^{\frac{n}{n-1},\infty}(\Omega)}\leq C$, where $C$ depends only on $n,\lambda$ and $\|b-c\|_{n,1}$. Hence, using \eqref{eq:umExists}, we obtain that, for almost every $y\in\Omega$,
\begin{align*}
|u_{j_m}(y)|&\leq C_n\|\nabla G_y\|_{L^{\frac{n}{n-1},\infty}(\Omega)}\|f_{j_m}\|_{L^{n,1}(\Omega)}+C_n\|\nabla G_y\|_{L^{\frac{n}{n-2},\infty}(\Omega)}\|g_{j_m}\|_{L^{\frac{n}{2},1}(\Omega)}\\
&\leq C\|f_{j_m}\|_{L^{n,1}(\Omega)}+C\|g_{j_m}\|_{L^{\frac{n}{2},1}(\Omega)},
\end{align*}
and letting $m\to\infty$ completes the proof.
\end{proof}

Note the sharp contrast between Lemma~\ref{BoundedSolutions} above and Lemma 7.4 in \cite{KimSak}, in which a solution $u\in W_0^{1,2}(\Omega)$ to $-\Delta u-\dive(cu)=f$ in $B_{1/e}$ is constructed, for the $c$ in \eqref{eq:CFormula} and some $f\in L^{\infty}(B)$, but where $u$ is not bounded. From Lemma~\ref{bNorms}, $c\in L^{n,q}(\Omega)$ for every $q>1$, and this shows the necessity and the optimality of $L^{n,1}$ to obtain pointwise bounds for solutions as in Lemma~\ref{BoundedSolutions}.

\begin{remark}\label{Rem}
We remark that the assumption $|\Omega|<\infty$ in Lemma~\ref{BoundedSolutions} can be dropped, but in this case we will have to add the assumptions $f\in L^2(\Omega)$ and $g\in L^{2_*}(\Omega)$. Then, using the estimates in \cite{MourgoglouRegularity} to show that $u_m\to u$ in $Y_0^{1,2}(\Omega)$, the same result will hold. 
\end{remark}

We now show a proposition whose proof is inspired by Theorem 8.1 in \cite{Gilbarg}. Note that under slightly weaker hypotheses, this proposition has appeared in \cite{MourgoglouDraft} (see also Theorem 5.1 in \cite{MourgoglouRegularity} for a different proof).

\begin{prop}\label{BeforeMaxPrinciplet}
Let $\Omega\subseteq\bR^n$ be a domain. Assume that $A$ is uniformly elliptic and bounded, with ellipticity $\lambda$, and $b,c\in L^{n,\infty}(\Omega)$, $d\in L^{\frac{n}{2},\infty}(\Omega)$, with $b-c\in L^{n,q}(\Omega)$ for some $q<\infty$ and $d\geq\dive c$. If $u\in Y^{1,2}(\Omega)$ is a subsolution of the equation
\[
-\dive(A\nabla u+bu)+c\nabla u+du\leq 0
\]
in $\Omega$, with $u\leq 0$ on $\partial\Omega$, then $u\leq 0$ in $\Omega$.
\end{prop}
\begin{proof}
By contradiction, assume that the set $[u>0]$ has positive measure in $\Omega$. Moreover, from \eqref{eq:LorentzNormsRelations}, we can assume that $q>n$.

From the assumptions, $u^+\in Y_0^{1,2}(\Omega)$. Let now $\delta>0$, and define
\[
u_{\delta}(x)=\delta-(\delta-u^+)^+=\left\{\begin{array}{c l}0, & u(x)\leq 0 \\ u(x), & 0<u(x)\leq\delta \\ \delta, & u(x)>\delta\end{array}\right.
\]
then $u_{\delta}\in Y_0^{1,2}(\Omega)$ and $uu_{\delta}\geq 0$ in $\Omega$. As in Proposition~\ref{MaxPrinciple}, we use $u_{\delta} $ as a test function, and since $A\nabla u\nabla u_{\delta}=A\nabla u_{\delta}\nabla u_{\delta}$ and $d\geq\dive c$, combining with Lemma~\ref{ImprovedSobolev} and \eqref{eq:Div} we obtain that
\[
\lambda\|\nabla u_{\delta}\|_{L^2(\Omega)}^2\leq\int_{\Omega}A\nabla u_{\delta}\nabla u_{\delta}\leq\int_{\Omega}(c-b)\nabla u_{\delta}\cdot u.
\]
Set $D_{\delta}=[0<u\leq\delta]$ and note that $u=u_{\delta}$ in $D_{\delta}$ and $\nabla u_{\delta}=0$ almost everywhere in $\Omega\setminus D_{\delta}$. Therefore, if $\frac{p}{2}$ is the conjugate exponent to $\frac{q}{2}>\frac{n}{2}$, as in \eqref{eq:Forb-c} we estimate
\begin{align*}
\lambda\|\nabla u_{\delta}\|_{L^2(\Omega)}^2&\leq\int_{D_{\delta}}|c-b||\nabla u_{\delta}||u_{\delta}|\leq C_n\|b-c\|_{L^{n,q}(D_{\delta})}\|\nabla u_{\delta}\|_{L^2(\Omega)}\|u_{\delta}\|_{L^{2^*,p}(\Omega)}\\
&\leq C_{n,q}\|b-c\|_{L^{n,q}(D_{\delta})}\|\nabla u_{\delta}\|_{L^2(\Omega)}\|u_{\delta}\|_{L^{2^*,2}(\Omega)}\leq C_{n,q}\|b\|_{L^{n,q}(D_{\delta})}\|\nabla u_{\delta}\|_{L^2(\Omega)}^2,
\end{align*}
where $C$ depends on $n,q$.

If $\|\nabla u_{\delta}\|_2=0$ for some $\delta>0$, then $u_{\delta}$ is constant in $\Omega$. Since $u_{\delta}\in Y_0^{1,2}(\Omega)$, this implies that $u_{\delta}\equiv 0$, hence $u\leq 0$, which is a contradiction; hence $\|\nabla u_{\delta}\|_2\neq 0$ for all $\delta>0$. Therefore, the last estimate shows that, for every $\delta>0$,
\[
\|b-c\|_{L^{n,q}(D_{\delta})}\geq C_{n,q,\lambda}.
\]
On the other hand, $\chi_{D_{1/m}}\to 0$ everywhere in $\Omega$ as $m\to\infty$, therefore Lemma~\ref{NormOnDisjoint} shows that $\|b-c\|_{L^{n,q}(D_{1/m})}\to 0$ as $m\to\infty$, which is a contradiction. Hence $u\leq 0$ in $\Omega$, which completes the proof.
\end{proof}

\begin{remark}
The estimate in Proposition~\ref{BeforeMaxPrinciplet} is notable since an assumption of the form $u\leq s$ on $\partial\Omega$, for $s>0$, does not guarantee that $u$ is bounded in $\Omega$, even if $s$ is assumed to be small and $b-c$ has small $L^{n,q}(\Omega)$ norm. This is exhibited by the argument in the proof of Lemma 7.4 in \cite{KimSak}.
\end{remark}

We now show a maximum principle for subsolutions in the case $d\geq\dive c$.

\begin{prop}\label{MaxPrinciplet}
Under the same assumptions as in Lemma~\ref{BoundedSolutions}, assume that $u\in W^{1,2}(\Omega)$ is a subsolution to the equation $-\dive(A\nabla u+bu)+c\nabla u+du\leq -\dive f+g$. Then,
\[
\sup_{\Omega}u\leq C\left(\sup_{\partial\Omega}u^++\|f\|_{n,1}+\|g\|_{\frac{n}{2},1}\right),
\]
where $C$ depends on $n,\lambda$ and $\|b-c\|_{n,1}$ only.
\end{prop}
\begin{proof}
Let $v\in W_0^{1,2}(\Omega)$ be the solution to the equation $-\dive(A\nabla v+bv)+c\nabla v+dv=-\dive f+g$ in $\Omega$, from Lemma 4.4 in \cite{KimSak}. Then, Lemma~\ref{BoundedSolutions} shows that $\|v\|_{\infty}\leq C\|f\|_{n,1}+C\|g\|_{\frac{n}{2},1}$ for some $C$ that depends on $n,\lambda$ and $\|b-c\|_{n,1}$, and also $w=u-v$ is a subsolution to
\[
-\dive(A\nabla w+bw)+c\nabla w+dw\leq 0
\]
in $\Omega$. Since $0$ is a subsolution to the equation above, we have that $w^+=\max\{w,0\}\in W^{1,2}(\Omega)$ is a subsolution to the same equation in $\Omega$: in the case that the operator is coercive, this follows from Theorem 3.5 in \cite{StampacchiaDirichlet}; in the general case, we split the domain in finitely many subdomains in which the operator is coercive, and we use a partition of unity argument. Since $w^+\geq 0$, the assumption $d\geq\dive c$ shows that $w^+$ is a subsolution to
\[
-\dive(A\nabla w^++(b-c)w^+)\leq 0.
\]
Let now $l=\sup_{\partial\Omega}u^+$, and assume that $l<\infty$. We then let $v_0\in W_0^{1,2}(\Omega)$ be the solution to the equation $-\dive(A\nabla v_0+(b-c)v_0)=-\dive(l(b-c))$ in $\Omega$. Then, Lemma~\ref{BoundedSolutions} shows that $\|v_0\|_{\infty}\leq C\|l(b-c)\|_{n,1}$. Moreover, since $w_0=w^+-l+v_0\in W^{1,2}(\Omega)$, and $-\dive(A\nabla w_0+(b-c)w_0)\leq 0$ in $\Omega$ with $w_0\leq 0$ on $\partial\Omega$, Proposition~\ref{BeforeMaxPrinciplet} shows that $w_0\leq 0$ in $\Omega$. Hence,
\[
\sup_{\Omega}u=\sup_{\Omega}(v+w)\leq\|v\|_{\infty}+\sup_{\Omega}w^+\leq\|v\|_{\infty}+\sup_{\Omega}(w_0+l-v_0)\leq \|v\|_{\infty}+\|v_0\|_{\infty}+l,
\]
and combining with the pointwise bounds on $v$ and $v_0$ above completes the proof.
\end{proof}

\begin{remark}
As in Remark~\ref{Rem}, the assumption $|\Omega|<\infty$ in the previous Proposition can be dropped, after assuming also that $g\in L^{2_*}(\Omega)$ and $f\in L^2(\Omega)$.
\end{remark}

\subsection{Local Boundedness}
We now turn to a local boundedness estimate. We will follow the idea in the proof of Proposition~\ref{SupByIntegral}; that is, we will first show the estimate in the case that the $L^{n,1}$ norm of $b-c$ is small, and using the maximum principle in Proposition~\ref{MaxPrinciplet} we will pass to general norms. The first step is the following.

\begin{lemma}\label{BeforeLocalBoundt}
Let $B_r\subseteq \bR^n$ be a ball of radius $r$. Let $A$ be uniformly elliptic and bounded in $B_r$ with ellipticity $\lambda$. There exists $\e_0>0$, depending only on $n$ and $\lambda$ such that, if $b\in L^{n,1}(B_r)$ with $\|b\|_{n,1}\leq\e_0$, then for every $u\in W^{1,2}(B_r)$ that is a nonnegative subsolution to $-\dive(A\nabla u+bu)\leq 0$ in $B_r$, we have that
\[
\sup_{B_{r/2}}u\leq C_0\fint_{B_r}u,
\]
where $C_0$ depends on $n,\lambda$ and $\|A\|_{\infty}$ only.
\end{lemma}
\begin{proof}
Since the estimate is scale invariant, we will assume that $r=1$.

Note first that there exists $\e_0>0$ depending only on $n$ and $\lambda$ such that, if $\|b\|_{n,1}\leq\e_0$, then
\begin{equation}\label{eq:Coercive}
\int_{B_1}A\nabla w\nabla w+b\nabla w\cdot w\geq\frac{\lambda}{3}\int_{B_1}|\nabla w|^2,
\end{equation}
for every $w\in W_0^{1,2}(\Omega)$. Then, an inspection of the proof of Lemma~\ref{DerivativeByU}, together with the Sobolev inequality and \eqref{eq:LorentzNormsRelations} show that, if $u_0$ is a nonnegative subsolution to $-\dive(A\nabla u_0+bu_0)\leq -\dive f_0$ in $B_1$ for some $f_0\in L^2(B_1)$, then for every $\phi\in C_c^{\infty}(B_1)$,
\begin{equation}\label{eq:PlugPhi}
\int_{B_1}|\phi\nabla u_0|^2\leq C\int_{B_1}|u_0\nabla\phi|^2+\int_{B_1}|f_0\nabla\phi|^2,\quad\left(\int_{B_1}|\phi u_0|^{2^*}\right)^{2/2^*}\leq C\int_{B_1}|u_0\nabla\phi|^2+\int_{B_1}|f_0\nabla\phi|^2,
\end{equation}
where $C$ depends only on $n,\lambda$ and $\|A\|_{\infty}$, and where the second estimate follows from the first one by adding the term $\int_{B_1}|u_0\nabla\phi|^2$ to both sides of the first estimate, and using Sobolev's inequality. For the rest of this proof, we will assume that $\|b\|_{L^{n,1}(B_1)}\leq\e_0$.

We first apply \eqref{eq:PlugPhi} for $u_0=u$, $f_0=0$ and $\phi=\phi_1$ being a smooth cutoff function supported in $B_1$, with $\phi_1\equiv 1$ in $B_{7/8}$ and $0\leq\phi_1\leq 1$. Then, we obtain that
\begin{equation}\label{eq:Deru}
\left(\int_{B_{7/8}}|u|^{2^*}\right)^{2/2^*}+\int_{B_{7/8}}|\nabla u|^2\leq C\int_{B_1}|u|^2.
\end{equation}
Set now $F=-\dive(A\nabla u+bu)$, then $F\in W^{-1,2}(B_{7/8})$ (the dual of $W_0^{1,2}(B_{7/8})$). Therefore, from Lemma 4.4 in \cite{KimSak} (for $c,d$ being identically equal to $0$), there exists $v\in W_0^{1,2}(B_{7/8})$ such that $-\dive(A\nabla v+bv)=F=-\dive(A\nabla u+bu)$. Then, using $v$ in \eqref{eq:Coercive} and using H{\"o}lder's inequality and Sobolev's inequality, we estimate
\begin{align*}
\frac{\lambda}{3}\int_{B_{7/8}}|\nabla v|^2&\leq \int_{B_{7/8}}A\nabla v\nabla v+b\nabla v\cdot v=\int_{B_{7/8}}A\nabla u\nabla v+b\nabla v\cdot u\\
&\leq C\|\nabla v\|_{L^2(B_{7/8})}\left(\|\nabla u\|_{L^2(B_{7/8})}+\|bu\|_{L^2(B_{7/8})}\right)\leq C\|\nabla v\|_{L^2(B_{7/8})}\|u\|_{L^2(B_1)},
\end{align*}
where $C$ depends on $n,\lambda$ and $\|A\|_{\infty}$, and where we used \eqref{eq:Deru} in the last step. Combining with Sobolev's inequality, we obtain that
\begin{equation}\label{eq:vbyu}
\left(\int_{B_{7/8}}|v|^{2^*}\right)^{2/2^*}+\int_{B_{7/8}}|\nabla v|^2\leq C\left(\int_{B_1}|u|^2\right)^{1/2}.
\end{equation}
Since $-\dive(A\nabla u+bu)\leq 0$, we obtain that $v\in W_0^{1,2}(B_{7/8})$ is a subsolution to $-\dive(A\nabla v+bv)\leq 0$, hence Proposition~\ref{BeforeMaxPrinciplet} shows that $v\leq 0$ in $B_{7/8}$. Therefore, setting $w=u-v$, we obtain that $w\in W^{1,2}(B_{7/8})$,
\[
-\dive(A\nabla w+bw)=0\,\,\,\text{in}\,\,\,B_{7/8},\,\,\,\text{and}\,\,\,u=v+w\leq w.
\]

Suppose that $\phi_2$ is a smooth cutoff function, with $\phi_2\equiv 1$ in $B_{3/4}$, $\phi_2$ supported in $B_{7/8}$ and $0\leq\phi_2\leq 1$. Then, using $\phi_2$ in \eqref{eq:PlugPhi}, for $u_0=w$ and $f_0=0$, we obtain that
\begin{equation}\label{eq:Nabw}
\left(\int_{B_{3/4}}|w|^{2^*}\right)^{2/2^*}+\int_{B_{3/4}}|\nabla w|^2\leq C\int_{B_{7/8}}|w|^2.
\end{equation}
Let now $\phi_3$ be a smooth cutoff function, with $\phi_3\equiv 1$ in $B_{5/8}$, $\phi_3$ supported in $B_{3/4}$ and $0\leq\phi_2\leq 1$, and set $w_0=w\phi_3\in W_0^{1,2}(B_1)$. Then, we have that $w_0$ solves the equation
\[
-\dive(A\nabla w_0+bw_0)=-\dive(A\nabla\phi_3\cdot w)-A\nabla w\nabla\phi_3-b\nabla\phi_3\cdot w
\]
in $B_1$. Define $f=A\nabla\phi_3\cdot w$ and $g=A\nabla w\nabla\phi_3+b\nabla\phi_3\cdot w\in L^2(B_2)$. Then, we estimate
\begin{equation}\label{eq:fSquared}
\|f\|_{L^2(B_1)}^2\leq C\int_{B_{3/4}}|w|^2\leq C\int_{B_1}|u|^2,
\end{equation}
where we used that $w=u-v$ and \eqref{eq:vbyu} in the last estimate. Moreover, using \eqref{eq:Nabw}, H{\"o}lder's inequality and \eqref{eq:vbyu}, we obtain that
\begin{equation}\label{eq:gSquared}
\|g\|_{L^2(B_1)}^2\leq C\int_{B_{3/4}}|\nabla w|^2+C\int_{B_{3/4}}|b-c|^2|w|^2\leq C\int_{B_{7/8}}|w|^2\leq C\int_{B_1}|u|^2.
\end{equation}
Note now that $f,g$ vanish in $B_{5/8}$, therefore there exist two sequences $(f_j),(g_j)$ of bounded functions in $B_1$, vanishing in $B_{9/16}$, with $f_j\to f$ and $g_j\to g$ in $L^2(B_1)$. We then consider Green's function $G(x,y)=G_y(x)$ for the operator $-\dive(A^t\nabla u)+b\nabla u$ in $B_1$ from Theorem~\ref{Green}, and set
\[
w_j(y)=\int_{B_1}\nabla G_y\cdot f_j-\int_{B_1}G_yg_j.
\]
From Theorem~\ref{Green}, $w_j\in W_0^{1,2}(B_1)$ and solves the equation $-\dive(A\nabla w_j+bw_j)=-\dive f_j-g_j$; hence $w_j-w_0\in W_0^{1,2}(B_1)$, and solves the equation
\[
-\dive(A\nabla(w_j-w_0)+b(w_j-w_0))=-\dive(f_j-f)-(g_j-g)
\]
in $B_1$. Moreover, since $f_j\to f$ and $g_j\to g$ in $L^2(B_1)$, Lemma 4.2 in \cite{KimSak} shows that $w_j\to w_0$ in $W_0^{1,2}(B_1)$. Hence, for a subsequence $(w_{j_i})$, $w_{j_i}\to w_0$ almost everywhere in $B_1$. Note also that, for every $y\in B_{1/2}$, $G_y\in W^{1,2}(B_1\setminus B_{9/16})$ from Theorem~\ref{Green}. Hence, using the formula of $w_j$ and the support properties of $f_j,g_j$ we obtain that, for almost every $y\in B_{1/2}$, 
\begin{align*}
|w_{j_i}(y)|&\leq \int_{B_1\setminus B_{9/16}}|\nabla G_y\cdot f_{j_i}|+\int_{B_1\setminus B_{9/16}}|G_yg_{j_i}|\\
&\leq C\|\nabla G_y\|_{L^2(B_1\setminus B_{1/16}(y))}\|f_{j_i}\|_{L^2(B_1)}+C\|G_y\|_{L^2(B_1\setminus B_{1/16}(y))}\|g_{j_i}\|_{L^2(B_1)},
\end{align*}
where we used that $B_{1/16}(y)\subseteq B_{9/16}$ for any $y\in B_{1/2}$. Letting $i\to\infty$, using the bounds in Theorem~\ref{Green}, \eqref{eq:fSquared} and \eqref{eq:gSquared} and also $u\leq w=w_0$ in $B_{1/2}$, we obtain that
\[
\sup_{B_{1/2}}u\leq C\left(\fint_{B_1}|u|^2\right)^{1/2},
\]
where $C$ depends on $n,\lambda$ and $\|A\|_{\infty}$ only.

We now let $0<t<s<1$ and set $\rho=\frac{s-t}{2}$. Then, for each $x$ with $|x|\leq t$, we use a scaling argument to apply the previous estimate in $B_{\rho}(x)$, and we obtain that
\[
\sup_{B_{\rho/2}(x)}u\leq\frac{C}{\rho^{n/2}}\left(\int_{B_{\rho}(x)}|u|^2\right)^{1/2}\leq\frac{C}{(s-t)^{n/2}}\left(\int_{B_s}|u|^2\right)^{1/2},
\]
since $B_{\rho}(x)\subseteq B_s$. This implies that
\[
\sup_{B_t}u\leq\frac{C}{(s-t)^{n/2}}\left(\int_{B_s}|u|^2\right)^{1/2}\leq\frac{C}{(s-t)^{n/2}}\left(\int_{B_s}|u|\cdot\sup_{B_s}u\right)^{1/2}\leq\frac{C^2}{2(s-t)^n}\int_{B_s}|u|+\frac{1}{2}\sup_{B_s}u
\] 
for all $0<t<s<1$, where $C$ only depends on $n,\lambda$ and $\|A\|_{\infty}$. Hence, using Lemma 5.1 on page 81 in \cite{Giaquinta} completes the proof.
\end{proof}

We now drop the smallness assumption in Lemma~\ref{BeforeLocalBoundt} and we show the following proposition.
\begin{prop}\label{LocalBoundt}
Let $B_r\subseteq \bR^n$ be a ball of radius $r$. Let also $A$ be uniformly elliptic and bounded in $B_r$ with ellipticity $\lambda$, and $b,c\in L^{n,q}(B_r)$ for some $q<\infty$, $d\in L^{\frac{n}{2},\infty}(B_r)$, with $b-c\in L^{n,1}(B_r)$ and $d\geq\dive c$, and $f\in L^{n,1}(B_r),g\in L^{\frac{n}{2},1}(B_r)$. Then, for every solution or nonnegative subsolution $u\in W^{1,2}(B_r)$ of $\mathcal{L}u=-\dive(A\nabla u+bu)+c\nabla u+du\leq -\dive f+g$ in $B_r$, we have that
\[
\sup_{B_{r/2}}|u|\leq C\left(\fint_{B_r}|u|+C\|f\|_{L^{n,1}(B_r)}+C\|g\|_{L^{\frac{n}{2},1}(B_r)}\right),
\]
where $C$ depends on $n,\lambda,\|A\|_{\infty}$ and $\|b-c\|_{n,1}$ only.
\end{prop}
\begin{proof}
First, we subtract from $u$ the solution $v\in W_0^{1,2}(B_r)$ with $\mathcal{L}v=-\dive f+g$ constructed in Lemma~\ref{BoundedSolutions}. If $u$ solves $\mathcal{L}u=-\dive f+g$, then $\mathcal{L}(u-v)=0$, and if $\mathcal{L}u\leq 0$ is a nonnegative subsolution, then $\mathcal{L}(u-v)\leq 0$, so as in the proof of Proposition~\ref{MaxPrinciplet}, $\mathcal{L}\left((u-v)^+\right)\leq 0$ and $(u-v)^+\geq 0$. Hence, using the estimate in Proposition~\ref{MaxPrinciplet}, we can assume that $f$ and $g$ are identically $0$ in both cases.
	
Since the estimate is scale invariant, we will assume that $r=1$. Then, if $u$ is a solution to $\mathcal{L}u=0$, then as in the proof of Proposition~\ref{MaxPrinciplet}, $u^+,u^-$ are subsolutions to $\mathcal{L}u\leq 0$, hence $|u|=u^++u^-$ is a subsolution. Therefore it suffices to show the proposition for nonnegative subsolutions $u$. Note then that $-\dive(A\nabla u+(b-c)u)\leq 0$ in $B_1$, so we can assume in the above that $c,d$ are identically equal to $0$ in $B_1$ and $b\in L^{n,1}(B_1)$. Suppose also that $\lambda$ and $\|A\|_{\infty}$ are fixed.

We follow the idea of the proof of Proposition~\ref{SupByIntegral}. Let $\e_0>0$ be the number in Lemma~\ref{BeforeLocalBoundt}. From Proposition~\ref{MaxPrinciplet}, for every $m\in\mathbb N$ there exists $c_m\geq 1$ depending on $n,\lambda$ and $m$ such that, if $\Omega$ has finite measure, $\|b\|_{L^{n,1}(\Omega)}\leq \e_0\sqrt[n]{m}$ and $u\in W^{1,2}(\Omega)$ is a nonnegative subsolution of $-\dive(A\nabla u+bu)\leq 0$ in $\Omega$, then
\[
\sup_{\Omega}u\leq c_m\sup_{\partial\Omega}u.
\]
We will now inductively show that, if $C_0$ is the constant in Lemma~\ref{BeforeLocalBoundt} and $u\in W^{1,2}(B_1)$ is a nonnegative subsolution to $-\dive(A\nabla u+bu)\leq 0$ in $B_1$, then
\begin{equation}\label{eq:InductiveHypothesist}
\sup_{B_{1/2}}u\leq 8^{(m-1)n}C_0\prod_{j=1}^mc_j\fint_{B_1}u,\quad\text{if}\quad\|b\|_{L^{n,1}(B_1)}^n\leq m\e_0^n.
\end{equation}
First, when $m=1$ the estimate holds, from Proposition~\ref{BeforeMaxPrinciplet}. Let now $m\geq 1$, and suppose that \eqref{eq:InductiveHypothesist} holds for $m$. Suppose now that $b\in L^{n,1}(B_1)$ is such that $m\e_0^n<\|b\|_{L^{n,1}(B_1)}^n\leq(m+1)\e_0^n$. We distinguish between two cases: $\|b\|_{L^{n,1}(B_{3/4})}^n\leq m\e_0^n$, and $\|b\|_{L^{n,1}(B_{3/4})}^n>m\e_0^n$.

In the first case, for any $x$ with $|x|<\frac{1}{2}$, $B_{1/4}(x)\subseteq B_{3/4}$, therefore  $\|b\|_{L^{n,1}(B_{1/4}(x))}^n\leq m\e_0^n$. Then, from the inductive hypothesis \eqref{eq:InductiveHypothesist} and a scaling argument,
\[
\sup_{B_{1/8}(x)}u\leq 8^{(m-1)n}C_0\prod_{j=1}^mc_j\fint_{B_{1/4}(x)}u\leq 8^{mn}C_0\prod_{j=1}^mc_j\fint_{B_1}u.
\]
Since this estimate holds for any $x$ with $|x|<\frac{1}{2}$, we obtain that
\begin{equation}\label{eq:Case1t}
\sup_{B_{1/2}}u\leq\sup_{|x|\leq 1/2}\left(\sup_{B_{1/8}(x)}u\right)\leq 8^{mn}C_0\prod_{j=1}^mc_j\fint_{B_1}u\leq 8^{mn}C_0\prod_{j=1}^{m+1}c_j\fint_{B_1}u.
\end{equation}
In the second case, we have that $\|b\|_{L^n(B_{3/4})}^n>m\e_0^n$, therefore Lemma~\ref{NormOnDisjoint0} shows that
\[
\|b\|_{L^{n,1}(B_1\setminus B_{3/4})}^n\leq \|b\|_{L^{n,1}(B_1)}^n-\|b\|_{L^{n,1}(B_{3/4})}^n\leq(m+1)\e_0^n-m\e_0^n=\e_0^n.
\]
Now, for any $y$ with $|y|=\frac{7}{8}$, we have that $B_{1/8}(y)\subseteq B_1\setminus B_{3/4}$, therefore $\|b\|_{L^{n,1}(B_{1/8}(y))}^n\leq\e_0^n$. So, from Lemma~\ref{BeforeLocalBoundt}, we have that
\[
\sup_{B_{1/16}(y)}u\leq C_0\fint_{B_{1/8}(y)}u\leq 8^nC_0\fint_{B_1}u.
\]
Hence, using the definition of $c_{m+1}$, we obtain that
\begin{equation}\label{eq:Case2t}
\sup_{B_{1/2}}u\leq\sup_{B_{7/8}}u\leq c_{m+1}\sup_{\partial B_{7/8}}u\leq 8^nC_0c_{m+1}\fint_{B_1}u\leq 8^{mn}C_0\prod_{j=1}^{m+1}c_j\fint_{B_1}u.
\end{equation}
Hence, in all cases, \eqref{eq:Case1t} and \eqref{eq:Case2t} show that, if $m\e_0^n<\|b\|_{n,1}^n\leq (m+1)\e_0^n$, then
\begin{equation}\label{eq:InductionEndt}
\sup_{B_{1/2}}u\leq 8^{mn}C_0\prod_{j=1}^{m+1}c_j\fint_{B_1}u.
\end{equation}
If now $\|b\|_{n,1}^n\leq m\e_0^n$, \eqref{eq:InductiveHypothesist} for $m$ shows that \eqref{eq:InductionEndt} holds in this case as well; therefore, \eqref{eq:InductionEndt} holds whenever $\|b\|_{n,1}^n\leq (m+1)\e_0^n$. Inductively, this shows that \eqref{eq:InductiveHypothesist} holds for every $m\in\mathbb N$.

Finally, if $b\in L^{n,1}(B_1)$, choosing $m\in\mathbb N$ such that $(m-1)\e_0^n\leq\|b\|_{n,1}^n\leq m\e_0^n$ and applying \eqref{eq:InductiveHypothesist} for this $m$ completes the proof.
\end{proof}

We remark that passing through the smallness assumption could be avoided by using a Cacciopoli estimate of the type that appears in \cite{MourgoglouRegularity}. However, we wanted to exhibit a similar application of the idea in Proposition~\ref{SupByIntegral}.

\bibliographystyle{amsalpha}
\bibliography{Bibliography}

\end{document}